\documentclass[11pt]{article}
\usepackage[utf8]{inputenc}

\usepackage{verbatim}
\usepackage{caption}
\usepackage{subcaption}
\usepackage{amssymb}
\usepackage{amsmath}
\usepackage{tabularx}
\usepackage{color}
\usepackage{stmaryrd}
\usepackage{multirow}
\usepackage{mathtools}
\usepackage{hyperref}
\usepackage{epsfig,latexsym}

\usepackage{amsmath,amssymb,amsthm,graphicx,verbatim,bm,fancyhdr,bbm,enumerate,color, tabularx}
\usepackage{hyperref}

\DeclareMathAlphabet{\mathpzc}{OT1}{pzc}{m}{it}  

\theoremstyle{definition}
\theoremstyle{plain}
\newtheorem{theorem}{Theorem}

\newtheorem{lemma}{Lemma}
\newtheorem{lemmalemma}{Lemma}[lemma]

\theoremstyle{remark}

\newtheorem*{remark*}{Remark}

\newtheorem*{terminology*}{Terminology}
\newtheorem*{notation*}{Notation}

\newenvironment{pfofthm}[1]
{\par\vskip2\parsep\noindent{\em Proof of\ #1. }}{{\hfill
$\Box$}
\par\vskip2\parsep}

\renewcommand{\P}{\mathbb{P}}

\newcommand{\E}{\mathbb{E}}

\newcommand{\R}{\mathbb{R}}

\newcommand{\Z}{\mathbb{Z}}

\usepackage{lineno}

\let\oldalign\align
\let\oldendalign\endalign

\renewenvironment{align}
  {\linenomathNonumbers\oldalign}
  {\oldendalign\endlinenomath}

\title{Spread of premalignant mutant clones and \\ cancer initiation in multilayered tissue}
\author{Jasmine Foo$^{1}$ \and\hspace*{-6pt} Einar Bjarki Gunnarsson$^{2}$ \and\hspace*{-6pt}  Kevin Leder$^{2}$ \and\hspace*{-6pt} Kathleen Storey$^{3}$}
\date{%
    \footnotesize $^1$School of Mathematics, University of Minnesota, Twin Cities \\[3pt]
    $^2$Department of Industrial and Systems Engineering, University of Minnesota, Twin Cities \\[0pt]
    $^3$Department of Mathematics, Lafayette College
}
\begin{document}

\begin{center}
{\bf\large Accepted author manuscript (Annals of Applied Probability)}
\end{center}

\vspace*{-12pt}

\begingroup
\let\newpage\relax
\maketitle
\endgroup
      
\begin{abstract}
Over 80\% of human cancers originate from the epithelium, which covers the outer and inner surfaces of organs and blood vessels.
In stratified epithelium, the bottom layers are occupied by stem and stem-like cells that continually divide and replenish the upper layers. 
In this work, we study the spread of premalignant mutant clones and cancer initiation in stratified epithelium, using the biased voter model on stacked two-dimensional lattices. 
Our main result is an estimate of the propagation speed of a premalignant mutant clone, which is asymptotically precise in the cancer-relevant weak-selection limit.
We use our main result to study cancer initiation under a two-step mutational model of cancer, which includes computing the distributions of the time of cancer initiation and the size of the premalignant clone giving rise to cancer.
Our work quantifies the effect of epithelial tissue thickness on the process of carcinogenesis, thereby contributing to an emerging understanding of the spatial evolutionary dynamics of cancer. \vspace*{3pt}
\end{abstract}
 
 
 
 
\section{Introduction}\label{sec_intro}

According to the widely held multi-stage model of carcinogenesis, cancer arises due to the accumulation of genetic mutations that culminate in malignant cells
able to proliferate uncontrollably
\cite{armitage1954age,armitage1957two,knudson1971mutation,knudson2001}.
Each mutation in this process
can afford a small selective advantage,
which can allow premalignant cells to expand into clones or ``fields'' that are further along the evolutionary pathway to cancer than normal cells and thus predisposed to becoming cancerous \cite{Curtius2018}.
The notion that cancer arises on the background of premalignant field expansion
is referred to as ``field cancerization'' or ``the cancer field effect''.
It has important clinical implications, since tumors surrounded by premalignant patches are at increased risk of recurrence following cancer treatment \cite{braak2003, chai2009}.
Premalignant fields often appear histologically normal, making them difficult to distinguish from healthy tissue.
This suggests that a mathematical understanding of the spatial evolutionary dynamics of cancer initiation can yield valuable insights into treatment decision-making, including optimal surgical excision margins and post-treatment surveillance protocols.

\begin{figure}
    \centering
    \includegraphics[scale=0.95]{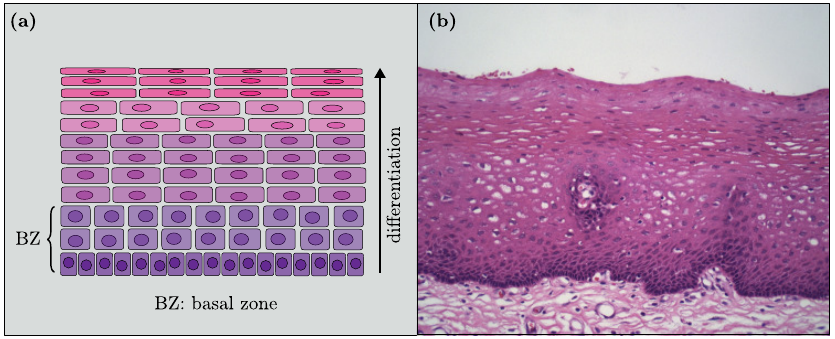}
    \caption{{\bf (a)} In stratified squamous epithelium of the esophagus, a basal layer of stem cells and 2-3 layers of proliferative basaloid cells form the basal zone, which continually replenishes the upper layers with differentiated cells.
    {\bf (b)} Histology of normal stratified squamous epithelium of the esophagus. 
    Figure 4.6 on page 78 of \cite{CHANDRASOMA201873}.
    }
    \label{fig:epithelium}
\end{figure}

Over 80\% of human cancers originate from the epithelium, which lines the outer and inner surfaces of organs and blood vessels \cite{NIH}. Simple epithelium consists of a single layer of proliferating cells, whereas in stratified epithelium, stem and stem-like cells proliferate along the bottom layers and continually replenish the upper layers with differentiated cells that lose their ability to proliferate.
For example, in {\em stratified squamous epithelium} of the esophagus, a basal layer of stem cells and 2-3 layers of proliferative basaloid cells form the {\em basal zone}, which accounts for less than 30\% of total epithelial thickness. As cells move upward, they become terminally differentiated keratinocytes with small nuclei 
that flatten out and eventually get shed at the top layer (Fig.~\ref{fig:epithelium}) \cite{CHANDRASOMA201873,Geboes94}.\footnote{Figure 1b is reprinted from GERD, Parakrama T.~Chandrasoma, Chapter 4 -- Histologic Definition and Diagnosis of Epithelia in the Esophagus and Proximal Stomach, pp.~73-107, Copyright 2018, with permission from Elsevier.}
Since the accumulation and spread of mutations is driven by the proliferating basal and basaloid cells, the basal zone is the appropriate setting to study the process of carcinogenesis in stratified epithelium.

In this work, we study the spread of premalignant mutant fields in epithelial basal zones, and we examine the effect of basal zone geometry on the process of cancer initiation.
Our main result determines the propagation speed of a premalignant mutant clone as a function of a small mutant selective advantage and the number of layers in the basal zone, which enables comparison of the evolutionary dynamics between different types of epithelial cancers.
We employ a spatially explicit model of cell division and replacement, where cells live on a set of stacked two-dimensional integer lattices, representing a multilayered basal zone.
The model dynamics are as follows: Cells of two types, normal and mutant, are arranged on the stacked lattices, with mutant cells dividing more frequently than normal cells. Upon cell division, one daughter cell stays put, and the other replaces a neighboring cell chosen uniformly at random.
This model was originally proposed by Williams \& Bjerknes \cite{Williams72} in the context of a single two-dimensional epithelial basal layer, and it arose independently within the field of interacting particle systems as the {\em biased voter model}.
Bramson and Griffeath \cite{BraGri81, BraGri80} showed in 1980-1981 that under the biased voter model, an advantageous mutant clone eventually assumes a convex, symmetric shape whose diameter grows linearly in time.
The Bramson-Griffeath shape theorem extends naturally to our stacked-lattice setting, and it plays a central role in the derivation of our main result.

Once we have determined the propagation speed of premalignant mutant clones in epithelial basal zones, we consider the implications of our result for the dynamics of cancer initiation.
The process of carcinogenesis under the multi-step model of cancer
has already been well-studied in the non-spatial, homogeneously mixed setting, see e.g.~the books by Nowak \cite{nowak2006} and Wodarz and Komarova \cite{wodarz2014}. In the spatial setting, Komarova \cite{kom2006} has analyzed the time of cancer initiation on a one-dimensional lattice under a two-step model of cancer, assuming a neutral or deleterious first-step mutation. 
Durrett and Moseley \cite{DurMose2015} extended Komarova's work to two and three dimensions assuming a neutral first step.
Durrett, Foo and Leder \cite{DurFooLed} considered the case of a small selective advantage (weak selection), and they derived the distribution of the time of cancer initiation under a two-step model of cancer in certain parameter regimes. Foo, Leder and Schweinsberg obtained more complete results in \cite{foo2020mutation}, where they also studied cancer initiation under a general $k$-step model.
In \cite{FLR2014}, Foo, Leder and Ryser studied field cancerization under a two-step model, which included computing size-distributions of premalignant fields at the time of cancer initiation. 
Upon establishing our main result of premalignant mutant propagation, we will adapt analysis from \cite{DurFooLed} and \cite{FLR2014} to gain insights into how cancer initiation and field cancerization is affected by the specific geometric setting of a multilayered basal zone.

The rest of the paper is organized as follows. In Section \ref{sec_mainstatement}, we propose a model of the spread of premalignant mutant fields along epithelial basal zones and state our main result of their long-run propagation speed.
In Section \ref{sec_proofoutline}, we present an outline of the proof of the main result, in which we exploit a duality between the biased voter model and a system of branching coalescing random walks. 
We use coupling to set up an approximation scheme, based on a pruning procedure of Durrett and Z{\"a}hle \cite{durrett2007}, which culminates in simple, coalescence-free, branching random walks that are more readily analyzed. We state ten technical lemmas, the most important of which are Lemmas \ref{lemma:upperbound} and \ref{lemma:lowerbound} that provide lower and upper bounds for the propagation speed of a premalignant mutant clone. In Section \ref{sec_mainproof}, we show how our main result follows from these two lemmas, and in Section \ref{sec:applications}, we discuss the implications of our main result for cancer initiation and field cancerization in multilayered epithelial basal zones. In Section \ref{sec:proofsoflemmas}, we present proofs of the ten lemmas, and we discuss how the Bramson-Griffeath shape theorem in two dimensions can be extended to our stacked-lattice setting.

\begin{notation*}
In our exposition, we make use of the following asymptotic notation, where $a$ is taken to be 0 or $\infty$ depending on the context:
\begin{itemize}
    \item[] $f(x) = o(g(x))$ and $g(x) = \omega(f(x))$ as $x \to a$ if $f(x)/g(x) \to 0$ as $x \to a$.
    \item[] $f(x) = O(g(x))$ and $g(x) = \Omega(f(x))$ as $x \to a$ if $\limsup_{x \to a} |f(x)/g(x)| < \infty$.
    \item[] $f(x) = \Theta(g(x))$ as $x \to a$ if $f(x) = O(g(x))$ and $f(x) = \Omega(g(x))$ as $x \to a$.
    \item[] $f(x) \sim g(x)$ as $x \to a$ if $f(x)/g(x) \to 1$ as $x \to a$.
\end{itemize}
\end{notation*}

\section{Model of spread of premalignant mutant fields and statement of main result}\label{sec_mainstatement}

Let $\Z_w := \Z \!\mod w$ denote the additive group of integers modulo $w \geq 1$.
We represent an epithelial basal zone as the set $\Z^2\times \Z_w$ of $w$ layers of two-dimensional integer lattices, with a periodic boundary condition along the third dimension.
For each site $x \in \Z^2  \times \Z_w$, we define its {\em neighborhood} as ${\cal N}(x) := \{x \pm e_i: i=1,2\}$ for $w=1$ and 
${\cal N}(x) := \{x \pm e_i: i=1,2,3\}$ for $w > 1$, where $e_i$ is the $i$-th unit vector, and addition along the third coordinate is carried out modulo $w$.
Note that for $w=1$, each site has four neighbors, for $w=2$, each site has five neighbors, and for $w>2$, each site has six neighbors.

To model the spread of premalignant mutant fields, we define the biased voter model on $\Z^2 \times \Z_w$ as follows: Each site in $\Z^2 \times \Z_w$  is occupied by either a type-0 cell, representing a normal cell, or a type-1 cell, representing a premalignant mutant cell. Type-1 cells have a fitness advantage $\beta > 0$ over type-0 cells, meaning that type-1 cells divide at exponential rate $1+\beta$, while type-0 cells divide at exponential rate 1.
Upon cell division at a site $x \in \Z^2\times \Z_w$, one daughter cell stays put, while the other daughter cell replaces a neighboring cell at a site $y \in {\cal N}(x)$ chosen uniformly at random (Fig.~\ref{fig:modeldynamics}a). 

We assume throughout that the fitness advantage $\beta$ is small. In \cite{bozic2010accumulation}, for example, 
Bozic et al.~show that data on multiple cancer types (glioblastoma, pancreatic cancer and colon cancer) is consistent with an average selective advantage of $\beta = 0.004$ per mutational step. They further argue that this estimate should be more broadly relevant across cancer types,
given the considerable overlap of pathways through which the selective mutations act. 

A few comments are in order on the biological significance of our modeling assumptions.
First, we allow cells on the top layer of $\Z^2\times \Z_w$ to replace cells on the bottom layer, and vice versa (Fig.~\ref{fig:modeldynamics}a). This assumption simplifies the analysis, as it means that the top and bottom layers have the same neighborhood structure as the intermediate layers,
but it is not biologically realistic. In Appendix \ref{app:boundarycondition}, we use simulation to show that this assumption does not significantly affect type-1 propagation when $\beta$ is small.
Secondly, the model dynamics are driven by cell division, with cell division preceding cell death, and the model assumes that type-0 and type-1 cells are equally likely to be replaced by a dividing cell.
When $\beta$ is small, we expect that the exact dynamics of cell division and cell death are not important for long-run type-1 propagation, and we plan to discuss this in future work.
Finally, we assume that cells are arranged on the lattice $\Z^2 \times \Z_w$ throughout, whereas in real tissue, the spatial structure may be different, and the structure may change
as the premalignancy progresses. 
As with any model of biological or physical phenomena, our model is simplified and not intended to capture the full complexity of the system.
Our focus here is on two important parameters of the process:
The fitness advantage $\beta> 0$ of mutant cells and the tissue thickness $w \geq 1$.

\begin{figure}
    \centering
    \includegraphics[scale=1]{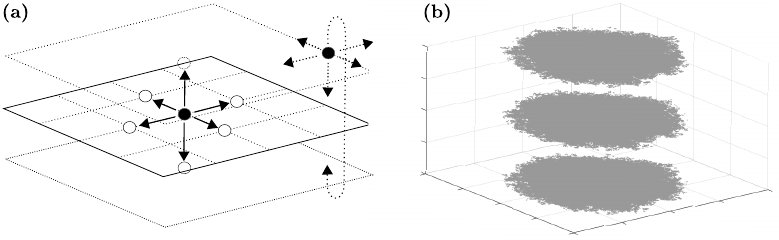}
    \caption{{\bf (a)} Model dynamics for the $w=3$ case (basal zone consists of three layers).
    When the black cell divides, one daughter cell stays put, and the other replaces one of six neighboring cells chosen uniformly at random.
    A cell on the top (resp.~bottom) layer can replace a cell on the same layer, the layer immediately below (resp.~above) or on the bottom (resp.~top) layer. {\bf (b)} Simulation of the model for $w=3$ and $\beta=0.1$, stopped at 50,000 cells. By the Bramson-Griffeath shape theorem \eqref{eq:shapethm}, the premalignant population eventually takes on a convex, symmetric shape, and representative simulations suggest that this limiting shape is a union of two-dimensional Euclidean disks.
    }
    \label{fig:modeldynamics}
\end{figure}

Let $\xi_t^A$ denote the set of sites in $\Z^2\times \Z_w$ occupied by type-1 cells at time $t$, given the initial condition $\xi_0^A = A$ with $A \subseteq \Z^2\times \Z_w$.
Our baseline assumption is that the system starts out with a single type-1 cell at the origin, i.e.~$A = \{0\}$.
Define
\begin{align*} 
\tau_\varnothing^0= \tau_\varnothing^0(\beta) :=\inf\{t\geq0:\xi_t^0=\varnothing\}    
\end{align*}
as the time of extinction of the process 
starting from the origin, with $\inf \varnothing = \infty$. 
The discrete-time jump process embedded in $(|\xi_t^0|)_{t \geq 0}$, with $|\cdot|$ denoting cardinality, is a simple, biased random walk on the nonnegative integers with absorption at 0.
The walk moves up with probability $(1+\beta)/(2+\beta)$ and down with probability $1/(2+\beta)$.
It follows by the gambler's ruin formula that a type-1 mutant 
expands into a successful type-1 clone with probability
\begin{align} \label{eq:gamblersruin}
    \P(\tau_\varnothing^0 = \infty) = {\beta}/({1+\beta}).
\end{align}
If the mutant is successful, the Bramson-Griffeath shape theorem on $\Z^2$ \cite{BraGri81, BraGri80} can be extended to show that the mutant clone on $\Z^2 \times \Z_w$ eventually assumes a convex, symmetric shape whose diameter grows linearly in time.
To carry out the extension, we need to introduce a notion of spatial scaling and distance in our stacked-lattice setting.
To that end, we go to the larger space $\R^2 \times \Z_w$.
For $t \in \mathbb{R}$ and $A \subseteq \R^2 \times \Z_w$, we define the scalar multiplication operation $tA$ as multiplication along the first two coordinates:
\begin{align} \label{eq:scalarmultpldef}
    tA := \{(tx_1,tx_2,x_3): (x_1,x_2,x_3) \in A\}.
\end{align}
The distance of a point $x = (x_1,x_2,x_3) \in \R^2 \times \Z_w$ from the origin is defined in terms of its two-dimensional projection as
\begin{align} \label{eq:seminormdef}
    ||x|| = ||(x_1,x_2,x_3)|| := \sqrt{x_1^2+x_2^2}.
\end{align}
Note that $\R^2 \times \Z_w$ is not a vector space since $(s+t)x=sx+tx$ does not hold in general for $s,t \in \R$ and $x \in \R^2 \times \Z_w$.
However, it does hold up to addition by a vector parallel to ${e}_3 = (0,0,1)$, which is sufficient for establishing a shape theorem (Section \ref{app:BGtheorem}).

Using the above definitions, we can modify Bramson and Griffeath's arguments to show that there exists a set $D = D(\beta)$ on $\R^2 \times \Z_w$
so that for any $\varepsilon>0$,
\begin{align} \label{eq:shapethm}
\P\big(\exists t_*<\infty: (1-\varepsilon)tD \cap (\mathbb{Z}^2 \times \Z_w) \subseteq \xi_t^0 \subseteq (1+\varepsilon)tD,\; t\geq t_* \ \big| \ \tau_\varnothing^0=\infty\big)=1.
\end{align}
The set $D$ can be written as $D = \bigcup_{i \in \Z_w} (X \times \{i\})$, where 
$X \subseteq \mathbb{R}^2$ 
is convex,
and $X$ has the same symmetries as those of $\Z^2$ that leave the origin fixed.
For example, $(x_1,x_2) \in X$ implies $(-x_1,x_2),(x_1,-x_2) \in X$ (reflection across an axis) and $(-x_2,x_1) \in X$ (rotation by 90 degrees).
The shape theorem does not offer a more explicit description of $D$, but representative simulations suggest that it is a union of two-dimensional Euclidean disks (Fig.~\ref{fig:modeldynamics}b).
In Section \ref{app:BGtheorem}, we discuss how to adapt Bramson and Griffeath's arguments from $\Z^2$ to $\Z^2 \times \Z_w$, and we provide an implicit definition of the set $D$ in terms of the process $(\xi_t^0)_{t \geq 0}$.

To determine the rate of expansion of the mutant clone $\xi_t^0$, 
we denote the radius of $D = D(\beta)$ by $c_w(\beta)$, and define it in terms of the projection of $D$ onto the $x$-axis as
\begin{align}
\{ x \in \R: (x,y,z) \in D \} =: [-c_w(\beta), c_w(\beta)].
\label{correctasymptoticspeed}
\end{align}
We furthermore define
\begin{align} \label{eq:birthfirstwodimensions}
p_w := \begin{cases} 1, & w=1, \\
4/5, & w=2, \\ 2/3, & w>2, \end{cases}    
\end{align}
as the probability that a cell giving birth on $\mathbb{Z}^2 \times \Z_w$ replaces a cell occupying the same layer.
Note that each cell has four neighboring cells occupying the same layer, independently of $w$.
The difference between the $w=1$, $w=2$ and $w>2$ cases lies in the fact that there are zero, one and two neighboring cells occupying other layers, respectively.

Our main result determines $c_w(\beta)$ as a function of a small selective advantage $\beta> 0$ and tissue thickness $w \geq 1$.
Intuitively, it is clear that $c_w(\beta) \to 0$ as $\beta \to 0$.
In order to determine $c_w(\beta)$ for small $\beta$, we therefore compute its rate of decrease as $\beta \to 0$.

\begin{theorem}\label{speed_thm}
For $\beta > 0$ and $w \geq 1$, let $c_w(\beta)$ be the radius of the asymptotic shape $D = D(\beta)$ of the biased voter model $(\xi_t^0)_{t \geq 0}$ on $\Z^2 \times \Z_w$,
conditioned on the event that it does not die out, as defined by \eqref{eq:shapethm} and \eqref{correctasymptoticspeed}. 
Then, as $\beta \to 0$,
\begin{align*} 
c_w(\beta) \sim p_w \sqrt{\pi w \beta}\big/\sqrt{\log(1/\beta)} = a_w \big/\sqrt{h(\beta)}, 
\end{align*}
where $p_w$ is defined as in \eqref{eq:birthfirstwodimensions}, $a_w := p_w \sqrt{\pi w}$ and $h(\beta) := (1/\beta) \cdot \log(1/\beta)$.
\end{theorem}

\begin{figure}
    \centering
    \includegraphics[scale=1]{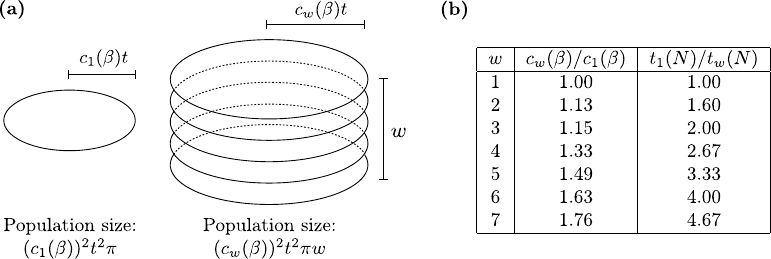}
    \caption{{\bf (a)} 
    For sufficiently large $t$, the population size of $\xi_t^0$ at time $t$ is approximately $(c_w(\beta))^2t^2\pi w$, where $c_w(\beta)$ is the radius of the asymptotic shape $D = D(\beta)$. {\bf (b)} The ratio $c_w(\beta)/c_1(\beta)$ shows how the radius of the asymptotic shape of $\xi_t^0$ changes with increasing $w$, and the ratio $t_1(N)/t_w(N)$, with $t_w(N)$ defined as in \eqref{eq:populationgrwoth}, shows how quickly the population reaches a given size $N$ as $w$ increases. The latter ratio grows according to $p_w w$.}
    \label{fig:speed_illustration}
\end{figure}

Our choice of analyzing type-1 propagation in the direction of the unit vector $e_1 = (1,0,0)$ is arbitrary.
Our arguments apply to any unit vector of the form $(\cos\theta, \sin \theta, 0)$, so propagation is the same in any direction along the first two coordinates 
in the $\beta \to 0$ regime.

In Theorem 1 of \cite{DurFooLed}, Durrett, Foo and Leder compute the propagation speed of the biased voter model on $\Z^2$ as $\sqrt{\pi (\beta/4)}/\sqrt{\log(1/\beta)} = (1/2) \sqrt{\pi\beta}/\sqrt{\log(1/\beta)}$ in the $\beta \to 0$ regime.
In contrast, our result for the $w=1$ case on $\Z^2 \times \Z_w$ is $c_1(\beta) \sim \sqrt{\pi\beta}/\sqrt{\log(1/\beta)}$ as $\beta \to 0$, which is larger by a factor of 2.
This discrepancy is due to a scaling error in the calculations in \cite{DurFooLed}.
On page 1396 of \cite{DurFooLed}, the authors change the time scale of the process in a way that reduces the fitness advantage of type-1 cells on $\Z^d$ to $\beta/(2d)$.
If the result for $\Z^2$ in \cite{DurFooLed} is read with this in mind, it is consistent with the $w=1$ case of our Theorem \ref{speed_thm}.

Note that if $\xi_t^0$, conditional on $\{\tau_\varnothing^0 = \infty\}$, is a union of two-dimensional disks of radius $c_w(\beta)t$ at time $t$, the total area across all $w$ layers is $(c_w(\beta))^2t^2 \pi w$ (Fig.~\ref{fig:speed_illustration}a). Thus, given $N$ sufficiently large, the time it takes for the type-1 population to reach size $N$ is approximately
\begin{align} \label{eq:populationgrwoth}
t_w(N) = (p_w \pi w)^{-1} h(\beta)^{1/2} N^{1/2}.
\end{align}
This implies that going from $w=1$ layer to $w > 1$ layers accelerates population growth by $p_w w$. For example, population growth is  twice as fast for $w=3$ layers as for $w=1$ layer, and over three times as fast for $w=5$ layers (Fig.~\ref{fig:speed_illustration}b).

\section{Outline of proof of main result}\label{sec_proofoutline}

\subsection{Duality} \label{sec:duality}

The biased voter model $(\xi_t^A)_{t \geq 0}$ admits a simple graphical construction, 
which allows us to define the entire system $\{(\xi_t^A)_{t \geq 0}: A \subseteq \mathbb{Z}^2 \times \Z_w\}$ on a common probability space, using a countable family of Poisson processes.
For a description of this construction, see e.g.~Section 2 of \cite{durrett95}, Section 3 of \cite{durrett1988lecture} or Appendix A of \cite{DurFooLed}.
By tracing the ancestry of particles in $\xi_t^A$ backwards in time, the biased voter model gives rise to a system of {\em branching coalescing random walks} $(\tilde\zeta_t^B)_{t \geq 0}$ on $\Z^2\times \Z_w$, which satisfy the duality relation
\begin{align}
\P(\xi_t^A \cap B \neq \varnothing) = \P(\tilde{\zeta}_t^B \cap A \neq \varnothing), \quad A,B \subseteq \Z^2\times \Z_w. \label{dual_relation}
\end{align}
The process $(\tilde\zeta_t^B)_{t \geq 0}$ can be described as follows: Each particle performs a simple, symmetric random walk (SSRW) on $\Z^2 \times \Z_w$ with jump rate 1, i.e.~each particle jumps at rate 1 to a randomly chosen neighboring site. Furthermore, each particle gives birth to a new particle at rate $\beta$, with the parent particle staying put and the daughter particle placed at a randomly chosen neighboring site. Any time two particles meet, they coalesce into a single particle. The following elementary properties of $\xi_t^A$ and $\tilde\zeta_t^B$ are easily verified:
\begin{itemize}
\item {\em Additivity}: For each $A,B \subseteq \Z^2 \times \Z_w$ and $t \geq 0$,
\begin{align} \label{eq:additivity}
    \xi_t^{A \cup B} = \xi_t^A \cup \xi_t^B \quad\text{and}\quad \tilde\zeta_t^{A \cup B} = \tilde\zeta_t^A \cup \tilde\zeta_t^B.
\end{align}
\item {\em Monotonicity:} For $A \subseteq B$ and $t \geq 0$,
\begin{align} \label{eq:monotonicity}
    \xi_t^{A} \subseteq \xi_t^B \quad\text{and}\quad \tilde\zeta_t^{A} \subseteq \tilde\zeta_t^B.
\end{align}
\item {\em Translation invariance:} For each $A \subseteq \Z^2 \times \Z_w$ and $x \in \Z^2 \times \Z_w$,
\begin{align} \label{eq:translinvariance}
    (\xi_t^{A})_{t \geq 0} \stackrel{d}{=} (x+\xi_t^{A-x})_{t \geq 0} \quad\text{and}\quad (\tilde\zeta_t^{A})_{t \geq 0} \stackrel{d}{=} (x+\tilde\zeta_t^{A-x})_{t \geq 0}.
\end{align}
\item {\em Symmetry:} For each $A,B \subseteq \Z^2 \times \Z_w$:
\begin{align} \label{eq:symmetry}
    (\xi_t^{A})_{t \geq 0} \stackrel{d}{=} (-\xi_t^{-A})_{t \geq 0} \quad\text{and}\quad (\tilde\zeta_t^{A})_{t \geq 0} \stackrel{d}{=} (-\tilde\zeta_t^{-A})_{t \geq 0}.
\end{align}
\end{itemize}
Due to the duality relation \eqref{dual_relation}, we can use the dual process $\tilde\zeta_t$ to study the propagation speed of the biased voter model $\xi_t$.
Direct analysis of $\tilde\zeta_t$ is complicated by its coalescing nature.
However, it turns out that when $\beta$ is small, most coalescence events in $\tilde\zeta_t$ will be between parent and daughter shortly after the daughter's birth.
Before elucidating this property further, we need to establish some fundamental properties of the dual process.

\subsection{Fundamental properties of dual process}

We begin our analysis of the dual process $\tilde\zeta_t$ by determining the long-run position of individual particles. In the following lemma, we extend the local central limit theorem (LCLT) for the discrete-time SSRW on $\Z^2$ to the multilayered setting $\Z^2 \times \Z_w$.
Since our arguments apply to $\mathbb{Z}^d \times \Z_w$ for any $d \geq 1$, we state and prove the result for the general case.
The proof is simple and proceeds as follows.
First, we decompose the discrete-time SSRW $(S_n)_{n \geq 0}$ on $\Z^d\times \Z_w$ into walks on $\mathbb{Z}^d$ and $\Z_w$, respectively,
and use a large deviations estimate to bound the number of steps in each direction.
We then apply the LCLT on $\mathbb{Z}^d$ to the $\Z^d$-walk (Theorem 2.1.3 of \cite{lawler2010}), and a convergence theorem for finite Markov chains to the $\Z_w$-walk (Theorem 4.9 of \cite{levin2017markov}).
Note that for $w=1$, the SSRW on $\Z^d \times \Z_w$ is equivalent to the SSRW on $\Z^d$, which is periodic.
For $w>1$, on the other hand, the SSRW on $\Z^d \times \Z_w$ is periodic if and only if $w$ is even.

\begin{lemma}[LCLT on $\Z^d \times \Z_w$] \label{lclt}
  Let $(S_n)_{n \geq 0}$ be the discrete-time SSRW on $\mathbb{Z}^d \times \Z_w$ with $S_0=0$. Set $b_1 :=2$, $b_w := 2$ if $w>1$ is even and $b_w := 1$ if $w>1$ is odd, and define
\begin{align} \label{eq:birthfirsttowdimensionsgeneral}
    p_{w,d} := \begin{cases} 1, & w=1, \\ 
    {2d}/({2d+1}), & w=2, \\ {d}/({d+1}), & w>2, \end{cases}
\end{align}
    as the probability that $(S_n)_{n \geq 0}$ takes a step in the $\mathbb{Z}^d$-direction, 
    with $p_{w,2} = p_w$ as defined in \eqref{eq:birthfirstwodimensions}.
    Then, for any $n \geq 1$ and $x \in \Z^d \times \Z_w$ so that $\P(S_n=x)>0$,
\begin{align} \label{eq:lcltdiscrete}
    n^{d/2} \P(S_n=x) = (b_w/w) \big(d/(2\pi p_{w,d})\big)^{d/2} + o(1). 
\end{align}
\end{lemma} 

\begin{proof}
Section \ref{app:localclt}.
\end{proof}

Let $(Z_t)_{t\geq 0}$ be the continuous-time SSRW on $\Z^d \times \Z_w$ with jump rate $\alpha > 0$ and $Z_0=0$.
Since in continuous time, the coordinates move independently, we can decompose the walk into independent components $Z_t = (\hat{Z}_t, Z^w_t)$, where $(\hat{Z}_t)_{t \geq 0}$ is the SSRW on $\Z^d$ with jump rate $\alpha p_{w,d}$, and $(Z^w_t)_{t \geq 0}$ is the SSRW on $\Z_w$ with jump rate $\alpha (1-p_{w,d})$, with $p_{w,d}$ defined as in \eqref{eq:birthfirsttowdimensionsgeneral}.
For $\hat{x} \in \Z^d$, set
\begin{align*}
   \textstyle  
   p_t(\hat{x}) := \big({d}/({2\pi \alpha p_{w,d}t})\big)^{d/2} \exp\big(\!-\!d \|\hat{x}\|^2/({2 \alpha p_{w,d} t})\big).
\end{align*}
By Theorem 2.1.3 of \cite{lawler2010}, 
there exists $c>0$ so that for all $\hat{x} \in \Z^d$ and all $t>0$,
\begin{align} \label{eq:lcltZdsqrt}
\textstyle 
& \textstyle|\P(\hat{Z}_t = \hat{x}) - p_t(\hat{x})| \leq c/(\alpha p_{w,d} t)^{(d+2)/2}.
\end{align}
Expression \eqref{eq:lcltZdsqrt} and independence of coordinates imply a continuous-time version of \eqref{eq:lcltdiscrete},
\begin{align*}
    \textstyle \lim_{t \to \infty} (\alpha t)^{d/2} \P(Z_t=x) = (1/w) \big(d/(2\pi p_{w,d})\big)^{d/2}, \quad x \in \Z^d \times \Z_w,
\end{align*}
since $(Z_t^w)_{t \geq 0}$ converges to the uniform distribution on $\Z_w$.

Whenever a new particle is born into the dual process $\tilde\zeta_t$, the parent and daughter perform independent SSRWs $Z_t^1$ and $Z_t^2$ on $\mathbb{Z}^2 \times \Z_w$ with jump rate 1, started at neighboring sites.
We next determine the asymptotic tail of $T_0$, the time at which $Z_t^1$ and $Z_t^2$ first meet, which we can equivalently view as the time of the first visit to the origin of the SSRW $\bar Z_t := Z_t^1-Z_t^2$ with jump rate 2.
Due to the recurrence of the SSRW in two dimensions, the two walks $Z_t^1$ and $Z_t^2$ are guaranteed to meet in finite time.
In the following lemma, we compute the rate of decrease of $\P(T_0>t)$ as $t \to \infty$.
In the proof, we generalize an argument given by Dvoretzky and Erd\"os for the SSRW on $\Z^2$ in \cite{dvor51}.
The only modification necessary is to substitute the LCLT on $\Z^2$ by the LCLT on $\Z^2 \times \Z_w$ (Lemma \ref{lclt}).

\begin{lemma}[Asymptotic tail of $T_0$]
\label{ret_time_thm}
Let $(Z_t)_{t \geq 0}$ be the SSRW on $\Z^2\times \Z_w$ with jump rate $\alpha> 0$, started at a nearest neighbor of the origin. Set $T_0:=\inf\{t \geq 0:Z_t=0\}$ and define
\begin{align*} 
\mu_w := p_w \pi w =
\begin{cases} \pi, & w=1, \\ 
(4/5){\pi w}, & w=2, \\ (2/3){\pi w}, & w>2, \end{cases}    
\end{align*}
where $p_w$ is the probability given by \eqref{eq:birthfirstwodimensions}. Then
\begin{align*} 
\P(T_0 > t)  \sim \mu_w/\log t, \quad t \to \infty.
\end{align*}
\end{lemma}

\begin{proof}
Section \ref{app:tailasymptotics}.
\end{proof}

Recall that each particle in $\tilde\zeta_t$ gives birth to a new particle at rate $\beta$, so the mean time between births along a particular lineage is $1/\beta$. Set
\begin{align} \label{eq:tau(beta)}
\tau(\beta) := (1/\beta) \big(1/\sqrt{\log(1/\beta)}\big).
\end{align}
By Lemma \ref{ret_time_thm}, a new particle avoids coalescence with its parent particle during the first $\tau(\beta)$ time units of its existence with probability
\begin{align*}
\P(T_0>\tau(\beta)) \sim \mu_w/\log(1/\beta), \quad \beta \to 0.
\end{align*}
Thus, most new particles coalesce with their parents before time $\tau(\beta)$, and since $\tau(\beta) = o(1/\beta)$, they are unlikely to produce their own offspring before coalescing. Ignoring such particles should simplify the process considerably without affecting its long-run growth. This is the basic idea of a pruning procedure suggested by Durrett and Z{\"a}hle \cite{durrett2007}, which we use to set up an approximation scheme to prove our main result (see Sections \ref{sec:upperbound} and \ref{sec:lowerboundapprox} below). 
The specific form of $\tau(\beta)$ in \eqref{eq:tau(beta)}
ensures not only that inconsequential particles are ignored, but also that new particles that do avoid coalescence up until time $\tau(\beta)$ are neither too far away from nor too close to their parent particles at that time.

It turns out that a separate approximation scheme is required for establishing an upper bound and a lower bound on the propagation speed of $\tilde\zeta_t$. Before describing the scheme in detail, we discuss it at a high level and provide some intuition for our main result.

\subsection{Overview of approximation scheme and intuition for main result} \label{sec:intuition}

Since particles in $\tilde\zeta_t$ branch at rate $\beta$, branching events become less and less frequent as $\beta \to 0$, and most events produce particles that coalesce with their parents shortly after birth.
If we ``reject'' branching events where new particles are lost to coalescence quickly, the rate of ``accepted'' events along a particular lineage is of order
\begin{align*} \label{eq:branchingrateunscaled}
\beta \cdot \big(\mu_w/\log(1/\beta)\big) = \mu_w/h(\beta), \quad \beta \to 0,
\end{align*}
where $h(\beta) := (1/\beta) \cdot \log(1/\beta)$. Assume for the moment that particles in $\tilde\zeta_t$ are sufficiently spread out that we can ignore other coalescence events. We then obtain a branching random walk with branching rate $\mu_w/h(\beta)$ and average number of particles $\exp\big((\mu_w/h(\beta)) t\big)$ alive at time $t$. If we project onto the $x$-axis, each particle performs a SSRW
on $\Z$ with jump rate $p_w/2$, where $p_w$ is defined as in \eqref{eq:birthfirstwodimensions}. For large $t$, its position has approximate distribution
\[
(\sqrt{p_w\pi t})^{-1} \cdot \exp\big(\!-\!x^2/(p_w t)\big), \quad x \in \Z,
\]
and the particle intensity (average number of particles) at $x \in \Z$ is approximately
\[
(\sqrt{p_w\pi t})^{-1}
\cdot \exp\big((\mu_w/h(\beta)) t -x^2/(p_w t)\big).
\]
If we set $|x| = ct$ for $c>0$, 
this quantity is nonzero in the $t \to \infty$ limit as long as
$c^2 < p_w \mu_w/h(\beta)$ i.e.~$c < a_w/\sqrt{h(\beta)}$, since $a_w = p_w \sqrt{\pi w} = \sqrt{p_w \mu_w}$. This suggests a long-run expansion rate of $a_w/\sqrt{h(\beta)}$ per unit time, which is our main result.

To make this argument rigorous, we need to show that for small $\beta$, the dual process $\tilde\zeta_t$ sufficiently resembles a branching random walk (BRW) with branching rate $\mu_w/h(\beta)$.
Since the time between accepted branching events is of order $h(\beta) \to \infty$ as $\beta \to 0$, and in this time, fluctuations in the movement of individual particles are of order $\sqrt{h(\beta)}$, it makes sense to speed up time by $h(\beta)$ and reduce space by $\sqrt{h(\beta)}$. We therefore introduce the scaled dual process
\begin{align} \label{eq:defnscaleddual}
    \tilde\zeta_t^\beta := h(\beta)^{-1/2} \cdot \tilde\zeta_{h(\beta)t},
\end{align}
and our goal is to show that for small $\beta$, this process sufficiently resembles a BRW with branching rate $\mu_w$.
Recall that by the definition \eqref{eq:scalarmultpldef} of scalar multiplication on $\R^2 \times \Z_w$, the spatial scaling by $h(\beta)^{-1/2}$ only affects the first two coordinates.

In the upper bound proof, the main work resides in identifying which branching events to accept, and in analyzing parent-daughter interactions under the accepted events. In the lower bound proof, 
we can only approximate $\tilde\zeta_t^\beta$ with a branching random walk on finite time intervals. We therefore discretize time and space and apply a percolation argument to obtain the long-run propagation speed of $\tilde\zeta_t^\beta$.

\def\arraystretch{1.2}
\begin{table}[]
\small
    \centering
    \begin{tabular}{|c|c|c|c|}
    \hline
             \multirow{2}{*}{$\xi_t$} & \multirow{2}{*}{Biased voter model} & \multirow{2}{*}{Section \ref{sec_mainstatement}} & Type-0 particles divide at rate 1, type-1 at rate $1+\beta$. \\
         &&& A neighbor selected uniformly at random is replaced.  \\
         \hline
         \multirow{2}{*}{$\tilde\zeta_t$} & \multirow{2}{*}{Dual process} & \multirow{2}{*}{Section \ref{sec:duality}} & Branching coalescing random walk (BCRW).  \\
         &&& Particles jump at rate 1, branch at rate $\beta$.  \\
         \hline
         \multirow{2}{*}{$\bar\phi_t$} & \multirow{2}{*}{Unaltered BRW} & \multirow{2}{*}{Section \ref{sec:unmodifiedbrw}} & Branching random walk (BRW) obtained by   \\
         &&& ignoring all coalescence events in dual process $\tilde\zeta_t$. \\
         \hline
        \multirow{2}{*}{$\mathring\phi_t$} & Pruned BRW, & \multirow{2}{*}{Section \ref{sec:unmodifiedbrw}} & BRW obtained by ignoring new particles in $\bar\phi_t$ \\
         &upper bound && that coincide quickly with their parent particles.  \\
         \hline 
                 \multirow{2}{*}{$\phi_t$} & Simple BRW, & \multirow{2}{*}{Section \ref{sec:markovianupper}} & BRW obtained by modifying particle paths in $\mathring\phi_t$ \\
         &upper bound && to uncondition movement at branching events. \\
         \hline 
                  \multirow{2}{*}{$\hat\zeta_t$} & Pruned dual, & \multirow{2}{*}{Section \ref{sec:pruneddual}} & BCRW obtained by ignoring new particles \\
         &lower bound&& that coalesce quickly with {\em any} particle in $\hat\zeta_t$.  \\
         \hline
          \multirow{2}{*}{$\mathring\psi_t$} & Pruned BRW, & \multirow{2}{*}{Section \ref{eq:prunedbrwlowerbound}} & BRW obtained by ignoring new particles in $\bar\phi_t$ \\
         &lower bound && that coincide quickly with their parent particles.  \\
         \hline 
                 \multirow{2}{*}{$\psi_t$} & Simple BRW, & \multirow{2}{*}{Section \ref{sec:propagationpruneddual}} & BRW obtained by modifying particle paths in $\mathring\psi_t$  \\
         &lower bound && to uncondition movement at branching events. \\
         \hline
    \end{tabular}
    \caption{List of the particle processes used in the proofs of Lemmas \ref{lemma:branchingrate} to \ref{lemma:lowerbound}.}
    \label{tab:my_label}
\end{table}

\subsection{Upper bound argument} \label{sec:upperbound}

To prove an upper bound, we couple the dual process $\tilde\zeta_t$ with a pruned branching random walk $\mathring\phi_t$, which we in turn couple with a simpler BRW $\phi_t$. We then analyze the propagation speed of $\phi_t$ to obtain an upper bound on the propagation speed of $\tilde\zeta_t$. For reference, we list the processes used in the proof of our main result along with a short description in Table \ref{tab:my_label}.

\subsubsection{Definition of pruned BRW $\mathring\phi_t$} \label{sec:unmodifiedbrw}
Consider a branching random walk $\bar\phi_t$ obtained by ignoring all coalescence events in the dual process $\tilde\zeta_t$.
In other words, particles in $\bar\phi_t$ jump at rate 1 and branch at rate $\beta$, and  whenever two particles meet, both are retained.
One particle follows the path of the coalesced particle in the dual process, and the other performs a new SSRW on $\Z^2 \times \Z_w$ independently of all other particles. By construction, each particle path in $\tilde\zeta_t$ also appears in $\bar\phi_t$, but $\bar\phi_t$ contains additional paths. 
Since $\bar\phi_t$ allows multiple particles to occupy the same site, it should be viewed as a sequence of sites in $\Z^2 \times \Z_w$, as opposed to a subset of $\Z^2 \times \Z_w$.
We will call $\bar\phi_t$ the {\em unaltered} BRW to distinguish it from the {\em pruned} BRW $\mathring\phi_t$, which we define now.

For a given branching event in $\bar\phi_t$, let $T_0$ be the time at which the new particle first coincides with its parent, and let $S$ be the time at which the new particle first produces its own offspring. Recall that $S$ is exponentially distributed with mean $1/\beta$, and it is independent of $T_0$. We categorize the branching events in $\bar\phi_t$ as follows:
\begin{itemize}
\item  {\bf Type-0:} $T_0 \leq \min\{S,\tau(\beta)\}$: The new particle quickly coincides with its parent.
    \item {\bf Type-1:} $S \leq \min\{T_0,\tau(\beta)\}$: The new particle quickly produces its own offspring.
    \item {\bf Type-2:} $\tau(\beta) \leq \min\{S,T_0\}$: The new particle neither coincides with its parent nor produces its own offspring before time $\tau(\beta)$.
\end{itemize}
We refer to $[0,T_0]$, $[0,S]$ and $[0,\tau(\beta)]$, respectively, as the {\em decision period} for each type of event. The pruned BRW $\mathring\phi_t$ is defined as follows (Fig.~\ref{fig:pruned_brw}):
\begin{itemize}
    \item A new particle born through a type-0 branching event in $\bar\phi_t$ is ignored in $\mathring\phi_t$.
    \item A new particle born through a type-1 event is introduced to $\mathring\phi_t$ at time $S$ after birth in $\bar\phi_t$, at the location it then occupies in $\bar\phi_t$. Its offspring is viewed as a new branching event in $\bar\phi_t$ and is evaluated according to the same rules as outlined here.
    \item A new particle born through a type-2 event is introduced to $\mathring\phi_t$ at time $\tau(\beta)$ after birth in $\bar\phi_t$, at the location it then occupies.
\end{itemize}
Once a new particle is introduced to $\mathring\phi_t$, it follows the same path as in $\bar\phi_t$.
Let $\mathring\phi_t^{(k)}$ for $k=0,1,2$ be the subprocess of $\bar\phi_t$ containing offsprings of particles in $\mathring\phi_t$ that have just been born through a type-$k$ branching event and whose decision period has not yet passed. Then
\begin{align} \label{eq:dualprunedbrwupper}
\tilde\zeta_t \subseteq \mathring\phi_t \cup \mathring\phi_t^{(0)} \cup \mathring\phi_t^{(1)} \cup \mathring\phi_t^{(2)},
\end{align}
i.e.~$\mathring\phi_t$ upper bounds the dual process $\tilde\zeta_t$ if we add newborn particles whose fate has not been decided yet. Expression \eqref{eq:dualprunedbrwupper} allows us to relate the propagation speed of $\tilde\zeta_t$ to that of $\mathring\phi_t$. Before doing so, we need more information on the branching dynamics of $\bar\phi_t$ and $\mathring\phi_t$.

\begin{figure}
    \centering
    \includegraphics[scale=1]{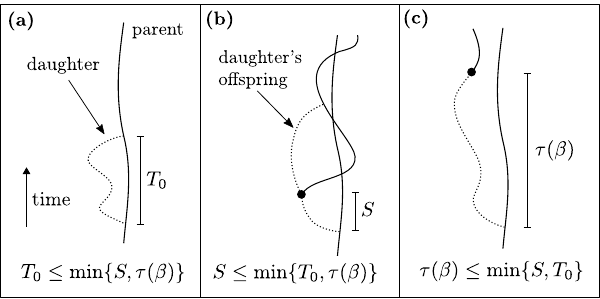}
    \caption{Categorization of branching events in the unaltered BRW $\bar\phi_t$ and definition of the pruned process $\mathring\phi_t$. {\bf (a)} If a new particle in $\bar\phi_t$ coincides quickly with its parent, it is not introduced to $\mathring\phi_t$. {\bf (b)} If the new particle produces its own offspring quickly, it gets introduced to $\mathring\phi_t$ at the time and location at which it gives birth. {\bf (c)} If the new particle neither coincides with its parent nor has its own offspring too quickly, it gets introduced to $\mathring\phi_t$ at time $\tau(\beta)$ after birth at the location it then occupies.}
    \label{fig:pruned_brw}
\end{figure}

\subsubsection{Branching in $\bar\phi_t$ and $\mathring\phi_t$}
In the following lemma, we show that in the $\beta \to 0$ regime, almost all branching events of the unaltered BRW $\bar\phi_t$ are type-0. In other words, only a small proportion of branching events is accepted to produce the pruned BRW $\mathring\phi_t$. We also show that type-2 branching events are much more frequent than type-1 events, meaning that most particles introduced to $\mathring\phi_t$ neither coincide with their parent nor produce their own offspring by time $\tau(\beta)$. 
We finally produce moment bounds on the distance traveled by a new particle during its decision period, as well the separation between parent and daughter throughout the decision period.

\begin{lemma} \label{lemma:branchingrate}
Let $(Z_t^1)_{t \geq 0}$ and $(Z_t^2)_{t \geq 0}$ be independent SSRWs on $\mathbb{Z}^2 \times \Z_w$ with jump rate 1, started at 0 and a nearest neighbor of 0. Set $T_0 := \inf\{t \geq 0: Z_t^1 = Z_t^2\}$, and let $S$ be an exponential random variable with mean $1/\beta$, independent of $(Z_t^1)_{t \geq 0}$ and $(Z_t^2)_{t \geq 0}$.
Then
\begin{enumerate}[(1)]
    \item $\alpha_0(\beta) := \P(T_0 \leq \min\{S,\tau(\beta)\}) \to 1$ as $\beta \to 0$.
    \item $\alpha_1(\beta) :=\P(S \leq \min\{T_0,\tau(\beta)\}) = \Theta\big(1/(\log(1/\beta))^{3/2}\big)$ as $\beta \to 0$,
    \item $\alpha_2(\beta) :=\P(\tau(\beta) \leq \min\{S,T_0\}) \sim \mu_w/\log(1/\beta)$ as $\beta \to 0$.
\end{enumerate}
Furthermore, there exists $C>0$ so that for sufficiently small $\beta$,
\begin{enumerate}[(1)]
    \item[(4)] $\E\big[\sup_{t \leq \tau(\beta)} ||Z_t^1||^j \,\big|\, T_0 \leq \min\{S,\tau(\beta)\}\big] \leq C j! \tau(\beta)^{j/2}$, \quad $j \geq 1$,
    \item[(5)] $\E\big[\sup_{t \leq \tau(\beta)} ||Z_t^1||^j \,\big|\, S \leq \min\{T_0,\tau(\beta)\}\big] \leq C j! (\log(1/\beta))^{1/2} \tau(\beta)^{j/2}$, \quad $j \geq 1$,
    \item[(6)] $\E\big[\sup_{t \leq \tau(\beta)} ||Z_{t}^1||^j \,\big|\, \tau(\beta) \leq \min\{S,T_0\}\big] \leq C j! (\log(1/\beta))^{1/2} \tau(\beta)^{j/2}$, \quad  $j \geq 1$.
\end{enumerate}
Each of (4)-(6) continues to hold if $Z_t^1$ is replaced by $Z_t^2$ or  $\bar{Z_t} = Z_t^1-Z_t^2$.
\end{lemma}

\begin{proof}
Section \ref{app:branchingratepruneddual}.
\end{proof}

\subsubsection{From dual $\tilde\zeta_t$ to pruned BRW $\mathring\phi_t$}

Using \eqref{eq:dualprunedbrwupper} and Lemma \ref{lemma:branchingrate}, we can obtain the following relationship between the propagation speed of the dual process $\tilde\zeta_t$ and the pruned BRW $\mathring\phi_t$ (Lemma \ref{lemma:approximateupperbound}). By \eqref{eq:dualprunedbrwupper}, $\mathring\phi_t$ upper bounds $\tilde\zeta_t$ if we add newborn particles whose fate has not been decided yet. By (4)-(6) in Lemma \ref{lemma:branchingrate}, these newborn particles will not be too far from their parent particles in $\mathring\phi_t$, so adding them should not materially affect the propagation speed of $\mathring\phi_t$.
As motivated by the discussion in Section \ref{sec:intuition}, we perform our analysis
using the scaled processes $\tilde\zeta_t^{\beta,0} := h(\beta)^{-1/2} \tilde\zeta_{h(\beta)t}^0$ and $\mathring\phi_t^{\beta,0} := h(\beta)^{-1/2} \mathring\phi_{h(\beta)t}^0$, where the 0 means that the processes are started with a single particle at the origin.

\begin{lemma} \label{lemma:approximateupperbound}
Fix $a>0$ and $0 < \rho < 1$, and define $A_r := [r,\infty) \times \mathbb{R} \times \Z_w$.
For each $\delta>0$, there exist $M>0$ and $\beta_0>0$ so that
\begin{align*}
     &\textstyle \P\big(\tilde\zeta_{t}^{\beta,0} \cap A_{b+at} \neq \varnothing\big) \leq 4\,\P\big(\mathring\phi_{t}^{\beta,0} \cap A_{b+\rho at} \neq \varnothing \big) + M e^{-\delta t}, \quad \beta \leq \beta_0, \; t>0, \; b \in \R.
\end{align*}
\end{lemma}

\begin{proof}
Section \ref{app:pruneddualapproxdual}.
\end{proof}

\subsubsection{Definition of simple BRW $\phi_t$} \label{sec:markovianupper}

At each accepted branching event of the pruned BRW $\mathring\phi_t$ (type-1 or type-2), the new particle is introduced with a time delay, and the location at which it is introduced is conditioned on it not coinciding too quickly with its parent. The parent's path during the decision period is likewise influenced by this conditioning. We next couple $\mathring\phi_t$ with a simpler BRW $\phi_t$ where we modify particle paths as follows:
\begin{itemize}
    \item At each accepted branching event of $\mathring\phi_t$, the paths of parent and daughter during the decision period are replaced by two independent SSRWs started at the parent's location.
    From the end of the decision period onward, the two new walks make the same transitions as parent and daughter make in $\mathring\phi_t$.
    An illustration of this procedure is shown in Figure \ref{fig:markovian_brw}, and a more formal mathematical description is given in the proof of Lemma \ref{lemma:approximateupperbound2} (Section \ref{app:brwapproxbrw_upper_2}).
\end{itemize}
With these modifications, both parent and daughter follow independent, unconditioned paths at each branching event of $\phi_t$. 
We must make further modifications, however, since the path followed by the parent at a type-0 branching event in $\mathring\phi_t$, in which case the daughter is not introduced to $\mathring\phi_t$, is conditioned on coinciding quickly with the daughter. This conditioning will affect particle paths in $\phi_t$ if not addressed. We therefore make the following modifications:
\begin{itemize}
    \item At a type-0 branching event in $\mathring\phi_t$, the parent follows a path $(Z_t^1)_{t \geq 0}$ conditioned on $\{T_0 \leq \min\{S,\tau(\beta)\}\}$ during $[0,\tau(\beta)]$, in the notation of Lemma \ref{lemma:branchingrate}.
    \begin{itemize}
        \item With probability $\alpha_0(\beta)$, with $\alpha_0(\beta)$ defined as in Lemma \ref{lemma:branchingrate}, we make no modification to the parent's path. 
        \item With probability $\alpha_1(\beta)$, we replace the parent's path on $[0,\tau(\beta)]$ with a path $(Z_t^1)_{t \geq 0}$ conditioned on $\{S \leq \min\{T_0,\tau(\beta)\}\}$. From time $\tau(\beta)$ onward, the new path makes the same transitions as the parent in $\mathring\phi_t$.
        \item With probability $\alpha_2(\beta)$, we replace the parent's path on $[0,\tau(\beta)]$ with a path $(Z_t^1)_{t \geq 0}$ conditioned on $\{\tau(\beta) \leq \min\{S,T_0\}\}$. From time $\tau(\beta)$ onward, the new path makes the same transitions as the parent in $\mathring\phi_t$.
    \end{itemize}
\end{itemize}
With these modifications, we remove any effect of daughter particles not introduced to $\mathring\phi_t$ on the paths followed by their parent particles in $\phi_t$. By the above construction, $\phi_t$ has the following three simplifying properties:
\begin{itemize}
    \item Particles follow independent, unconditioned SSRWs at all times.
    \item Time between branching events is exponentially distributed.
    \item New particles are born to their parents' locations.
\end{itemize}

\begin{figure}
    \centering
    \includegraphics[scale=0.95]{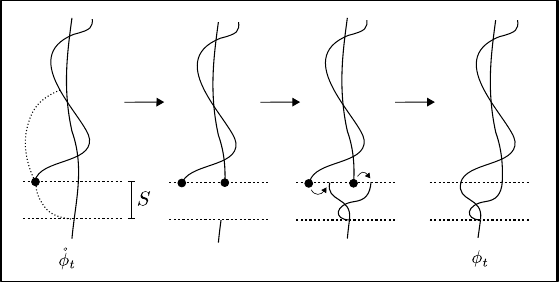}
    \caption{The simple BRW $\phi_t$ is obtained from the pruned BRW $\mathring\phi_t$ by making modifications to the paths followed by parent and daughter at each branching event of $\mathring\phi_t$. In the figure, we show the procedure for a type-1 branching event. Two new independent SSRWs are started at the parent's location at the beginning of the decision period, which replace the paths of parent and daughter during the decision period. From the end of the decision period onward, the paths followed by parent and daughter in $\mathring\phi_t$ are shifted to meet the two new walks in $\phi_t$.
    }
    \label{fig:markovian_brw}
\end{figure}

\subsubsection{From pruned BRW $\mathring\phi_t$ to simple BRW $\phi_t$} \label{sec:fromprunedbrwtomarkovian}

Working with the scaled versions $\mathring\phi_t^\beta := h(\beta)^{-1/2} \mathring\phi_{h(\beta)t}$ and $\phi_t^\beta := h(\beta)^{-1/2} \phi_{h(\beta)t}$, we show in the following lemma that the path followed by an arbitrary particle in $\phi_t^\beta$ is never too far away from the corresponding path in $\mathring\phi_t^\beta$ when $\beta$ is small. 
Since $\phi_t^\beta$ is defined by perturbing particle paths in $\mathring\phi_t^\beta$ at branching events, we need to show that the accumulated perturbation up until time $t$ is not too large. To do so, we first establish an upper bound on the number of perturbations by time $t$, and we then use the moment bounds established in (4)-(6) of Lemma \ref{lemma:branchingrate} to bound the accumulated perturbation.

\begin{lemma} \label{lemma:approximateupperbound2}
For a particle chosen uniformly at random from $\phi_t^{\beta,0}$, let $(Y_s^\beta)_{s \leq t}$ be the path followed by this particle and its ancestors, and let $(\mathring Y_s^\beta)_{s \leq t}$ be the corresponding path in $\mathring \phi_t^{\beta,0}$.  Then, for any $r>0$ and $\delta>0$, there exist
$M>0$ and  $\beta_0>0$ so that
\begin{align*}
     &\textstyle \P\big(\sup_{s \leq t} ||Y_s^\beta-\mathring Y_s^\beta||>r t\big) \leq M e^{-\delta t}, \quad \beta \leq \beta_0, \; t>0.
\end{align*}
\end{lemma}

\begin{proof}
Section \ref{app:brwapproxbrw_upper_2}.
\end{proof}

Using Lemma \ref{lemma:approximateupperbound2}, we can obtain the following relationship between the propagation speed of $\mathring\phi_t^\beta$ and $\phi_t^\beta$ (Lemma \ref{lemma:approximateupperbound3}). In the proof, we use the fact that the mean number of particles alive at time $t$ in $\phi_t^\beta$ is $\exp((\mu_w+o(1))t)$, and that the error in approximating $\mathring\phi_t^\beta$ with $\phi_t^\beta$ on a particle-by-particle basis is sufficiently small by Lemma \ref{lemma:approximateupperbound2} to ensure a small total error.

\begin{lemma} \label{lemma:approximateupperbound3}
Fix $a>0$ and $0 < \rho < 1$, and define $A_r := [r,\infty) \times \mathbb{R} \times \Z_w$.
For each $\delta>0$, there exist
$M>0$ and 
$\beta_0>0$ so that
\begin{align*}
     &\textstyle \P\big(\mathring\phi_{t}^{\beta,0} \cap A_{b+at} \neq \varnothing\big) \leq \P\big(\phi_{t}^{\beta,0} \cap A_{b+\rho at} \neq \varnothing \big) + M e^{-\delta t}, \quad \beta \leq \beta_0, \; t>0, \; b \in \R.
\end{align*}
\end{lemma}

\begin{proof}
Section \ref{app:brwapproxbrw_upper}.
\end{proof}

\subsubsection{Upper bound result for $\xi_t$}

With the above ingredients, we can establish the following upper bound result on the propagation speed of the biased voter model $(\xi_t^0)_{t \geq 0}$ on $\Z^2 \times \Z_w$ conditioned on nonextinction (Lemma \ref{lemma:upperbound}). The proof is split into three key steps. First, we remove the conditioning on nonextinction by waiting until $\xi_t^0$ has covered a sufficiently large box. We then introduce duality using \eqref{dual_relation} and use Lemmas \ref{lemma:approximateupperbound} and \ref{lemma:approximateupperbound3} to pass from the dual process $\tilde\zeta_t$ to the simple BRW $\phi_t$. We finally obtain the desired result by analyzing the tail of $\phi_t$. 

\begin{lemma} \label{lemma:upperbound}
Define $\tau_{\varnothing}^A := \min\{t \geq 0 : \xi_{t}^A = \varnothing \}$ for $A \subseteq \mathbb{Z}^2 \times \Z_w$. For each $\kappa>1$, there 
exists a family of random variables $(S_{\beta})_{\beta>0}$, with $\P(S_\beta<\infty|\tau_\varnothing^0=\infty)=1$ for each $\beta>0$, so that
\begin{align*}
\textstyle \lim_{\beta\to 0} \liminf_{t\to\infty} \P\big(
\lceil\kappa a_wh(\beta)^{1/2} t\rceil e_1 \notin \xi_{S_\beta+h(\beta)t}^0 \,\big|\,  \tau_{\varnothing}^0 = \infty \big)  = 1,
\end{align*}
where $a_w := p_w \sqrt{\pi w}$. 
\end{lemma}

\begin{proof}
Section \ref{app:upperboundproof}.
\end{proof}

\subsection{Lower bound argument} \label{sec:lowerboundapprox}

To prove a lower bound, we couple the dual process $\tilde\zeta_t$ with a pruned dual process $\hat\zeta_t$ and a pruned BRW $\mathring\psi_t$, which we in turn couple with a simpler BRW $\psi_t$. Unfortunately, as mentioned previously, the scaled versions of these processes only behave similarly on finite time intervals. 
We therefore discretize time and space and apply a percolation argument to determine the propagation speed of $\psi_t$ in the discretized spacetime. 

\subsubsection{Definition of pruned dual $\hat\zeta_t$} \label{sec:pruneddual}
To define the pruned dual process $\hat\zeta_t$, we start with the dual process $\tilde\zeta_t$.
For a given branching event in $\hat\zeta_t$, let $T_0'$ be the time at which the new particle coalesces with {\em any other particle in $\hat\zeta_t$}, and let $S$ be the time at which the new particle first produces its own offspring. We categorize the branching events in $\hat\zeta_t$ as follows:
\begin{itemize}
\item {\bf Type-0:} $T_0' \leq \min\{S,\tau(\beta)\}$: The new particle quickly coalesces with another particle.
    \item {\bf Type-1:} $S \leq \min\{T_0',\tau(\beta)\}$: The new particle quickly produces its own offspring.
    \item {\bf Type-2:} $\tau(\beta) \leq \min\{T_0',S\}$: The new particle neither coalesces with another particle nor produces its own offspring before time $\tau(\beta)$.
\end{itemize}
The pruned process $\hat\zeta_t$ is obtained by only accepting type-2 branching events.
In case of acceptance, the new particle is introduced to $\hat\zeta_t$ at time $\tau(\beta)$ after birth in $\tilde\zeta_t$, at the location it then occupies.
From the time of introduction to $\hat\zeta_t$, the new particle follows the same path as in the dual process $\tilde\zeta_t$, and it coalesces with other particles in $\hat\zeta_t$.
By construction, the pruned process $\hat\zeta_t$ lower bounds $\tilde\zeta_t$ in the following sense:
    \begin{align} \label{eq:dualprunedduallower}
    \tilde\zeta_t \supseteq \hat\zeta_t, \quad t \geq 0.
    \end{align}
    
    \subsubsection{From pruned dual $\hat\zeta_t$ to pruned BRW $\mathring\psi_t$} \label{eq:prunedbrwlowerbound}
    
    As in the proof of the upper bound, we consider an unaltered BRW $\bar\phi_t$ obtained from ignoring all coalescence events in the dual process $\tilde\zeta_t$. We classify its branching events into type-0, type-1 and type-2 as in Section \ref{sec:unmodifiedbrw}, using $T_0$, the time at which a new particle first coincides with its {\em parent particle}. We then define the pruned BRW $\mathring\psi_t$ by only accepting type-2 branching events.
    Note that the difference between $\hat\zeta_t$ and $\mathring\psi_t$ is that the latter process accepts a new particle that coincides with a particle other than its parent before time $\tau(\beta)$, and it ignores any coalescence (whether with the parent or another particle) that occurs after the new particle has been introduced.
    
    The pruned BRW $\mathring\psi_t$ may not appear useful for determining a lower bound on the propagation of $\hat\zeta_t$, as its growth is not checked to the same degree by coalescence.
    However, working with the scaled versions $\hat\zeta_t^\beta := h(\beta)^{-1/2} \hat\zeta_{h(\beta)t}$ and $\mathring\psi_t^\beta := h(\beta)^{-1/2} \mathring\psi_{h(\beta)t}$, we can show that if $\hat\zeta_t^\beta$ is started with sufficient spacing between the initial particles, then as $\beta \to 0$, the only coalescence that occurs during a finite time interval will be between parent and daughter during a decision period. In other words, $\hat\zeta_t^\beta$ behaves like $\mathring\psi_t^\beta$ on finite time intervals (in the scaled spacetime).
    We obtain the following lemma, whose proof follows 
    from an induction argument similar to the one given on pages 1758-1759 of Durrett and Z{\"a}hle \cite{durrett2007}.

\begin{lemma} \label{lemma:approximation}
Set $d(\beta) := \beta^{-1/2}(\log(1/\beta))^{-1}$. Let ${\cal A} = {\cal A}(\beta)$ denote the collection of finite subsets of $\Z^2 \times \Z_w$ in which points are pairwise separated by at least $d(\beta)$.
Set $A^\beta := h(\beta)^{-1/2}A$ for $A \in {\cal A}$ and ${\cal A}^\beta := h(\beta)^{-1/2} {\cal A}$.
Then, for any $K>0$ and $T>0$,
\[
\textstyle \sup_{A \in {\cal A}, \; |A| \leq K} \P\big(\{(\hat\zeta_t^{\beta,A^\beta})_{t \leq T} \neq (\mathring\psi_t^{\beta,A^\beta})_{t \leq T}\} \cup
\{\mathring\psi_T^{\beta,A^\beta} \notin {\cal A}^\beta\}\big) 
\to 0, \quad \beta \to 0.
\]
\end{lemma}

\begin{proof}
Section \ref{app:prunedbrwapproximatedpruneddual}.
\end{proof}

\subsubsection{Propagation of pruned dual $\hat\zeta_t$} \label{sec:propagationpruneddual}

In the following lemma (Lemma \ref{lemma:percolation}), we show that for any $2/3<\theta<1$, we can find $L$ and $K$ so that if $\hat\zeta_t^\beta$ is started with $K$ particles in a box of diameter $\theta a_w L$, then $L$ time units later, there will be at least $K$ particles in an adjacent box of diameter $\theta a_wL$, with high probability. This suggests that during $[0,L]$, the propagation speed of $\hat\zeta_t^\beta$ is at least $\theta a_w$, which translates into the desired lower bound of $\theta a_w / \sqrt{h(\beta)}$ in the unscaled spacetime. 
The specific form of the result \eqref{eq:lemma9result} of Lemma \ref{lemma:percolation} enables us to define a lower-bounding percolation process (Fig.~\ref{fig:percolation}a) using a comparison theorem from Section 4 of \cite{durrett95}, in which we discretize time into blocks of length $L$ and space into boxes of diameter $\theta a_w L$.
A similar construction is carried out in Durrett and Z{\"a}hle \cite{durrett2007}, except we must shorten their time blocks from length $L^2$ to length $L$ to obtain a tight lower bound on the propagation speed of $\hat\zeta_t^\beta$.
To prove Lemma \ref{lemma:percolation}, we use Lemma \ref{lemma:approximation} to approximate $\hat\zeta_t^\beta$ with $\mathring\psi_t^\beta$, and we then approximate $\mathring\psi_t^\beta$ with a simpler BRW $\psi_t^\beta$ as in the proof of the upper bound.
We finally analyze $\psi_t^\beta$ to obtain the result.

\begin{lemma} \label{lemma:percolation}
For $0<\theta<1$ and $L>0$, define (Fig.~\ref{fig:percolation}b)
 \begin{align*}
     & I_0^\theta := [-(1/2)\theta a_w L,(1/2)\theta a_w L]^2 \times \Z_w,  \\
     & I_k^\theta := I_0^\theta + k \cdot \theta a_w Le_1, \; k \in \mathbb{Z},
 \end{align*}
and
 \[
 {\cal A}^{\beta,\theta,K,k} := \{A^\beta \in {\cal A}^{\beta}: |A^\beta \cap I_k^\theta | \geq K\},
 \]
 with ${\cal A}^{\beta}$ defined as in Lemma \ref{lemma:approximation}.
 Let $(\hat\zeta_{t}^{\beta,A^\beta,\theta})_{t \geq 0}$ denote a pruning of $(\hat\zeta_{t}^{\beta,A^\beta})_{t \geq 0}$ with particles killed as soon as they exit the box $I_\Delta^\theta$ with $I_\Delta^\theta := [-2\theta a_w L,2\theta a_w L]^2 \times \mathbb{Z}_w$.
 Then, for any $2/3<\theta<1$ and $\varepsilon>0$, there exist $L=L(\theta)>0$, $K=K(\theta,\varepsilon)>0$ and $\beta_0 = \beta_0(\theta,\varepsilon)>0$ such that for any $A^\beta \in {\cal A}^{\beta,\theta,K,0}$ with $|A^\beta|=K$, and any $\beta \leq \beta_0$,
 \begin{align} \label{eq:lemma9result}
\P(\hat{\zeta}_{{L}}^{\beta,A^\beta,\theta} \in {\cal A}^{\beta,\theta,K,k}) \geq 1-\varepsilon, \quad k \in \{-1,1\}.
\end{align}
\end{lemma}

\begin{proof}
Section \ref{app:percolation}.
\end{proof}

    \begin{figure}
    \centering
    \includegraphics[scale=1]{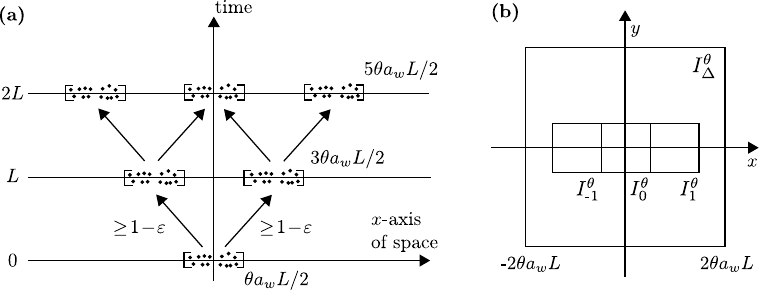}
    \caption{{\bf (a)} Graphical depiction of the percolation construction embedded in Lemma \ref{lemma:percolation}. {\bf (b)} Illustration of the two-dimensional projection of the sets $I_k^\theta$ for $k \in \{-1,0,1\}$ and the set $I_\Delta^\theta$ defined in Lemma \ref{lemma:percolation}.
    }
    \label{fig:percolation}
\end{figure}

\subsubsection{Lower bound result for $\xi_t$}

With the above ingredients, we can establish the following lower bound on the propagation speed of the biased voter model $(\xi_t^0)_{t \geq 0}$ on $\Z^2 \times \Z_w$ conditioned on nonextinction (Lemma \ref{lemma:lowerbound}). The proof is split into three key steps. In the first two steps, we remove the conditioning on nonextinction and introduce duality. In the final step, we use Lemma \ref{lemma:percolation} to define a lower-bounding percolation process, which we then analyze using results from \cite{Dur84} and \cite{liggett2012interacting}. 

\begin{lemma} \label{lemma:lowerbound}
Define $\tau_{\varnothing}^A := \min\{t \geq 0 : \xi_{t}^A = \varnothing \}$ for $A \subseteq \mathbb{Z}^2 \times \Z_w$. For each $2/3<\rho<1$, there exists a constant $L>0$ and a family of random variables $(S_{\beta})_{\beta>0}$, with $\P(S_\beta<\infty|\tau_\varnothing^0=\infty)=1$ for each $\beta>0$, so that
\begin{align*}
\textstyle \lim_{\beta\to 0} 
\liminf_{n\to\infty} \P\big(
\xi_{S_\beta+2nLh(\beta)}^0 \cap \big[2n \rho a_w L h(\beta)^{1/2},\infty\big) e_1 \neq  \varnothing \,\big|\,  \tau_{\varnothing}^0 = \infty \big) = 1,
\end{align*}
where $a_w := p_w \sqrt{\pi w}$. 
\end{lemma}

\begin{proof}
Section \ref{app:lowerboundproof}.
\end{proof}

Note that due to the discretization involved in the proof, Lemma \ref{lemma:lowerbound} only addresses the propagation of $(\xi_t)_{t \geq 0}$ along a subsequence of timepoints. This turns out to be sufficient when combined with the Bramson-Griffeath shape theorem \eqref{eq:shapethm}, as we show next.

\section{Proof of main result}\label{sec_mainproof}

In this section, we complete the proof of our main theorem (Theorem \ref{speed_thm}) by showing how it follows from the lower and upper bound results of Lemmas \ref{lemma:upperbound} and \ref{lemma:lowerbound}, together with the shape theorem \eqref{eq:shapethm}. \\

\begin{pfofthm}{Theorem \ref{speed_thm}}
    Fix $0<\delta<1$ and let $(\beta_i)_{i \geq 1}$ be a sequence of real numbers converging to 0. 
    Apply Lemma \ref{lemma:upperbound} with $\kappa := 1+\delta$ to obtain finite random variables 
    $(S_{\beta_i})_i$ and an integer $i_0$ so that for $i \geq i_0$, 
    \begin{align*}
    \textstyle \liminf_{n\to\infty} \P\big(\lceil(1+\delta) a_wh(\beta_i)^{1/2} n\rceil e_1 \notin \xi_{S_{\beta_i}+h(\beta_i)n}^0\,\big|\,  \tau_{\varnothing}^0 = \infty\big) \geq 1/2.
    \end{align*}
    It follows that for $i \geq i_0$,
    \begin{align}
    \textstyle &\P\big(\lceil(1+\delta) a_wh(\beta_i)^{1/2} n\rceil e_1 \notin \xi_{S_{\beta_i}+h(\beta_i)n}^0 \text{ for infinitely many $n$}\,\big|\,  \tau_{\varnothing}^0 = \infty\big) \geq 1/2.
    \label{eq:upperboundcontradiction}
    \end{align}
       Fix $i \geq i_0$. By the shape theorem \eqref{eq:shapethm}, 
    \begin{align} \label{eq:consofshapetheorem}
    \P\big(\exists t_\ast<\infty: \big\llbracket\!-\!\lfloor (1-\delta)c_w(\beta_{i})t\rfloor,\lfloor(1-\delta)c_w(\beta_{i}) t\rfloor\big\rrbracket e_1 \subseteq \xi_t^0, \; t \geq t_\ast \,\big|\, \tau_\varnothing^0 = \infty\big) = 1,
    \end{align}
    where $c_w(\beta_{i})$ is the radius of the asymptotic shape $D = D(\beta_i)$ as defined in \eqref{correctasymptoticspeed}, and $\llbracket m,n\rrbracket = \{m,m+1,\ldots,n\}$ for integers $m<n$.
    Assume now, by way of contradiction, that
    \[
    c_w(\beta_i) \geq \big((1+2\delta)/(1-2\delta)\big) \cdot a_wh(\beta_i)^{-1/2}.
    \]
    For sufficiently large $n$ (which depends on the outcome $\omega$),
        \begin{align*}
         & \lceil(1+\delta) a_wh(\beta_i)^{1/2} n\rceil/(S_{\beta_i}+h(\beta_i) n) \leq (1+2\delta) a_wh(\beta_i)^{-1/2}, \\
         &     \lfloor (1-\delta) c_w(\beta_{i}) (S_{\beta_i}+h(\beta_i)n)\rfloor/(S_{\beta_i}+h(\beta_i)n) \geq (1-2\delta)c_w(\beta_{i}),
    \end{align*}
    which implies
        \[
    \lceil(1+\delta) a_wh(\beta_i)^{1/2} n\rceil \leq \lfloor (1-\delta) c_w(\beta_{i}) (S_{\beta_i}+h(\beta_i) n)\rfloor.
    \]
    Conditional on $\{\tau_\varnothing^0 = \infty\}$, we must then have $\lceil(1+\delta) a_wh(\beta_i)^{1/2} n\rceil e_1 \in \xi_{S_{\beta_i}+h(\beta_i) n}^0$ for all but finitely many $n$ with probability 1 by \eqref{eq:consofshapetheorem}, which contradicts \eqref{eq:upperboundcontradiction}. We can therefore conclude that
    \[
    c_w(\beta_i) \leq \big((1+2\delta)/(1-2\delta)\big) \cdot a_wh(\beta_i)^{-1/2}, \quad i \geq i_0.
    \]
    Sending $i \to \infty$, 
    and noting that the subsequence $\{\beta_i\}_{i \geq 1}$ is arbitrary, we obtain
    \begin{align*}
      \textstyle \limsup_{\beta \to 0} c_w(\beta)\big/\big({a_wh(\beta)^{-1/2}}\big) \leq (1+2\delta)/(1-2\delta).
    \end{align*}
    Sending $\delta \to 0$ then yields $\limsup_{\beta \to 0} c_w(\beta)/(a_wh(\beta)^{-1/2}) \leq 1$. 
    Applying a similar argument to the lower bound result of Lemma \ref{lemma:lowerbound} will show that
        \begin{align*}
      \textstyle \liminf_{\beta \to 0} c_w(\beta)\big/\big({a_wh(\beta)^{-1/2}}\big) \geq 1,
    \end{align*}
    and we can conclude that $c_w(\beta) \sim a_w h(\beta)^{-1/2}$ as desired.
\end{pfofthm}

\section{Application to cancer initiation and field cancerization} \label{sec:applications}

\begin{figure}[t]
    \centering
    \hspace*{-0.6cm}\includegraphics[scale=1.2]{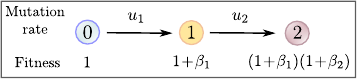}
    \caption{In the two-step mutational model of cancer (Section \ref{sec:applications}), the mutation rates are $u_1$ and $u_2$ for type-0 and type-1 cells, respectively. Type-1 cells have fitness advantage $\beta_1$ over type-0 cells, and type-2 cells have fitness advantage $\beta_2$ over type-1 cells.
    }
    \label{fig:twostepmodel}
\end{figure}

\begin{figure}[b]
    \centering
    \hspace*{-0.6cm}\includegraphics[scale=1]{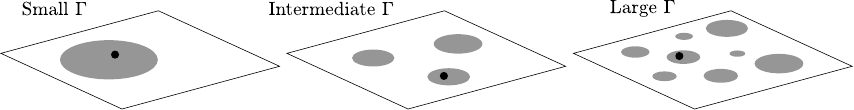}
    \caption{The dynamics of cancer initiation on $\Z_L^2 \times \Z_w$ under the two-step mutational model of Section \ref{sec:applications} are determined by the metaparameter $\Gamma$ as defined in \eqref{eq:gammma}. For small $\Gamma$, the first successful cancer cell (type-2) arises from the first successful premalignant clone (type-1), but as $\Gamma$ increases, cancer can originate from one of several successful type-1 clones (intermediate $\Gamma$), or even an unsuccessful type-1 clone before it goes extinct (large $\Gamma$).}
    \label{fig:regimes}
\end{figure}

We now use our main result to explore the dynamics of cancer initiation and field cancerization under a two-step mutational model of cancer.
In this section, as in \cite{DurMose2015}, \cite{DurFooLed} and \cite{FLR2014}, we assume finite tissue of the form $\Z_L^2 \times \Z_w$, where $L$ is chosen so that the total number of cells in the tissue is $N$, with $N$ typically of order at least $10^6$.
We now impose the same periodic boundary condition along the first two dimensions as along the third dimension.

\begin{figure}
    \centering
    \includegraphics[scale=1]{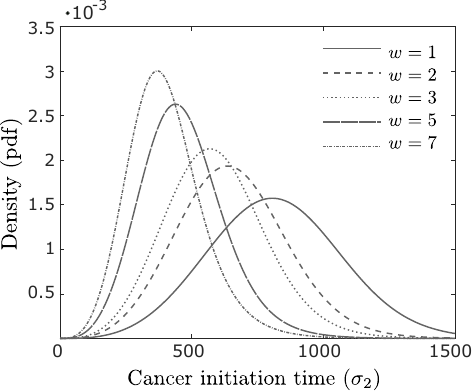}
    \caption{The density of the cancer initiation time $\sigma_2$ for a few values of the tissue thickness $w$. 
    Other parameter values are $N = 10^6$, $u_1 = 10^{-6}$, $u_2 = 10^{-5}$ and $\beta = 0.01$. }
    \label{fig:sigma2}
\end{figure}

Suppose each site in $\Z_L^2 \times \Z_w$ is initially occupied by a normal cell (type-0). Each type-0 cell mutates to a premalignant type-1 cell, with fitness advantage $\beta_1>0$ over normal cells, at exponential rate $u_1$. A type-1 cell gives rise to a successful type-1 clone (one that does not go extinct) with probability $\beta_1/(1+\beta_1)$ by \eqref{eq:gamblersruin}, in which case its long-run expansion rate $c_w(\beta_1)$ is given by our main theorem (Theorem \ref{speed_thm}). Each type-1 cell mutates to a cancer cell (type-2), with fitness advantage $\beta_2>0$ over type-1 cells, at rate $u_2$ (Fig.~\ref{fig:twostepmodel}). As before, a type-2 cell gives rise to a successful clone with probability $\beta_2/(1+\beta_2)$. We let $\sigma_2$ denote the time at which the first successful type-2 cell arises in the population, which we consider the time of cancer initiation. To simplify the following discussion, we assume that $\beta_1 = \beta_2 =: \beta$.

In \cite{FLR2014}, Foo, Leder and Ryser analyze an approximated version of the above model for the $w=1$ case. They assume that cells occupy a spatial continuum, and that type-1 clones grow deterministically with radial growth rate $c_1(\beta)$. Under this simplified model, the dynamics of cancer initiation are governed by the value of the metaparameter
\begin{align} \label{eq:gammma}
\Gamma := N^3 (u_1\beta)^3 c_w(\beta)^{-2} (u_2\beta)^{-1}.
\end{align}
When $\Gamma$ is small, the first successful cancer cell (type-2) typically arises within the first successful type-1 clone.
As $\Gamma$ increases, it becomes possible for cancer to initiate from one of several successful type-1 clones, and when $\Gamma$ is large, it may even arise from an unsuccessful type-1 clone before it goes extinct (Fig.~\ref{fig:regimes}).
A more detailed discussion of these regimes can be found in \cite{DurFooLed}, \cite{FLR2014} and \cite{foo2020mutation}.

Fortunately, the analysis in \cite{FLR2014} carries over to the more general $w>1$ case, with the assumption that type-1 clones grow deterministically as a union of two-dimensional disks (Fig.~\ref{fig:modeldynamics}b) with radial growth rate $c_w(\beta) = p_w \sqrt{\pi w \beta} \big/ \sqrt{\log(1/\beta)}$.
We begin by considering the metaparameter $\Gamma$.
Note first the asymmetric role of the mutation rates $u_1$ and $u_2$: Increasing $u_1$ increases the likelihood that multiple successful type-1 clones arise prior to cancer initiation (large-$\Gamma$ regime), whereas increasing $u_2$ has the reverse effect (small-$\Gamma$ regime). Note also the asymmetric role of $\beta$ and $w$: As $\beta$ increases, $\Gamma$ increases according to $\beta \log(1/\beta)$ for small $\beta$, while as $w$ increases, $\Gamma$ decreases according to $1/(p_w^2w)$. Both parameters affect how quickly type-1 clones expand, but $\beta$ also affects the success probability of type-1 clones by \eqref{eq:gamblersruin}. Thus, whereas a larger $w$ means faster type-1 clonal expansion and a greater chance that cancer initiates within the early clones (small-$\Gamma$ regime), for larger $\beta$, faster type-1 expansion is counterbalanced by the fact that more successful type-1 clones arise, which turns out to push the dynamics toward several successful type-1 clones (large-$\Gamma$ regime).

We next consider the distribution of $\sigma_2$, the time of cancer initiation. Its density is given by (4) in \cite{FLR2014}, with the substitution $\gamma_2 := \pi w$ (area of stacked unit disks in $\Z^2 \times \Z_w$) and $c_w(\beta) = p_w \sqrt{\pi w \beta} \big/ \sqrt{\log(1/\beta)}$. Predictably, as the tissue thickness $w$ increases, faster type-1 expansion translates into earlier cancer initiation (Fig.~\ref{fig:sigma2}). In Figure \ref{fig:speed_illustration}b of Section \ref{sec_mainstatement}, we noted that premalignant population growth is over three times as fast on $w=5$ layers as on $w=1$ layer, whereas cancer initiation speeds up around twofold over this range according to Figure \ref{fig:sigma2}. To see why, note that the probability of the event $\{\sigma_2 \in dt\}$ depends on the ``total mass'' or ``spacetime volume'' of type-1 particles up until time $t$, i.e.~the time-integral of the size of the type-1 population. Under our deterministic growth assumption, a successful type-1 clone that originates at time 0 grows to size $(c_w(\beta))^2s^2\pi w$ by time $s$, and it reaches spacetime volume $V$ by time
\begin{align*} \label{eq:timetospacetimevol}
t_w(V) = 3^{1/3} (p_w \pi w)^{-2/3} h(\beta)^{1/3} V^{1/3}.    
\end{align*}
Thus, going from $w=1$ to $w > 1$ layers should accelerate cancer initiation by around $(p_w w)^{2/3}$, which is consistent with a twofold increase from $w=1$ to $w=5$. Of course, while these calculations give us some idea of what to expect, the dynamics are more complex in general.
For small $u_1$ or small $\beta$, for example, it may take a long time for the first successful type-1 clone to arise, and for larger values, cancer may originate from one of several successful type-1 clones originating at distinct times.

\begin{figure}
    \centering
    \hspace*{-0.6cm}\includegraphics[scale=1]{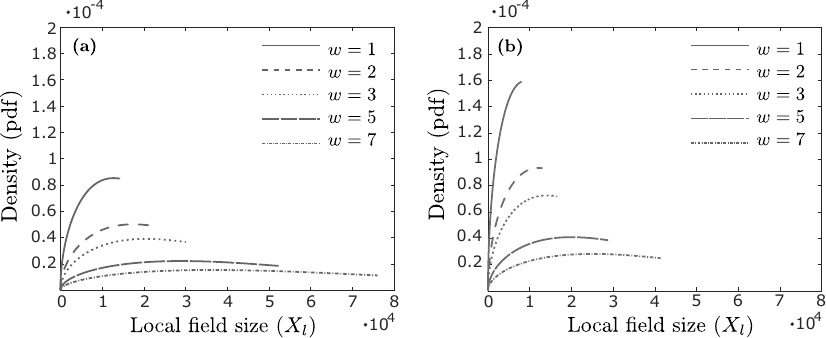}
    \caption{Distribution of the local field size $X_l$, conditioned on $\{\sigma_2 \in dt\}$ with $t = \E[\sigma_2]$, for {\bf (a)} $\beta = 0.01$ and {\bf (b)} $\beta =0.05$. Other parameters are $N = 10^6$, $u_1 = 10^{-6}$ and $u_2 = 10^{-5}$.}
    \label{fig:fieldcancer_beta}
\end{figure}

We finally consider the size $X_l$ of the {\em local field}, i.e.~the premalignant clone from which the first successful cancer cell arises, at the time of cancer initiation. The density of $X_l$, conditioned on the event $\{\sigma_2 \in dt\}$, is given by (8) in \cite{FLR2014}. We focus here on the case $t = \E[\sigma_2]$ when cancer initiates at its expected time. In Figure \ref{fig:fieldcancer_beta}, we show how the conditional distribution of $X_l$ changes with tissue thickness $w$, given fitness advantage $\beta=0.01$ (Fig.~\ref{fig:fieldcancer_beta}a) and $\beta=0.05$ (Fig.~\ref{fig:fieldcancer_beta}b). Since we condition on $\{\sigma_2 \in dt\}$, and type-1 clones are assumed to grow deterministically, the support of $X_l$ is finite and reflects the maximum possible size of a type-1 clone at time $t$, which is $(c_w(\beta))^2t^2 \pi w$.

In Figure \ref{fig:fieldcancer_beta}, we see that as $w$ increases, the local field size $X_l$ increases and varies across a wider range. This reflects the fact that increasing $w$ pushes the dynamics toward the small-$\Gamma$ regime, in which cancer initiates within one or a few large clones. 
When Figures \ref{fig:fieldcancer_beta}a and \ref{fig:fieldcancer_beta}b are compared, we see that increasing $\beta$ results in a smaller local field, and the local field size appears less sensitive to $\beta$ than to $w$. As noted above, increasing $\beta$ both leads to faster expansion of type-1 clones and improved viability of these clones, with counteracting effects on the metaparameter $\Gamma$. 
The fact that $X_l$ decreases with increasing $\beta$ is consistent with $\Gamma$ increasing, moving the dynamics toward a greater number of smaller premalignant fields.

The above discussion reveals how
our main result enables prediction of how
premalignant fields evolve and how they give rise to cancer,
given information on the tissue thickness $w$ and the fitness advantage $\beta$. We have seen how the number of premalignant patches, the time of cancer initiation $\sigma_2$ and the local field size $X_l$ is significantly affected by $w$, and how $w$ and $\beta$ affect the dynamics in distinct ways that would be difficult to anticipate without the aid of a mathematical model. These insights are furthermore clinically relevant, since premalignant fields often appear histologically normal, making them difficult to distinguish from normal tissue. Thus, the capability to make quantitative predictions on the spatial evolutionary history of the tumor can yield valuable insights into e.g.~optimal excision margins under surgery, and into when and where recurrence can be expected to occur following treatment.

\setcounter{lemma}{0}

\section{Proofs of lemmas} \label{sec:proofsoflemmas}

\subsection{Proof of Lemma \ref{lclt}} \label{app:localclt}

\begin{lemma}[LCLT on $\Z^d \times \Z_w$]
  Let $(S_n)_{n \geq 0}$ be the discrete-time SSRW on $\mathbb{Z}^d \times \Z_w$ with $S_0=0$. Set $b_1 :=2$, $b_w := 2$ if $w>1$ is even and $b_w := 1$ if $w>1$ is odd, and define
\begin{align*}
    p_{w,d} := \begin{cases} 1, & w=1, \\ 
    {2d}/({2d+1}), & w=2, \\ {d}/({d+1}), & w>2, \end{cases}
\end{align*}
    as the probability that $(S_n)_{n \geq 0}$ takes a step in the $\mathbb{Z}^d$-direction, 
    with $p_{w,2} = p_w$ as defined in \eqref{eq:birthfirstwodimensions}.
    Then, for any $n \geq 1$ and $x \in \Z^d \times \Z_w$ so that $\P(S_n=x)>0$,
\begin{align*}
    n^{d/2} \P(S_n=x) = (b_w/w) \big(d/(2\pi p_{w,d})\big)^{d/2} + o(1). 
\end{align*}
\end{lemma} 

\begin{proof}
For $w=1$, the SSRW on $\Z^d\times \Z_w$ is equivalent to the SSRW on $\Z^d$, and it suffices to refer to the LCLT on $\Z^d$, given by \eqref{eq:random_walk_Zd} below.
We can therefore assume that $w>1$.

Let $\hat N(n)$ denote the number of steps taken along the first $d$ dimensions of $\Z^d \times \Z_w$
by time $n$. Conditional on $\hat N(n)=k$, we can decompose $S_n$ as $S_n = (\hat S_k, S_{n-k}^w)$, where $(\hat S_j)_{j \geq 0}$ and $(S_j^w)_{j \geq 0}$ are independent SSRWs on $\mathbb{Z}^d$ and $\Z_w$, respectively.
In other words, for any $x \in \mathbb{Z}^d \times \Z_w$ with $x = (\hat x, x_{d+1})$,
\[
\P(S_n=x | \hat N(n)=k) = \P(\hat S_k = \hat x) \P(S_{n-k}^w = x_{d+1}).
\]
Thus, by conditioning on the value of $\hat N(n)$, we can analyze the large-$n$ asymptotics of $S_n$ by analyzing the two random walks $(\hat S_j)_{j \geq 0}$ and $(S_j^w)_{j \geq 0}$ separately. 

We start with the latter walk $(S_j^w)_{j \geq 0}$ on $\mathbb{Z}_w$. Note first that $(S_j^w)_{j \geq 0}$ is aperiodic if $w>1$ is odd, while it is periodic with period 2 if $w>1$ is even. For the $w$ odd case, the transition matrix $P_w$ for $(S_j^w)_{j \geq 0}$ is irreducible and aperiodic, and the uniform distribution on $\Z_w$ is stationary for $P_w$, so by Theorem 4.9 of \cite{levin2017markov}, there exist constants $\gamma_1=\gamma_1(w) \in (0,1)$ and $C_1=C_1(w)>0$ so that
\[
|\P(S_j^w=x_{d+1})-1/w| \leq C_1 \gamma_1^j,
\quad x_{d+1} \in \Z_w.
\]
For the $w$ even case, by rearranging the state space into odds and evens, we can write $P_w^2$ as a block diagonal matrix consisting of two identical irreducible, aperiodic blocks $Q_w$ of dimension $(w/2) \times (w/2)$.
If $S_0^w = 0$ and $x_{d+1}$ is even, we can apply the same theorem as above to get constants $\gamma_2=\gamma_2(w) \in (0,1)$ and $C_2=C_2(w) >0$ so that
\[
|\P(S_{2j}^w=x_{d+1})-2/w| \leq C_2 \gamma_2^{j} = C_2 \big(\gamma_2^{1/2 }\big)^{2j}.
\]
If $x_{d+1}$ is odd, we condition on the first step and obtain as before
\[
|\P(S_{2j+1}^w=x_{d+1})-2/w| \leq C_2 \gamma_2^{j} = \big(C_2\gamma_2^{-1/2}\big) \big(\gamma_2^{1/2}\big)^{2j+1}.
\]
Recall that $b_w = 1$ for $w>1$ odd and $b_w=2$ for $w>1$ even.
Combining the above observations, we see that for the SSRW $(S_j^w)_{j \geq 0}$ on $\mathbb{Z}_w$, there exist $\gamma=\gamma(w) \in (0,1)$ and $C=C(w)>0$ so that for each $x_{d+1} \in \Z_w$ with $\P(S_j^w = x_{d+1})>0$,
\begin{align} \label{eq:formerlemma11}
    |\P(S_j^w = x_{d+1}) - (b_w/w)| \leq C\gamma^j.
\end{align}

We next turn to the random walk $(\hat S_j)_{j \geq 0}$ on $\mathbb{Z}^d$.
For $\hat x \in \mathbb{Z}^d$ and $j \geq 1$, set
\[
p_j(\hat x) := \big(d/(2\pi j)\big)^{d/2} \exp\big(\!-\!d \|\hat x\|^2/(2j)\big),
\]
where $\|\cdot\|$ is the Euclidean norm on $\Z^d$.
By Theorem 2.1.3 of \cite{lawler2010}, there exists $c>0$ so that for all $\hat{x} \in \Z^d$,
\begin{align*}
    |\P(\hat S_j = \hat x)+\P(\hat S_{j+1} = \hat x)-2 p_j(\hat{x})| \leq c/j^{(d+2)/2}.
\end{align*}
Since $(\hat S_j)_{j \geq 0}$ is periodic with period 2, 
we obtain for $j$ and $\hat x$ such that $\P(\hat S_j = \hat x)>0$, 
\begin{align} \label{eq:random_walk_Zd}
j^{d/2} \P(\hat S_j = \hat x) = 2 \big(d/(2\pi)\big)^{d/2} + O(1/j).
\end{align}

We finally establish bounds on $\hat{N}(n)$, the number of steps the SSRW on $\Z^d \times \Z_w$ takes along the first $d$ dimensions by time $n$.
Since $\hat N(n)$ is binomially distributed with success probability $p_{w,d}$, we obtain by Hoeffding's inequality for any $\nu>0$,
\begin{align} \label{eq:hoeffding}
\P(|\hat N(n)-np_{w,d}|>n^{1/2+\nu}) \leq 2 \exp(-2n^{2\nu}).
\end{align}
For each $n$ and each $\nu \in (0,1/2)$, define the neighborhoods
\[
A_n(\nu) := \{1 \leq k \leq n: |k-np_{w,d}| \leq n^{1/2+\nu}\},
\]
and note that by \eqref{eq:hoeffding}, $\hat N(n)$ takes values in $A_n(\nu)$ with high probability for $n$ large. 

We are now ready to carry out the main calculations. Fix $\nu \in (0,1/2)$ and note first that
\begin{align*}
n^{d/2} \P(S_n=x) &= n^{d/2} \P(S_n=x | \hat N(n) \in A_n(\nu)) \P(\hat N(n) \in A_n(\nu)) \\
&\quad + n^{d/2} \P(S_n=x | \hat N(n) \notin A_n(\nu)) \P(\hat N(n) \notin A_n(\nu)).
\end{align*}
Write $x = (\hat x, x_{d+1})$. Since $\P(\hat N(n) \notin A_n(\nu)) \leq 2\exp(-2n^{2\nu})$ by \eqref{eq:hoeffding}, and $n^{d/2} = o(\exp(2n^{2\nu}))$, it suffices to study the large-$n$ asymptotics of
\begin{align*}
& n^{d/2} \P(S_n=x | \hat N(n) \in A_n(\nu)) \\
&= \textstyle n^{d/2} \sum_{k \in A_n(\nu)} \P(\hat S_k = \hat x) \P(S_{{n-k}}^w = x_{d+1}) \P(\hat N(n) = k | \hat N(n) \in A_n(\nu)).
\end{align*}
For $n$ and $x$ so that $\P(S_n=x)>0$, we can assume without loss of generality that $\P(\hat S_k=\hat x)>0$ and $\P(S_{n-k}^w=x_{d+1})>0$ for $k \in A_n(\nu)$ even and $n$ sufficiently large. Define
\[
A_n^2(\nu) := \{k \in A_n(\nu): k \text{ even}\}.
\]
For $k \in A_n^2(\nu)$, we can write $k^{d/2} \P(\hat S_k = \hat x) = 2(d/(2\pi))^{d/2} + O(1/k)$ by \eqref{eq:random_walk_Zd}.
By the definition of $A_n(\nu)$, we furthermore have
\[
(1/p_{w,d}) \cdot \big(1/(1+p_{w,d}^{-1}n^{\nu-1/2})\big) \leq n/k \leq (1/p_{w,d}) \cdot \big(1/(1-p_{w,d}^{-1}n^{\nu-1/2})\big),
\]
which implies that $n/k = 1/p_{w,d} + O(n^{\nu-1/2})$.
Therefore,
\begin{align*}
& n^{d/2} \P(S_n=x | \hat N(n) \in A_n(\nu)) \\
&= \textstyle \sum_{k \in A_n^2(\nu)} (n/k)^{d/2} \cdot k^{d/2} \P(\hat S_k = \hat x) \cdot \P(S_{{n-k}}^w = x_{d+1}) \P(\hat N(n) = k | \hat N(n) \in A_n(\nu)) \\
&= 2(d/(2\pi p_{w,d}))^{d/2} \textstyle \sum_{k \in A_n^2(\nu)} \P(S_{{n-k}}^w = x_{d+1}) \P(\hat N(n) = k | \hat N(n) \in A_n(\nu)) + o(1).
\end{align*}
By \eqref{eq:formerlemma11}, there exist constants $\gamma=\gamma(w) \in (0,1)$ and $C=C(w)>0$ so that for $k \in A_n^2(\nu)$,
\begin{align*}
|\P(S_{{n-k}}^w = x_{d+1})-(b_w/w)| \leq C \gamma^{n(1-p_{w,d})-n^{1/2+\nu}},
\end{align*}
which implies that $\P(S_{{n-k}}^w = x_{d+1}) = b_w/w + O(\gamma^{n(1-p_{w,d})-n^{1/2+\nu}})$.
We thus obtain
\begin{align*}
& n^{d/2} \P(S_n=x | \hat N(n) \in A_n(\nu)) \\
&= \textstyle 2(b_w/w) (d/(2\pi p_{w,d}))^{d/2} \sum_{k \in A_n^2(\nu)} \P(\hat N(n)=k | \hat N(n) \in A_n(\nu)) + o(1).
\end{align*}
The remaining sum is the probability that $\hat N(n)$ is even given that $\hat N(n) \in A_n(\nu)$. Since for $X \sim {\rm Bin}(n,p)$, $\P(X \text{ is even}) = 1/2 + (1/2)  (1-2p)^n$, and the probability that $\hat N(n) \in A_n(\nu)$ converges to 1 as $n \to \infty$, the sum converges to $1/2$ as $n \to \infty$. The result follows.
\end{proof}

\subsection{Proof of Lemma \ref{ret_time_thm}} \label{app:tailasymptotics}

\begin{lemma}[Asymptotic tail of $T_0$]
Let $(Z_t)_{t \geq 0}$ be the SSRW on $\Z^2\times \Z_w$ with jump rate $\alpha> 0$, started at a nearest neighbor of the origin. Set $T_0:=\inf\{t \geq 0:Z_t=0\}$ and define
\begin{align*} 
\mu_w := p_w \pi w =
\begin{cases} \pi, & w=1, \\ 
(4/5){\pi w}, & w=2, \\ (2/3){\pi w}, & w>2, \end{cases}    
\end{align*}
where $p_w$ is the probability given by \eqref{eq:birthfirstwodimensions}. Then
\begin{align*} 
\P(T_0 > t)  \sim \mu_w/\log t, \quad t \to \infty.
\end{align*}
\end{lemma}

\begin{proof}
We begin by proving the result for the embedded discrete-time SSRW $(S_n)_{n \geq 0}$ on $\Z^2 \times \Z_w$, defined by $S_n := Z_{\sigma_n}$, where $\sigma_n$ is the time of the $n$-th jump of $(Z_t)_{t \geq 0}$.
Set $\tau_0 := \min\{n > 0 : S_n = 0\}$.
We want to show that
\begin{align*}
\P(\tau_0 > n) \sim \mu_w/\log n, \quad n \to \infty.
\end{align*}
The $w=1$ case has already been established by 
Dvoretzky and Erd{\"o}s in \cite{dvor51}.
We can easily extend their argument to the $w>1$ case by substituting the LCLT on $\Z^2$ by our LCLT on $\Z^2 \times \Z_w$.
We sketch the argument briefly below.

Note first that instead of assuming that $(S_n)_{n \geq 0}$ is started at a nearest neighbor of the origin, we can assume that $(S_n)_{n \geq 0}$ is started at the origin itself, since for $n \geq 1$, 
\begin{align*}
    \P(\tau_0>n|S_0=0) = \P(\tau_0>n-1|S_0=e_1).
\end{align*}
Define $\gamma_2(n) := \P(\tau_0 > n-1|S_0=0)$ and $u_2(n) := \P(S_n=0|S_0=0)$ using the notation of \cite{dvor51}.
Clearly, $\gamma_2(n)$ in decreasing in $n$ with $\gamma_2(1) = 1$ and $\gamma_2(n)>0$ for all $n \geq 1$.
Assume that $w$ is even, in which case $(S_n)_{n \geq0}$ is periodic with period 2.
Then $u_2(2n-1) = 0$ for $n \geq 1$, and by the discrete-time LCLT of Lemma \ref{lclt},
\begin{align} \label{eq:dvoretzkyestimate}
    \textstyle u_2(2n) = \P(S_{2n}=0|S_0=0) 
    = (1/\mu_w)(1/n)+o(1/n), \quad n \geq 1.
\end{align}
If $S_0=0$, then for any $n \geq 1$, the walk visits the origin at least once by time $n-1$ with probability 1. Decomposing this event in terms of the last return to the origin by time $n-1$, and setting $m := n/2-1$ for even $n$ and $m := (n-1)/2$ for odd $n$, we can write
\begin{align*}
u_2(0)\gamma_2(n) + u_2(2) \gamma_2(n-2) + \cdots + u_2(2m) \gamma_2(n-2m)=1.
\end{align*}
It follows from \eqref{eq:dvoretzkyestimate} that $\sum_{i=0}^{k} u_2(2i) = (1/\mu_w)(\log k) (1+o(1))$.
Using the monotonicity of $\gamma_2$, one can now show that
$\gamma_2(n) = (\mu_w+o(1))/\log n$, see page 356 of \cite{dvor51}.
The argument for odd $w$, in which case the walk is aperiodic, follows along similar lines.

It remains to translate the above discrete-time result into continuous time.
Fix $\varepsilon>0$.
Let $T_0 = \inf\{t \geq 0 : Z_t = 0\}$ be the time of first visit of $(Z_t)_{t \geq 0}$ to the origin, and let $N(t)$ denote the number of jumps $Z_t$ makes by time $t$. Then
\begin{align*}
\P(T_0 > t) & =  \P(\tau_0 > N(t)) 
\leq \P\left(\tau_0 > \alpha t(1-\varepsilon) \right) + \P\left(N(t) < \alpha t(1-\varepsilon) \right),
\end{align*}
and
\begin{align*}
    \P(T_0>t) &= \P(\tau_0 > N(t)) \geq \P\left(\tau_0 > \alpha t(1+\varepsilon) \right) - \P\left( N(t) >  \alpha t(1+\varepsilon) \right).
\end{align*}
Since $(N(t))_{t \geq 0}$ is a Poisson process with rate $\alpha$, we have $N(t)/(\alpha t) \to 1$ almost surely as $t\to\infty$.
The result follows.
\end{proof}

\subsection{Proof of Lemma \ref{lemma:branchingrate}} \label{app:branchingratepruneddual}

\begin{lemma} 
Let $(Z_t^1)_{t \geq 0}$ and $(Z_t^2)_{t \geq 0}$ be independent SSRWs on $\mathbb{Z}^2 \times \Z_w$ with jump rate 1, started at 0 and a nearest neighbor of 0. Set $T_0 := \inf\{t \geq 0: Z_t^1 = Z_t^2\}$, and let $S$ be an exponential random variable with mean $1/\beta$, independent of $(Z_t^1)_{t \geq 0}$ and $(Z_t^2)_{t \geq 0}$.
Then
\begin{enumerate}[(1)]
    \item $\alpha_0(\beta) := \P(T_0 \leq \min\{S,\tau(\beta)\}) \to 1$ as $\beta \to 0$.
    \item $\alpha_1(\beta) :=\P(S \leq \min\{T_0,\tau(\beta)\}) = \Theta\big(1/(\log(1/\beta))^{3/2}\big)$ as $\beta \to 0$,
    \item $\alpha_2(\beta) :=\P(\tau(\beta) \leq \min\{S,T_0\}) \sim \mu_w/\log(1/\beta)$ as $\beta \to 0$.
\end{enumerate}
Furthermore, there exists $C>0$ so that for sufficiently small $\beta$,
\begin{enumerate}[(1)]
    \item[(4)] $\E\big[\sup_{t \leq \tau(\beta)} ||Z_t^1||^j \,\big|\, T_0 \leq \min\{S,\tau(\beta)\}\big] \leq C j! \tau(\beta)^{j/2}$, \quad $j \geq 1$,
    \item[(5)] $\E\big[\sup_{t \leq \tau(\beta)} ||Z_t^1||^j \,\big|\, S \leq \min\{T_0,\tau(\beta)\}\big] \leq C j! (\log(1/\beta))^{1/2}\tau(\beta)^{j/2}$, \quad $j \geq 1$,
    \item[(6)] $\E\big[\sup_{t \leq \tau(\beta)} ||Z_{t}^1||^j \,\big|\, \tau(\beta) \leq \min\{S,T_0\}\big] \leq C j! (\log(1/\beta))^{1/2} \tau(\beta)^{j/2}$, \quad  $j \geq 1$.
\end{enumerate}
Each of (4)-(6) continues to hold if $Z_t^1$ is replaced by $Z_t^2$ or  $\bar{Z_t} = Z_t^1-Z_t^2$.
\end{lemma}

\begin{proof}
Recall that $\tau(\beta) = (1/\beta) \big(1/\sqrt{\log(1/\beta)}\big)$, and define
\begin{align*} 
    a(\beta):= (1/\beta) \big(1/(\log(1/\beta))^{3/2}\big) = o(\tau(\beta)).
\end{align*}
\begin{enumerate}[(1)]
\item Follows from (2) and (3).
\item Since $S$ is independent of $T_0$, we can write
\begin{align*}
\P(S\leq \min\{T_0,\tau(\beta)\}) &=\textstyle \int_0^{\tau(\beta)}\P(s\leq T_0) \cdot \beta e^{-\beta s}ds\\
&=
\textstyle \int_0^{a(\beta)} \P(s \leq T_0) \cdot \beta e^{-\beta s}ds +\int_{a(\beta)}^{\tau(\beta)}\P(s\leq T_0) \cdot \beta e^{-\beta s}ds.
\end{align*}
For the former integral, it is easy to see that
\begin{align*}
    \textstyle \int_0^{a(\beta)} \P(s \leq T_0) \cdot \beta e^{-\beta s}ds &\leq \textstyle\int_0^{a(\beta)} \beta e^{-\beta s}ds \\
    &= 1-e^{-\beta a(\beta)} \sim 1/(\log(1/\beta))^{3/2}, \quad \beta \to 0.
\end{align*}
For the latter integral, we can write
\begin{align}
    &\textstyle \int_{a(\beta)}^{\tau(\beta)}\P(s\leq T_0) \cdot \beta e^{-\beta s}ds 
    \leq (\tau(\beta)-a(\beta)) \cdot \beta e^{-\beta a(\beta)} \cdot \P(a(\beta) \leq T_0).
    \label{eq:log3/2}
\end{align}
Now, $(\tau(\beta)-a(\beta)) \sim \tau(\beta) = (1/\beta) \big(1/\sqrt{\log(1/\beta)}\big)$ as $\beta \to 0$, $e^{-\beta a(\beta)} \to 1$ as $\beta \to 0$, and by Lemma \ref{ret_time_thm}, $\P(a(\beta) \leq T_0) \sim \mu_w/\log(1/\beta)$ as $\beta \to 0$.
The right-hand side of \eqref{eq:log3/2} is therefore of order $1/(\log(1/\beta))^{3/2}$ as $\beta \to 0$. Thus,
\[
\P\big(S \leq \min\{T_0,\tau(\beta)\}\big) = O\big(1/(\log(1/\beta))^{3/2}\big).
\]
On the other hand, by independence,
\begin{align*}
\P(S \leq \min\{T_0, \tau(\beta)\}) &\geq \P(S \leq \tau(\beta) \leq T_0) \\
&= \P(S \leq \tau(\beta)) \cdot \P(\tau(\beta) \leq T_0) \\
&\sim \mu_w/(\log(1/\beta))^{3/2}, \quad \beta \to 0,
\end{align*}
since $\P(S \leq \tau(\beta)) = 1-e^{-\beta \tau(\beta)} \sim 1/\sqrt{\log(1/\beta)}$ as $\beta \to 0$, and $\P(\tau(\beta) \leq T_0) \sim \mu_w/\log(1/\beta)$ as $\beta \to 0$ by Lemma \ref{ret_time_thm}. Therefore,
\[
\P(S \leq \min\{T_0,\tau(\beta)\}) = \Omega\big(1/(\log(1/\beta))^{3/2}\big).
\]
Since $\P(S \leq \min\{T_0,\tau(\beta)\})$ is both $O\big(1/(\log(1/\beta))^{3/2}\big)$ and $\Omega\big(1/(\log(1/\beta))^{3/2}\big)$ as $\beta \to 0$, we have established part (2).
\item By independence,
\begin{align*}
    \P(\tau(\beta) \leq \min\{S,T_0\}) &= \P(\tau(\beta) \leq S) \cdot \P(\tau(\beta) \leq T_0) \\
    &= e^{-\beta \tau(\beta)} \cdot \P(\tau(\beta) \leq T_0) \\
    &\sim \mu_w/\log(1/\beta), \quad \beta \to 0,
\end{align*}
which follows from the fact that $\exp(-\beta \tau(\beta)) \to 1$ as $\beta \to 0$, and that $\P(\tau(\beta) \leq T_0) \sim \mu_w/\log(1/\beta)$ as $\beta \to 0$ by Lemma \ref{ret_time_thm}.
\item  Note that we can write $Z_t^1 = (Z_{1,t}^1,Z_{2,t}^1,Z_{3,t}^1)$, where $Z_{1,t}^1$ and $Z_{2,t}^1$ are SSRWs on $\mathbb{Z}$ with jump rate $p_w/2$ each, and $Z_{3,t}^1$ is the SSRW on $\Z_w$ with jump rate $1-p_w$, where $p_w$ is defined as in \eqref{eq:birthfirstwodimensions}.
All walks are started at 0.
By part (1) above, we have $\P(T_0 \leq \min\{S,\tau(\beta)\}) \geq 1/2$ for sufficiently small $\beta$, which implies that for any $x>0$,
\begin{align} \label{eq:removeconditionongprob1}
    &\textstyle \P\big(\tau(\beta)^{-1/2} \sup_{t \leq \tau(\beta)}||Z_{t}^1||>x \,\big|\, T_0 \leq \min\{S,\tau(\beta)\}\big) \nonumber \\
    &\leq \textstyle 2\,\P\big(\tau(\beta)^{-1/2} \sup_{t \leq \tau(\beta)}||Z_{t}^1||>x\big).
\end{align}
Then note that for any $x>0$,
\begin{align} \label{eq:lemma6intermediate22}
& \textstyle \P\big(\tau(\beta)^{-1/2} \sup_{t \leq \tau(\beta)}||Z_{t}^1||>x\big) \leq \textstyle 4 \, \P\big(\tau(\beta)^{-1/2} \sup_{t \leq \tau(\beta)}Z_{1,t}^1>x/2\big),
\end{align}
since $Z_{1,t}^1$ and $Z_{2,t}^1$ have the same distribution.
Now, $Z_{1,t}^1$ takes steps $\pm 1$ with equal probability at rate $p_w/2$. The steps have moment generating function
    \[
    \phi(\theta) = ({e^{\theta}+e^{-\theta}})/2,
    \]
    so $Z_{1,t}^1$ has moment generating function
    \begin{align} \label{eq:momgenrwonedim}
     \psi_t(\theta) &= \E\big[\exp(\theta Z_{1,t}^1)\big] \nonumber \\
     &= 
     \textstyle\sum_{n=0}^\infty e^{-(p_w/2)t}  \big(((p_w/2) t)^n/{n!}\big)  \cdot \phi(\theta)^n \nonumber \\
     &= \exp\big((p_w/2)t \cdot (\phi(\theta)-1)\big). 
        \end{align}
        For any $x>0$ and $\theta>0$, we thus obtain by Doob's inequality,
        \begin{align*} \label{eq:lemma6intermediate33}
            &\textstyle \P(\tau(\beta)^{-1/2} \sup_{t \leq \tau(\beta)} Z_{1,t}^1>x/2) \nonumber \\
            &= \textstyle \P\big(\sup_{t \leq \tau(\beta)} \exp(\theta \tau(\beta)^{-1/2} Z_{1,t}^1) > \exp(\theta x/2)\big) \nonumber \\
            &\leq \exp\!\big((p_w/2) \tau(\beta) \cdot (\phi(\theta \tau(\beta)^{-1/2})-1)\big)  \cdot \exp(-\theta x/2).
        \end{align*}
        Since $\phi(\theta\tau(\beta)^{-1/2}) = 1 + (1/2)\theta^2\tau(\beta)^{-1} + \cdots$, we can find $C_1=C_1(\theta)>0$ so that for sufficiently small $\beta$,
        \[
        (p_w/2)\tau(\beta) \cdot (\phi(\theta\tau(\beta)^{-1/2})-1)  \leq \log C_1.
        \]
        If we take $\theta=2$, we then obtain for sufficiently small $\beta$,
        \begin{align} \label{eq:lemma6intermediate44}
            &\textstyle \P(\tau(\beta)^{-1/2} \sup_{t \leq \tau(\beta)} Z_{1,t}^1>x/2) \leq C_1 \exp(-x).
        \end{align}
       Now fix $j \geq 1$ and write, by \eqref{eq:removeconditionongprob1}, \eqref{eq:lemma6intermediate22} and \eqref{eq:lemma6intermediate44},
        \begin{align*}
            &\textstyle\tau(\beta)^{-j/2} \E\big[\sup_{t \leq \tau(\beta)} ||Z_{t}^1||^j\,\big|\, T_0 \leq \min\{S,\tau(\beta)\}\big] \\
            &= \textstyle\int_0^\infty \P\big(\tau(\beta)^{-1/2}\sup_{t \leq \tau(\beta)} ||Z_{t}^1||>x^{1/j} \,\big|\, T_0 \leq \min\{S,\tau(\beta)\}\big) dx \\
             &\leq C_2 \textstyle \int_0^\infty \exp\big(\!-x^{1/j}\big) dx,
        \end{align*}
        where $C_2 := 8C_1>0$.
        Using the substitution $u= x^{1/j}$, $x = u^j$, $dx = j u^{j-1} du$, we obtain
        \begin{align*}
            &\textstyle \tau(\beta)^{-j/2} \E\big[\sup_{t \leq \tau(\beta)} ||Z_{t}^1||^j\,\big|\, T_0 \leq \min\{S,\tau(\beta)\}\big]\nonumber   \leq \textstyle C_2 j \int_0^\infty e^{-u} u^{j-1} du = C_2 j!,
        \end{align*}
        since $\int_0^\infty e^{-u} u^{j-1} du = (j-1)!$. The result follows.
\item We begin by noting that for any $x>0$, by independence and Cauchy-Schwarz,
    \begin{align*}
        &\textstyle\P\big(\tau(\beta)^{-1/2} \sup_{t \leq \tau(\beta)} ||Z_{t}^1||>x, S \leq \min\{T_0,\tau(\beta)\}\big) \nonumber \\
        &=\textstyle \int_0^{\tau(\beta)} \P\big(\tau(\beta)^{-1/2} \sup_{t \leq \tau(\beta)}||Z_{t}^1||>x, s \leq T_0\big) \cdot \beta e^{-\beta s}ds \nonumber \\
        &\leq \textstyle \sqrt{\P\big(\tau(\beta)^{-1/2} \sup_{t \leq \tau(\beta)}||Z_{t}^1||>x\big)} \cdot \int_0^{\tau(\beta)} \sqrt{\P(s \leq T_0)} \cdot \beta e^{-\beta s} ds.
    \end{align*}
    By the same analysis as in part (2) above, the integral is $O\big(1/\log(1/\beta)\big)$. Since
    \[
    \P(S \leq \min\{T_0,\tau(\beta)\}) = \Theta\big(1/(\log(1/\beta))^{3/2}\big)
    \]
    by part (2), we obtain for some $C_3>0$ and sufficiently small $\beta$,
        \begin{align*}
        &\textstyle\P\big(\tau(\beta)^{-1/2} \sup_{t \leq \tau(\beta)}||Z_{t}^1||>x \,\big|\, S \leq \min\{T_0,\tau(\beta)\}\big) \nonumber \\
        &\leq \textstyle C_3 (\log(1/\beta))^{1/2} \sqrt{\P\big(\tau(\beta)^{-1/2} \sup_{t \leq \tau(\beta)}||Z_{t}^1||>x\big)}.
    \end{align*}
        The same argument as in part (4) above now yields the desired result. The only modification is that we take $\theta=4$ instead of $\theta=2$ in the calculations due to the square root.
\item We begin by writing for any $x>0$, by Cauchy-Schwarz,
        \begin{align*} 
        &\textstyle\P\big(\tau(\beta)^{-1/2} \sup_{t \leq \tau(\beta)}||Z_{t}^1||>x, \tau(\beta) \leq \min\{S,T_0\}\big) \nonumber \\
        &\leq \textstyle  \sqrt{\P\big(\tau(\beta)^{-1/2} \sup_{t \leq \tau(\beta)}||Z_{t}^1||>x\big)} \cdot \sqrt{\P(\tau(\beta) \leq \min\{S,T_0\})},
    \end{align*}
    and the same argument as in part (5) yields the desired result, the only difference being that we appeal to the result of part (3) instead of part (2). \qedhere
\end{enumerate}
\end{proof}

\subsection{Proof of Lemma \ref{lemma:approximateupperbound}} \label{app:pruneddualapproxdual}

\begin{lemma}
Fix $a>0$ and $0 < \rho < 1$, and define $A_r := [r,\infty) \times \R \times \Z_w$.
For each $\delta>0$, there exist $M>0$ and $\beta_0>0$ so that
\begin{align*}
     &\textstyle \P\big(\tilde\zeta_{t}^{\beta,0} \cap A_{b+at} \neq \varnothing\big) \leq 4 \, \P\big(\mathring\phi_{t}^{\beta,0} \cap A_{b+\rho at} \neq \varnothing \big) + M e^{-\delta t}, \quad \beta \leq \beta_0, \; t>0, \; b \in \R.
\end{align*}
\end{lemma}

\begin{proof}
To avoid notational overload, we suppress the initial conditions of the processes used in the proof, and assume that each is started with a single particle at the origin.

As in Section \ref{sec:unmodifiedbrw}, let $\mathring\phi_t^{(k)}$ for $k=0,1,2$ be the subprocess of the unaltered BRW $\bar\phi_t$ obtained by gathering offsprings of particles in the pruned BRW $\mathring\phi_t$ that have just been born through a type-$k$ branching event and whose decision period has not yet passed. Recall that by \eqref{eq:dualprunedbrwupper},
\begin{align*} 
\tilde\zeta_t \subseteq \mathring\phi_t \cup \mathring\phi_t^{(0)} \cup \mathring\phi_t^{(1)} \cup \mathring\phi_t^{(2)}.
\end{align*}
Since $0 < \rho < 1$, it follows that
\begin{align} \label{eq:originalstatement}
    \textstyle \P\big(\tilde\zeta_t^\beta \cap A_{b+at} \neq \varnothing\big) 
    &\leq \textstyle\P\big(\mathring\phi_t^\beta \cap A_{b+\rho at} \neq \varnothing\big) + \sum_{k=0}^2 \P\big(\mathring\phi_t^{(k),\beta} \cap A_{b+at} \neq \varnothing\big).
\end{align}
To analyze the terms in the sum, consider $\P\big(\mathring\phi_t^{(k),\beta} \cap A_{b+at} \neq \varnothing\big)$ for $k=1$. We begin by writing
\begin{align}
&\P\big(\mathring\phi_t^{(1),\beta} \cap A_{b+at} \neq \varnothing\big) \nonumber \\
&\leq \P\big(\mathring\phi_t^{\beta} \cap A_{b+\rho at} \neq \varnothing\big) + \P\big(\mathring\phi_t^{(1),\beta} \cap A_{b+at} \neq \varnothing, \mathring\phi_t^{\beta} \cap A_{b+\rho at} = \varnothing\big). \label{eq:lemma7intermediate1}
\end{align}
Enumerate the particles in $\mathring \phi_t^{(1),\beta}$ as $X^{\beta,1}_t,X^{\beta,2}_t,\ldots,X_{t}^{\beta,N_t^{(1),\beta}}$, where $N_t^{(1),\beta}$ is the number of particles in $\mathring \phi_t^{(1),\beta}$.
For each $X_t^{\beta,i}$, let $Y_t^{\beta,i}$ denote the position of its parent particle
in $\mathring \phi_t^\beta$.
Then
\begin{align} \label{eq:sumofparticles}
 \textstyle &\P\big(\mathring\phi_t^{(1),\beta} \cap A_{b+at} \neq \varnothing, \mathring\phi_t^{\beta} \cap A_{b+\rho at} = \varnothing\big) &\textstyle \nonumber \\
 &\leq \textstyle \P\big(\bigcup_{i=1}^{N_t^{(1),\beta}} \{||X_t^{\beta,i}-Y_t^{\beta,i}|| > (1-\rho)at\} \big)  \nonumber \\
 &\leq \textstyle\sum_{i=1}^\infty \P\big(||X_t^{\beta,i}-Y_t^{\beta,i}|| > (1-\rho)at \,\big|\, i \leq N_t^{(1),\beta}\big) \P\big(i \leq N_t^{(1),\beta}\big).
\end{align}
Let $(Z_s^1)_{s \geq 0}$ and $(Z_s^2)_{s \geq 0}$ be independent SSRWs on $\Z^2 \times \Z_w$, started at 0 and a randomly chosen neighbor of 0, let $T_0$ be the time at which they first meet, and let $S$ be an exponential random variable with mean $1/\beta$, independent of $(Z_s^1)_{s \geq 0}$ and $(Z_s^2)_{s \geq 0}$.
Since each particle in $\mathring\phi_t^{(1)}$ has existed for at most $\tau(\beta)$ time units, and each particle is conditioned on producing its own offspring quickly,
we can write for any $i \geq 1$,
\begin{align} \label{eq:distributionstatement}
    &\textstyle \P\big(||X_t^{\beta,i}-Y_t^{\beta,i}|| > (1-\rho)at \,\big|\, i \leq N_t^{(1),\beta}\big) \nonumber \\
    &\leq \textstyle \P\big(h(\beta)^{-1/2} \sup_{s \leq \tau(\beta)} \|Z_s^1-Z_s^2\|> (1-\rho)at\,\big|\,  S \leq \min\{T_0,\tau(\beta)\}\big).
\end{align}
For ease of notation, let $(\tilde Z_s)_{s \geq 0}$ denote a walk on $\Z^2 \times \Z_w$ with the distribution of $(Z_s^1-Z_s^2)_{s \geq 0}$ conditional on $\{S \leq \min\{T_0,\tau(\beta)\}\}$.
Write $\tilde Z_s = (\tilde Z_{1,s}, \tilde Z_{2,s}, \tilde Z_{3,s})$, and 
note that
\begin{align} \label{eq:4upperbound}
    &\textstyle \P\big(h(\beta)^{-1/2} \sup_{s \leq \tau(\beta)} \|\tilde Z_s\|> (1-\rho)at\big) \nonumber \\
    &\leq \textstyle 4 \, \P\big(h(\beta)^{-1/2} \sup_{s \leq \tau(\beta)} \tilde Z_{1,s} > (1/2)(1-\rho)at\big).
\end{align}
We now consider the moment generating function
\begin{align*} 
\textstyle \E\big[\exp\big(\nu h(\beta)^{-1/2} \tilde Z_{1,\tau(\beta)} \big) \big].
\end{align*}
By part (5) in Lemma \ref{lemma:branchingrate}, there exists $C>0$ so that for sufficiently small $\beta$,
\begin{align*}
    \textstyle \E|\tilde Z_{1,\tau(\beta)}|^j \leq C j! (\log(1/\beta))^{1/2} \tau(\beta)^{j/2}, \quad j \geq 1.
\end{align*}
We therefore obtain, with $\tau(\beta) = (1/\beta) \big(1/\sqrt{\log(1/\beta)}\big)$ and $h(\beta) = (1/\beta) \cdot \log(1/\beta)$,
\begin{align*}
    h(\beta)^{-j} \E|\tilde Z_{1,\tau(\beta)}|^{2j} \leq C (2j)!(\log(1/\beta))^{(1-3j)/2} \leq C (2j)! (\log(1/\beta))^{-j}, \quad j \geq 1,
\end{align*}
where we use that $(1-3j)/2 \leq -j$ whenever $j \geq 1$. Using that $\tilde Z_{1,\tau(\beta)} \stackrel{d}{=} -\tilde Z_{1,\tau(\beta)}$ by symmetry, we obtain for sufficiently small $\beta$,
\begin{align} \label{eq:mgfperturbation}
    \textstyle \E\big[\exp\big(\nu h(\beta)^{-1/2} \tilde Z_{1,\tau(\beta)}\big)\big] &= \textstyle \sum_{j=0}^{\infty} \big(1/(2j)!\big) \cdot \nu^{2j} h(\beta)^{-j} \E|\tilde Z_{1,\tau(\beta)}|^{2j} \nonumber \\
    &\leq 1+\textstyle C\sum_{j=1}^\infty \big(1/(2j)!\big) \cdot \nu^{2j} (2j)! (\log(1/\beta))^{-j} \nonumber \\
    &= 1 + O_{C,\nu}\big((\log(1/\beta))^{-1}\big), 
\end{align}
where $O_{C,\nu}$ signifies that the error term depends on $C$ and $\nu$. 
By Doob's inequality,
\begin{align*}
&\textstyle \P\big(h(\beta)^{-1/2} \sup_{s \leq \tau(\beta)} \tilde Z_{1,s} > (1/2)(1-\rho)at\big) \\
&\leq  \exp\big(\!-\!(1/2)(1-\rho) \nu a t\big) \cdot \E\big[\exp\big(\nu h(\beta)^{-1/2} \tilde Z_{1,\tau(\beta)}\big)\big], \quad t > 0.
\end{align*}
By first choosing $\nu$ sufficiently large and then $\beta$ sufficiently small, we can for any $\delta>0$ find a constant $M_1>0$ so that for sufficiently small $\beta$,
\begin{align*}
&\textstyle \P\big(h(\beta)^{-1/2} \sup_{s \leq \tau(\beta)} \tilde Z_{1,s} > (1/2)(1-\rho)at\big) \leq M_1 e^{-(\delta+2\mu_w) t}, \quad t > 0.
\end{align*}
It then follows from \eqref{eq:sumofparticles}, \eqref{eq:distributionstatement} and \eqref{eq:4upperbound} that for $M_2 := 4M_1$ and sufficiently small $\beta$,
\begin{align*}
    \P\big(\mathring\phi_t^{(1),\beta} \cap A_{b+at} \neq \varnothing, \mathring\phi_t^{\beta} \cap A_{b+\rho at} = \varnothing\big) &\textstyle\leq M_2 e^{-(\delta+2\mu_w)t} \cdot \sum_{i=1}^\infty \P\big(i \leq N_t^{(1),\beta}\big) \\
    &= M_2 e^{-(\delta+2\mu_w)t} \cdot \E\Big[N_t^{(1),\beta}\Big], \quad t>0.
\end{align*}
Now, $\E\Big[N_t^{(1),\beta}\Big] = \E\big|\mathring\phi_t^{(1),\beta}\big| \leq \E\big|\mathring\phi_t^{\beta}\big| \leq \E\big|\phi_t^\beta\big|$, where $\phi_t^\beta$ denotes the simple BRW defined in Section \ref{sec:markovianupper}.
By Lemma \ref{lemma:branchingrate}, $\phi_t^\beta$ branches at exponential rate $\gamma(\beta) = \mu_w+o(1)$, with mean number of new particles $\ell(\beta) = 1+o(1)$ introduced per branching event.
Since $\gamma(\beta) \cdot \ell(\beta) \leq 2\mu_w$ for sufficiently small $\beta$, we can write 
\[
\E\Big[N_t^{(1),\beta}\Big] \leq \E\big|\phi_t^\beta\big| \leq \exp(2\mu_wt)
\]
for sufficiently small $\beta$ and any $t>0$.
We thus obtain for sufficiently small $\beta$,
\begin{align*}
    \P\big(\mathring\phi_t^{(1),\beta} \cap A_{b+at} \neq \varnothing, \mathring\phi_t^{\beta} \cap A_{b+\rho at} = \varnothing\big) &\textstyle\leq M_2 e^{-\delta t}, \quad t>0.
\end{align*}
The same bound holds for $k=0$ and $k=2$ by parts (4) and (6) of Lemma \ref{lemma:branchingrate}.
By setting $M := 3M_2$, the desired result then follows from \eqref{eq:originalstatement} and \eqref{eq:lemma7intermediate1}.
\end{proof}

\subsection{Proof of Lemma \ref{lemma:approximateupperbound2}} \label{app:brwapproxbrw_upper_2}

\begin{lemma} 
For a particle chosen uniformly at random from $\phi_t^{\beta,0}$, let $(Y_s^\beta)_{s \leq t}$ be the path followed by this particle and its ancestors, and let $(\mathring Y_s^\beta)_{s \leq t}$ be the corresponding path in $\mathring \phi_t^{\beta,0}$. Then, for any $r>0$ and $\delta>0$, there exist
$M>0$ and 
$\beta_0>0$ so that
\begin{align*}
     &\textstyle \P\big(\sup_{s \leq t} ||Y_s^\beta-\mathring Y_s^\beta||>r t\big) \leq M e^{-\delta t}, \quad \beta \leq \beta_0, \; t>0.
\end{align*}
\end{lemma}

\begin{proof}
As in the proof of Lemma \ref{lemma:approximateupperbound}, we suppress the initial conditions of the processes from the notation, and assume that each is started with a single particle at the origin.

Consider the pruned BRW $(\mathring\phi_s)_{s \geq 0}$ (in the unscaled spacetime). Type-1 and type-2 branching events occur at total rate $\sim \mu_w/h(\beta)$ as $\beta \to 0$ by Lemma \ref{lemma:branchingrate}, where we recall that $h(\beta) = (1/\beta) \cdot \log(1/\beta)$. In the simple BRW $(\phi_s)_{s \geq 0}$, we modify the path of parent and daughter at each type-1 and type-2 event as outlined in Section \ref{sec:markovianupper}. Type-0 branching events, where the daughter is not introduced to $(\mathring\phi_s)_{s \geq 0}$, occur at rate $\sim \beta$ as $\beta \to 0$ by Lemma \ref{lemma:branchingrate}. In the simple BRW $(\phi_s)_{s \geq 0}$, we modify the parent's path at a type-0 event with probability $1-\alpha_0(\beta) \sim \mu_w/\log(1/\beta)$ as $\beta \to 0$, so the modification rate for type-0 events is $\sim \mu_w/h(\beta)$. The total rate of modifications for type-0, type-1 and type-2 events is therefore $\sim 2\mu_w/h(\beta)$ as $\beta \to 0$, which translates into $\sim 2\mu_w$ in the scaled spacetime.

Consider now a type-1 branching event in $(\mathring\phi_s)_{s \geq 0}$.
Assume for simplicity that the parent is at the origin at the time of branching.
Let $(Z_t^1)_{t \geq 0}$ and $(Z_t^2)_{t \geq 0}$ be the SSRWs followed by parent and daughter from the time of branching, let $T_0$ be the time at which they first meet, and let $S$ be the time at which the daughter first branches. Note that the two paths are conditioned on $\{S \leq \min\{T_0,\tau(\beta)\}\}$. Next, let $(Z_t^3)_{t \geq 0}$ and $(Z_t^4)_{t \geq 0}$ be independent SSRWs with jump rate 1 that are independent of $(Z_t^1)_{t \geq 0}$, $(Z_t^2)_{t \geq 0}$ and $S$, both started at the origin. Define
\[
\hat{Z}_t^1 := \begin{cases} Z_t^3, & t \leq S, \\ Z_S^3+(Z_t^1-Z_S^1), & t \geq S, \end{cases}
\]
and
\[
\hat{Z}_t^2 := \begin{cases} Z_t^4, & t \leq S, \\ Z_S^4+(Z_t^2-Z_S^2), & t \geq S. \end{cases}
\]
In the simple BRW $(\phi_s)_{s \geq 0}$, we replace $(Z_t^1)_{t \geq 0}$ and $(Z_t^2)_{t \geq 0}$ by $(\hat Z_t^1)_{t \geq 0}$ and $(\hat Z_t^2)_{t \geq 0}$, respectively.
Note that in the definition of $(\hat Z_t^1)_{t \geq 0}$, we connect the paths $(Z_t^3)_{t \leq S}$ and $(Z_t^1)_{t \geq S}$ at time $S$ (see Figure \ref{fig:markovian_brw} in the main text), which requires perturbing $(Z_t^1)_{t \geq S}$ by $Z_S^3-Z_S^1$ conditional on $\{S \leq \min\{T_0,\tau(\beta)\}\}$. Call this perturbation $X$.
By the same argument as used to establish part (5) in Lemma \ref{lemma:branchingrate},
there exists $C>0$ so that for sufficiently small $\beta$,
\[
\E||X||^j = \E\big[||Z_{S}^3-Z_{S}^1||^j \,\big|\, S \leq \min\{T_0,\tau(\beta)\}\big] \leq C j! (\log(1/\beta))^{1/2} \tau(\beta)^{j/2}, \quad j \geq 1.
\]
The perturbation of $(Z_t^2)_{t \geq S}$ in the definition of $(\hat Z_t^2)_{t \geq 0}$ has the same upper bound by Lemma \ref{lemma:branchingrate}.
Furthermore, the same argument, using part (6) of Lemma \ref{lemma:branchingrate} instead of part (5), will show that this upper bound also applies to perturbations on type-2 branching events.

On type-0 branching events, $(Z_t^1)_{t \geq 0}$ and $(Z_t^2)_{t \geq 0}$ will be conditioned on $\{T_0 \leq \min\{S,\tau(\beta)\}\}$.
Here, $(Z_t^3)_{t \geq 0}$ and $(Z_t^4)_{t \geq 0}$ will be conditioned on $\{S' \leq \min\{T_0',\tau(\beta)\}\}$ or $\{\tau(\beta) \leq \min\{S',T_0'\}\}$, with $S'$ and $T_0'$ defined in terms of $(Z_t^3)_{t \geq 0}$ and $(Z_t^4)_{t \geq 0}$. 
Let $X$ be the perturbation required to connect $(Z_t^3)_{t \leq \tau(\beta)}$ and $(Z_t^1)_{t \geq \tau(\beta)}$ at time $\tau(\beta)$ (recall that we only modify the parent's path on a type-0 branching event).
By the same argument as used to establish parts (4)-(6) in Lemma \ref{lemma:branchingrate}, we can show that
\begin{align} \label{eq:perturbationbound}
    \E||X||^j \leq C j! (\log(1/\beta))^{1/2} \tau(\beta)^{j/2}, \quad j \geq 1,
\end{align}
the same as for the type-1 and type-2 branching events.

Consider next the scaled processes $(\mathring \phi_s^\beta)_{s \geq 0}$ and $(\phi_s^\beta)_{s \geq 0}$.
We have already observed that each particle path in $(\mathring\phi_s^\beta)_{s \geq 0}$ is perturbed at total rate $\sim 2\mu_w$ as $\beta \to 0$ to produce the corresponding path in $(\phi_s^\beta)_{s \geq 0}$.
Some of the perturbations occur on type-0 branching events, in which case no new particle is introduced, while the remaining perturbations occur on type-1 and type-2 branching events, in which case one new particle is introduced.
To keep track of the perturbations, we define a branching process embedded in $(\mathring \phi_s^\beta)_{s \geq 0}$ as follows:
\begin{itemize}
    \item Each individual in the branching process is associated with a particle in $(\mathring \phi_s^\beta)_{s \geq 0}$.
    \item At each type-0 perturbation in $(\mathring \phi_s^\beta)_{s \geq 0}$, the corresponding individual in the branching process is killed and replaced by another individual.
    \item At each type-1 or type-2 perturbation in $(\mathring \phi_s^\beta)_{s \geq 0}$, the corresponding individual in the branching process is killed and replaced by two individuals.
\end{itemize}
For a given particle path in $(\mathring \phi_s^\beta)_{s \geq 0}$, the number of perturbations up until time $t$ is the generation number at time $t$ of the corresponding individual in the branching process.
The branching process has branching rate $\sim 2\mu_w$ and mean number of offspring $\sim 3/2$ per branching event as $\beta \to 0$.
To obtain an upper bound on the number of perturbations by time $t$, it therefore suffices to establish the following lemma.

\begin{lemmalemma} \label{lemma:generationnumber}
Let $(\varphi_s)_{s \geq 0}$ be a continuous-time branching process with exponential branching rate $\alpha>0$ and offspring distribution $(p_j)_{j=0}^\infty$ with $p_0=0$, started with a single individual.
Let $G_t$ be the generation number of an individual selected uniformly at random from $\varphi_t$.
Set $m := \sum_j j p_j$ and assume $1 < m \leq m_0$. Then, for any $c>m_0$,
\[
\P(G_t>c\alpha t) \leq \exp\big(\alpha t \cdot (2c(m_0-1)-(c-m_0)^2)/(2c)\big), \quad t>0.
\]
\end{lemmalemma}

\begin{proof}
Let $\varphi_{j,t}$ be the number of individuals
at time $t$ that are in generation $j$.
Then
\begin{align*} 
    \textstyle \P(G_t>c\alpha t) &= \textstyle  \E\big[(1/{\varphi_t}) \cdot \sum_{j > c\alpha t} \varphi_{j,t}\big] \leq \sum_{j > c\alpha t} \E\big[\varphi_{j,t}\big],
\end{align*}
where we use that $\varphi_t \geq 1$ since $p_0=0$. Note that $\E\big[\varphi_{j,t}\big] = m^j \P(R_t = j)$, where $R_t$ is Poisson distributed with mean $\alpha t$.
Let $\hat{R}_t$ be Poisson distributed with mean $m_0\alpha t$.
Then
\begin{align} \label{eq:tailprob}
     \textstyle \P(G_t>c\alpha t) \leq 
     \textstyle\sum_{j > c\alpha t} m_0^j \, \P(R_t = j) 
     = \textstyle \exp\big(\alpha t (m_0-1)\big) \cdot \P(\hat{R}_t > c\alpha t), \quad t>0.
\end{align}
Note next that by Theorem 1 of \cite{Canonne2017}, for $c>m_0$,
\begin{align*}
    \P(\hat{R}_t > c\alpha t) 
    &\leq \exp\big(\!-\!\alpha t \cdot (c-m_0)^2/(2c)\big), \quad t>0.
\end{align*}
The result then follows from \eqref{eq:tailprob}.
\end{proof}

We are now ready to begin the main calculations.
First, select a particle uniformly at random from $\phi_t^\beta$.
Let $(Y_s^\beta)_{s \leq t}$ be the path followed by this particle and its ancestors, and let $(\mathring Y_s^\beta)_{s \leq t}$ be the corresponding path in $\mathring \phi_t^\beta$.
The number of perturbations $G_t^\beta$ between the two paths by time $t$ is the generation number of a particle selected uniformly at random from the embedded branching process.
By the previous lemma, for a given $\delta>0$, we can select $c>0$ sufficiently large so that for sufficiently small $\beta$,
\begin{align*}
    & \P(G_t^{\beta} > ct) \leq \exp(-\delta t), \quad t>0.
\end{align*}
This implies
\begin{align} \label{eq:conditionalprob1}
&\textstyle\P\big(\sup_{s \leq t} ||Y_s^\beta-\mathring Y_s^\beta||>r t\big) \nonumber \\
&\leq \textstyle\P\big(\sup_{s \leq t} ||Y_s^\beta-\mathring Y_s^\beta||>r t, G_t^{\beta} \leq ct\big) + e^{-\delta t}, \quad t>0.
\end{align}
To analyze the probability in \eqref{eq:conditionalprob1}, note that by the above observations, we can write
\begin{align*}
\textstyle Y_t^{\beta} = \mathring Y_t^{\beta} + \sum_{k=1}^{G_t^{\beta}} X_k^{\beta},
\end{align*}
where $X_1^{\beta},X_2^{\beta},\ldots$ is the (independent) sequence of scaled perturbations,
and we assume for simplicity that the time point $t$ does not occur during a decision period. It follows that
\begin{align} \label{eq:hardpart}
&\textstyle\textstyle\P\big(\sup_{s \leq t} ||Y_s^\beta-\mathring Y_s^\beta||>r t, G_t^{\beta} \leq ct\big) 
    \leq \textstyle \textstyle \P\big(\sup_{m \leq ct} \big\|\sum_{k=1}^{m} X_k^{\beta}\big\|> rt\big).
\end{align}
To estimate this probability, we begin by noting that 
\begin{align} \label{eq:goingtoonedim}
& \textstyle \textstyle \P\big(\sup_{m \leq ct} \big\|\sum_{k=1}^{m} X_k^{\beta}\big\|> rt\big) \nonumber \\
&= \textstyle \textstyle \P\big(\sup_{m \leq ct} h(\beta)^{-1/2} \big\|\sum_{k=1}^{m} X_k\big\|> rt\big) \nonumber \\
&\leq \textstyle \textstyle 4\,\P\big( \sup_{m \leq ct} h(\beta)^{-1/2} \sum_{k=1}^{m} X_{1,k} > (1/2)  rt\big),
\end{align}
where we write $X_k = (X_{1,k},X_{2,k},X_{3,k})$.
We next consider the moment generating function
\begin{align*} \label{eq:mgf}
\textstyle \E\big[\exp\big(\nu h(\beta)^{-1/2} \sum_{k=1}^{\lfloor ct\rfloor} X_{1,k}\big)\big] &= \textstyle \prod_{k=1}^{\lfloor ct\rfloor} \E\big[\exp(\nu h(\beta)^{-1/2} X_{1,k})\big]. 
\end{align*}
Using the moment bound \eqref{eq:perturbationbound}, we can show that
\begin{align*} 
    \textstyle \E\big[\exp(\nu h(\beta)^{-1/2} X_{1,k})\big] &=
1 + O_{C,\nu} \big((\log(1/\beta))^{-1}\big), \quad k \geq 1, 
\end{align*}
using the same argument we used to establish \eqref{eq:mgfperturbation} in the proof of Lemma \ref{lemma:approximateupperbound}.
For a given $\nu>0$, take $\beta_0 = \beta_0(\nu)>0$ so that for $\beta \leq \beta_0$,
\[
\E\big[\exp(\nu h(\beta)^{-1/2} X_{1,k})\big] \leq \exp\big((1/4)(r\nu/c)\big), \quad k \geq 1,
\]
which implies that for $\beta \leq \beta_0$,
\[
\textstyle \prod_{k=1}^{\lfloor ct\rfloor} \E\big[\exp(\nu h(\beta)^{-1/2} X_{1,k})\big]  \leq \exp\big((1/4) r\nu t\big), \quad t>0.
\]
We then obtain by Doob's inequality for $\beta \leq \beta_0$,
\begin{align*}
    &\textstyle \P\big( \sup_{m \leq ct} h(\beta)^{-1/2} \sum_{k=1}^{m} X_{1,k} > (1/2) rt\big) 
    \leq\exp\big(\!-\!(1/4) r\nu t\big), \quad t > 0.
\end{align*}
Choosing $\nu$ appropriately, we obtain by \eqref{eq:goingtoonedim} for 
sufficiently small $\beta$,
\begin{align*}
    \textstyle \P\big(\sup_{m \leq ct} \big\|\sum_{k=1}^{m} X_k^{\beta}\big\|> rt\big) \leq 4 \, e^{-\delta t}, \quad t > 0.
\end{align*}
Finally, combining with \eqref{eq:conditionalprob1} and \eqref{eq:hardpart}, we obtain the desired result.
\end{proof}

\subsection{Proof of Lemma \ref{lemma:approximateupperbound3}} \label{app:brwapproxbrw_upper}

\begin{lemma} 
Fix $a>0$ and $0 < \rho < 1$, and define $A_r := [r,\infty) \times \R \times \Z_w$.
For each $\delta>0$, there exist
$M>0$ and 
$\beta_0>0$ so that
\begin{align*}
     &\textstyle \P\big(\mathring\phi_{t}^{\beta,0} \cap A_{b+at} \neq \varnothing\big) \leq \P\big(\phi_{t}^{\beta,0} \cap A_{b+\rho at} \neq \varnothing \big) + M e^{-\delta t}, \quad \beta \leq \beta_0, \; t>0, \; b \in \R.
\end{align*}
\end{lemma}

\begin{proof}
As in the proofs of Lemmas \ref{lemma:approximateupperbound} and \ref{lemma:approximateupperbound2}, we suppress the initial conditions of the processes from the notation, and assume that each is started with a single particle at the origin.
Since $0<\rho<1$, we can begin by writing
\begin{align*}
     &\textstyle \P\big(\mathring\phi_{t}^\beta \cap A_{b+at} \neq \varnothing\big) \\
     &\leq\textstyle \P\big(\phi_{t}^\beta \cap A_{b+\rho at} \neq \varnothing\big)+ \P\big(\mathring\phi_{t}^\beta \cap A_{b+at} \neq \varnothing, \phi_{t}^\beta \cap A_{b+\rho at} = \varnothing\big).
\end{align*}
Enumerate the particles in $\phi_t^\beta$ as $Y^{\beta,1}_t,Y^{\beta,2}_t,\ldots,Y_{t}^{\beta,N_t^\beta}$, where $N_t^\beta$ is the number of particles in $\phi_t^\beta$.
For each $Y_t^{\beta,i}$, let $\mathring Y_t^{\beta,i}$ denote the position of the corresponding particle
in $\mathring \phi_t^\beta$.
Then 
\begin{align*} 
\P\big(\mathring\phi_{t}^\beta \cap A_{b+at} \neq \varnothing, \phi_{t}^\beta \cap A_{b+\rho at} = \varnothing\big)  &\leq\textstyle \P\big(\bigcup_{i=1}^{N_t^\beta}\{||\mathring Y_t^{\beta,i}-Y_t^{\beta,i}|| > (1-\rho)at\}\big),
\end{align*}
and by Markov's inequality, 
\begin{align*} 
\textstyle \P\big(\bigcup_{i=1}^{N_t^\beta}\{||\mathring Y_t^{\beta,i}-Y_t^{\beta,i}|| > (1-\rho)at\}\big) \nonumber
    &\leq\textstyle \E\!\left[\sum_{i=1}^{N_t^\beta} 1_{\{||\mathring Y_t^{\beta,i}-Y_t^{\beta,i}|| > (1-\rho)at\}}\right].
\end{align*}
Let $I$ be the index of a particle chosen uniformly at random from
$Y^{\beta,1}_t,Y^{\beta,2}_t,\ldots,Y_{t}^{\beta,N_t^\beta}$,
i.e.~$\P\big(I=j\big|N_t^\beta\big)=1/N_t^\beta$ for $j=1,\ldots,N_t^\beta$.
Then
\begin{align*}
 \textstyle \E\!\left[\sum_{i=1}^{N_t^\beta} 1_{\{||\mathring Y_t^{\beta,i}-Y_t^{\beta,i}|| > (1-\rho)at\}}\right] = \E\!\left[N_t^\beta \cdot 1_{\{||\mathring Y_t^{\beta,I}-Y_t^{\beta,I}|| > (1-\rho)at\}}\right],
\end{align*}
and by Cauchy-Schwarz,
\begin{align*}
    \textstyle \E\!\left[N_t^\beta \cdot 1_{\{||\mathring Y_t^{\beta,I}-Y_t^{\beta,I}|| > (1-\rho)at\}}\right] \leq \sqrt{\E\big[(N_t^\beta)^2\big]} \cdot \sqrt{\P\big(||\mathring Y_t^{\beta,I}-Y_t^{\beta,I}|| > (1-\rho)at\big)}.
\end{align*}
Now, $\phi_t^\beta$ branches at exponential rate $\gamma(\beta) = \mu_w+o(1)$, with mean number of new particles $\ell(\beta) = 1+o(1)$ introduced per branching event.
Since $\gamma(\beta) \cdot \ell(\beta) \leq 2\mu_w$ for sufficiently small $\beta$, by Lemma 5 in \cite{foo2014escape}, there exists $M_1>0$ so that for sufficiently small $\beta$,
\begin{align*}
    \E\big[(N_t^\beta)^2\big] \leq M_1 (e^{2\mu_wt})^2, \quad t>0.
\end{align*}
Furthermore, for any $\delta>0$, by Lemma \ref{lemma:approximateupperbound2} above, there exists $M_2>0$ so that for sufficiently small $\beta$,
\begin{align*} 
\P\big(||\mathring Y_t^{\beta,I}-Y_t^{\beta,I}|| > (1-\rho)at\big) \leq M_2\, e^{-2(\delta+2\mu_w)t}, \quad t>0.
\end{align*}
Combining the above, there exists $M>0$ so that for sufficiently small $\beta$,
\begin{align*}
\P\big(\mathring\phi_{t}^\beta \cap A_{b+at} \neq \varnothing, \phi_{t}^\beta \cap A_{b+\rho at} = \varnothing\big) &\leq \textstyle M e^{-\delta t}, \quad t>0,
\end{align*}
and the result follows. \qedhere
\end{proof}

\subsection{Proof of Lemma \ref{lemma:upperbound}} \label{app:upperboundproof}

\begin{lemma} 
Define $\tau_{\varnothing}^A := \min\{t \geq 0 : \xi_{t}^A = \varnothing \}$ for $A \subseteq \mathbb{Z}^2 \times \Z_w$. For each $\kappa>1$, there 
exists a family of random variables $(S_{\beta})_{\beta>0}$, with $\P(S_\beta<\infty|\tau_\varnothing^0=\infty)=1$ for each $\beta>0$, so that
\begin{align*}
\textstyle \lim_{\beta\to 0} \liminf_{t\to\infty} \P\big(
\lceil\kappa a_wh(\beta)^{1/2} t\rceil e_1 \notin \xi_{S_\beta+h(\beta)t}^0 \,\big|\,  \tau_{\varnothing}^0 = \infty \big)  = 1,
\end{align*}
where $a_w := p_w \sqrt{\pi w}$. 
\end{lemma}

\begin{proof}
Take $\varepsilon>0$. We segment the proof into three main steps. \\

\noindent {\bf Step 1: Remove conditioning on nonextinction.} 
Let $B_r$ denote a box in $\Z^2 \times \Z_w$ centered at 0 with radius $r$, i.e.
\[
B_r := \{(x_1,x_2,x_3) \in \mathbb{Z}^2 \times \Z_w: \max\{|x_1|,|x_2|\} \leq r\}.
\]
If we set $r = r(\beta) = h(\beta)^{1/2}$, then by the gambler's ruin formula,
\begin{align*}
\P(\tau_\varnothing^{B_r}<\infty) &= (1+\beta)^{-(2\lfloor r\rfloor+1)^2w} 
= \exp\big(\!-\! \Theta(h(\beta)) \cdot \log(1+\beta)\big)
= o(1), \quad \beta \to 0,
\end{align*}
since $h(\beta) = (1/\beta) \cdot \log(1/\beta)$ and $\log(1+\beta) = \beta + o(\beta)$. 
Define $\sigma_{r,R} = \sigma_{r,R}(\beta)$ as
\begin{align*} 
\sigma_{r,R} := \inf\{t \geq 0: B_{r}\subseteq \xi_t^0 \subseteq B_{Rt}\}  . 
\end{align*}
By (3) and (7) in Bramson \& Griffeath \cite{BraGri80} (i.e.~the corresponding results for $\mathbb{Z}^2 \times \Z_w$), there exists a constant $R>0$ so that
\begin{align*} 
    \P(\sigma_{r,R}<\infty \,|\, \tau_\varnothing^0=\infty) =1,
\end{align*}
which implies $\P(\tau_\varnothing^0 = \infty) \leq \P(\sigma_{r,R}<\infty)$.
We then get by the strong Markov property and the monotonicity property \eqref{eq:monotonicity} of $\xi_t$, for any $t>0$, $m \geq 1$ and sufficiently small $\beta$,
\begin{align*}
    &\P\big( \lceil\kappa a_wh(\beta)^{1/2} t\rceil e_1 \notin  \xi_{\sigma_{r,R}+h(\beta)t}^0, \tau_\varnothing^0=\infty\big) \nonumber \\
    &\geq \textstyle \P\big( \lceil\kappa a_wh(\beta)^{1/2} t\rceil e_1 \notin  \xi_{\sigma_{r,R}+h(\beta)t}^0, \tau_\varnothing^0=\infty, \sigma_{r,R} \leq m\big) \nonumber \\
     &= \textstyle \sum_{B_r \subseteq \Lambda \subseteq B_{mR}} \P\big(\lceil\kappa a_wh(\beta)^{1/2} t\rceil e_1 \notin  \xi_{h(\beta)t}^\Lambda, \tau_\varnothing^\Lambda=\infty\big) \P(\sigma_{r,R} \leq m, \xi_{\sigma_{r,R}}^0 = \Lambda) \nonumber\\
          &\geq \textstyle \P\big(\lceil\kappa a_wh(\beta)^{1/2} t\rceil e_1 \notin  \xi_{h(\beta)t}^{B_{mR}}, \tau_\varnothing^{B_r}=\infty\big) \P(\sigma_{r,R} \leq m) \nonumber\\
          &\geq \textstyle \big(\P\big(\lceil\kappa a_wh(\beta)^{1/2} t\rceil e_1 \notin  \xi_{h(\beta)t}^{B_{mR}}\big) - \P\big(\tau_\varnothing^{B_r}<\infty\big)\big) \P(\sigma_{r,R} \leq m) \nonumber\\
          &\geq \textstyle \big(\P\big(\lceil\kappa a_wh(\beta)^{1/2} t\rceil e_1 \notin  \xi_{h(\beta)t}^{B_{mR}}\big) - \varepsilon\big) \P(\sigma_{r,R} \leq m). \nonumber 
\end{align*}
If we can establish that for sufficiently small $\beta$,
\begin{align*}
\textstyle\lim_{t \to \infty} \P\big(\lceil\kappa a_wh(\beta)^{1/2} t\rceil e_1 \notin  \xi_{h(\beta)t}^{B_{mR}}\big) = 1, \quad m \geq 1,
\end{align*}
then sending $m \to \infty$ will yield for sufficiently small $\beta$,
\begin{align*}
&  \textstyle \liminf_{t \to \infty} \P\big( \lceil\kappa a_wh(\beta)^{1/2} t\rceil e_1 \notin  \xi_{\sigma_{r,R}+h(\beta)t}^0, \tau_\varnothing^0=\infty\big) \geq (1-\varepsilon) \P(\sigma_{r,R}<\infty),
\end{align*}
and the desired result will follow from $\P(\sigma_{r,R}<\infty) \geq \P(\tau_\varnothing^0 = \infty)$. Equivalently, it is sufficient to show that for sufficiently small $\beta$,
\begin{align}
\textstyle\lim_{t \to \infty} \P\big(\lceil\kappa a_wh(\beta)^{1/2} t\rceil e_1 \in  \xi_{h(\beta)t}^{B_{mR}}\big) = 0, \quad m \geq 1. \label{eq:upperwanttoshow}    
\end{align}

\noindent {\bf Step 2: Introduce duality and apply approximation scheme.}
Note first that by the duality relation \eqref{dual_relation} between $\xi_t$ and $\tilde\zeta_t$, the translation invariance \eqref{eq:translinvariance} and symmetry property \eqref{eq:symmetry} of the dual process $\tilde\zeta_t$, and the definition \eqref{eq:defnscaleddual} of the scaled dual process $\tilde\zeta_t^\beta$,
    \begin{align}  \label{eq:intermediate3point5}
        \P\big(\lceil\kappa a_wh(\beta)^{1/2} t\rceil e_1 \in  \xi_{h(\beta)t}^{B_{mR}}\big) &= \P\big(\tilde\zeta_{h(\beta)t}^{\lceil\kappa a_wh(\beta)^{1/2} t\rceil e_1} \cap B_{mR} \neq \varnothing\big) \nonumber \\
        &= \P\big(\tilde\zeta_{h(\beta)t}^{0} \cap (B_{mR} + \lceil\kappa a_wh(\beta)^{1/2} t\rceil e_1) \neq \varnothing\big) \nonumber \\
        &= \P\big(\tilde\zeta_{t}^{\beta,0} \cap h(\beta)^{-1/2}(B_{mR} + \lceil\kappa a_wh(\beta)^{1/2} t\rceil e_1) \neq \varnothing\big).
    \end{align}
     Set $A_r := [r,\infty) \times \R \times \Z_w$
  and take $\kappa_1$ and $\kappa_2$ so that $1<\kappa_2<\kappa_1<\kappa$.
  For any $x = (x_1,x_2,x_3) \in B_{mR}$, we have $|x_1| \leq mR$ where $R>0$ is a constant.
  We can therefore choose $\beta$ sufficiently small that $|x_1^\beta| = h(\beta)^{-1/2} |x_1| \leq m$ for all $x \in B_{mR}$. Using the approximation Lemmas \ref{lemma:approximateupperbound} and \ref{lemma:approximateupperbound3}, we then obtain for sufficiently small $\beta$,
    \begin{align}  \label{eq:intermediate4}
    & \textstyle \limsup_{t \to \infty} \P\big(\tilde\zeta_{t}^{\beta,0} \cap h(\beta)^{-1/2}(B_{mR} + \lceil\kappa a_wh(\beta)^{1/2} t\rceil e_1) \neq \varnothing\big)  \nonumber \\
    &\leq \textstyle \limsup_{t \to \infty} \P(\tilde\zeta_t^{\beta,0} \cap A_{-m+\kappa a_wt} \neq \varnothing) \nonumber \\
        &\leq\textstyle 4 \limsup_{t \to \infty}\P\big( \mathring\phi_{t}^{\beta,0} \cap A_{-m+\kappa_1 a_w t} \neq \varnothing \big) \nonumber \\
        &\leq\textstyle 4 \limsup_{t \to \infty}\P\big( \phi_{t}^{\beta,0} \cap A_{-m+\kappa_2 a_w t} \neq \varnothing \big).
    \end{align}
    We have now reduced the problem to analyzing the tail of $\phi_t^{\beta,0}$, which is straightforward. \\
    
      \noindent {\bf Step 3: Analyze simple BRW.} We begin by using Markov's inequality to write
      \begin{align} \label{eq:expnumpart}
\P\big( \phi_{t}^{\beta,0} \cap A_{-m+\kappa_2 a_w t} \neq \varnothing \big) \leq \E\big|\phi_{t}^{\beta,0} \cap A_{-m+\kappa_2 a_w t} \big|.
\end{align}
    Recall that $\phi_t^{\beta,0}$ has branching rate $\mu_w+o(1)$, and on average, $1+o(1)$ new particles are added per branching event. Therefore,
    \begin{align} \label{eq:manytoone_brw_upper}
   &\E\big|\phi_{t}^{\beta,0} \cap A_{-m+\kappa_2 a_w t}  \big| 
   = \exp\big((\mu_w+o(1))t\big)  \cdot \P(h(\beta)^{-1/2} S_{h(\beta)t} \geq -m+\kappa_2 a_w t),     
    \end{align}
    where $(S_{s})_{s \geq 0}$ is the path of the SSRW on $\mathbb{Z}$ started at 0 with jump rate $p_w/2$, where $p_w$ is defined as in \eqref{eq:birthfirstwodimensions}. By \eqref{eq:momgenrwonedim}, its moment generating function is
        \begin{align*}
     \psi_s(\theta)&= \exp\big((p_w/2)s \cdot (\phi(\theta)-1)\big),
        \end{align*}
        where $\phi(\theta) = (e^\theta+e^{-\theta})/2$.
    Set $\theta_0 := 2\sqrt{\pi w } \kappa_2$, and note that
            \begin{align*}
     \psi_{h(\beta)t}(\theta_0 h(\beta)^{-1/2}) &=  \exp\big((p_w/2) h(\beta)t \cdot (\phi(\theta_0 h(\beta)^{-1/2})-1)\big). 
        \end{align*}
    Since $\phi(\theta_0 h(\beta)^{-1/2}) = 1+(1/2) \theta_0^2 h(\beta)^{-1} + o(h(\beta)^{-1})$, and $\mu_w = p_w \pi w$, we get
    \begin{align*}
        \psi_{h(\beta)t}(\theta_0h(\beta)^{-1/2}) &= 
       \exp\big((p_w \pi w \kappa_2^2 + o(1)) t\big) = 
      \exp\big((\mu_w \kappa_2^2 + o(1)) t\big). 
    \end{align*}
    Now, since $a_w = p_w \sqrt{\pi w}$, $\theta_0 = 2\sqrt{\pi w}\kappa_2$ and $\mu_w = p_w \pi w$, we have $\kappa_2 a_w \theta_0 = 2\mu_w \kappa_2^2$. We therefore obtain by Markov's inequality:
    \begin{align*}
    &\P(h(\beta)^{-1/2} S_{h(\beta)t} \geq -m+\kappa_2 a_w t ) \\
        &=\P\big(\exp(\theta_0 h(\beta)^{-1/2} S_{h(\beta)t}) \geq \exp( -m\theta_0+\kappa_2 a_w \theta_0 t)\big) \\
        &\leq 
       \exp\big((\mu_w \kappa_2^2 + o(1)) t\big) \cdot \exp\big(\!-2 \mu_w\kappa_2^2 t\big) \cdot e^{\theta_0m} \\
        &= 
       \exp\big((-\kappa_2^2 \mu_w + o(1))t\big) \cdot e^{\theta_0m}.
    \end{align*}
    Combining with \eqref{eq:manytoone_brw_upper}, we obtain
    \begin{align*}
       \E\big|\phi_{t}^{\beta,0} \cap A_{-m+\kappa_2 a_w t} \big|  &\leq 
       \exp\big(((1-\kappa_2^2)\mu_w + o(1))t\big) \cdot e^{\theta_0m}.
    \end{align*}
    Take $\kappa_3$ such that $1<\kappa_3<\kappa_2^2$. Then for sufficiently small $\beta$,
    \begin{align*} \label{eq:tailprobrw}
    \E\big|\phi_{t}^{\beta,0} \cap A_{-m+\kappa_2 a_w t} \big| &\leq \exp\big((1-\kappa_3)\mu_wt\big) \cdot e^{\theta_0m}.
    \end{align*}
   Combining this with \eqref{eq:intermediate3point5}, \eqref{eq:intermediate4} and \eqref{eq:expnumpart}, we obtain for sufficiently small $\beta$ and any $m \geq 1$,
   \begin{align*}
       \textstyle \limsup_{t \to \infty} \P\big(\lceil\kappa a_wh(\beta)^{1/2} t\rceil e_1 \in  \xi_{h(\beta)t}^{B_{mR}}\big) = 0, 
   \end{align*}
   which yields \eqref{eq:upperwanttoshow}, as desired.
    \end{proof}
    
    \subsection{Proof of Lemma \ref{lemma:approximation}} \label{app:prunedbrwapproximatedpruneddual}
    
    \begin{lemma} 
    Set $d(\beta) := \beta^{-1/2}(\log(1/\beta))^{-1}$. Let ${\cal A} = {\cal A}(\beta)$ denote the collection of finite subsets of $\Z^2 \times \Z_w$ in which points are pairwise separated by at least $d(\beta)$.
Set $A^\beta := h(\beta)^{-1/2}A$ for $A \in {\cal A}$ and ${\cal A}^\beta := h(\beta)^{-1/2} {\cal A}$.
Then, for any $K>0$ and $T>0$,
\[
\textstyle \sup_{A \in {\cal A}, \; |A| \leq K} \P\big(\{(\hat\zeta_t^{\beta,A^\beta})_{t \leq T} \neq (\mathring\psi_t^{\beta,A^\beta})_{t \leq T}\} \cup
\{\mathring\psi_T^{\beta,A^\beta} \notin {\cal A}^\beta\}\big) 
\to 0, \quad \beta \to 0.
\]
\end{lemma}

\begin{proof}
Recall that the pruned dual process $\hat\zeta_t$ includes any particle from the dual process $\tilde\zeta_t$ that has not coalesced with {\em any other particle in the process} by time $\tau(\beta)$. 
To show that
\begin{align*} 
(\hat\zeta_t^{\beta,A^\beta})_{t \leq T} = (\mathring\psi_t^{\beta,A^\beta})_{t \leq T}    
\end{align*}
with high probability for sufficiently small $\beta$, 
we need to show that any particle in the dual process $\tilde\zeta_t$ that does not coalesce with its parent by time $\tau(\beta)$ will, with high probability, (i) not coalesce with any other particle in the process before time $\tau(\beta)$, and (ii) neither coalesce with its parent nor another particle in the process after time $\tau(\beta)$.

Assume that the starting set $A$ has at most $K$ particles, i.e.~$|A| \leq K$.
On pages 1758-1759 of \cite{durrett2007}, Durrett and Z{\"a}hle establish the following for the $w=1$ case, i.e.~for $\Z^2$:
\begin{itemize}
    \item For any $\varepsilon>0$, there exists $M = M(\varepsilon,K,T)$ and $\beta_0 >0$ such that 
    \begin{align} \label{eq:durrzahl0}
    \P(|\hat\zeta_T^{\beta,A^\beta}| > M) \leq \varepsilon, \quad \beta \leq \beta_0,
    \end{align}
    i.e.~with high probability, the total number of particles in the scaled, pruned dual process at time $T$ is finite.
    \item If $Z_t^1$ and $Z_t^2$ are independent SSRWs on $\Z^2$ with jump rate 1, $\bar{Z}_t := Z_t^1-Z_t^2$, and $x = x(\beta) \in \Z^2$ with $||x||>d(\beta)$, then for any $R>0$,
\begin{align} \label{eq:durrzahl1.5}
\P(||\bar{Z}_t|| \leq R \text{ for some $t \leq h(\beta)T$} \,|\, \bar{Z}_0 = x) \to 0, \quad \beta \to 0.
\end{align}
\item If $Z_t^1$ and $Z_t^2$ are started at nearest neighbors, and $T_0$ denotes the time at which they first meet, then by the local central limit theorem on $\Z^2$, 
\begin{align} \label{eq:durrzahl2}
    \P(||\bar{Z}_{\tau(\beta)}|| \leq d(\beta) \,|\,  T_0 > \tau(\beta)) 
    = O\big((\log(1/\beta))^{-1/2}\big) = o(1), \quad \beta \to 0.
\end{align}
\end{itemize}
All statements continue to be true on $\Z^2 \times \Z_w$ for $w>1$.
In addition, following the same argument as used to establish \eqref{eq:durrzahl2},
we note that if $Z_t^1$ and $Z_t^2$ are independent SSRWs on $\Z^2 \times \Z_w$ with jump rate 1, $\bar{Z}_t = Z_t^1-Z_t^2$, 
and $g(\beta) = \omega\big(d(\beta)^2\big)$, i.e.~$d(\beta)^2/g(\beta) \to 0$ as $\beta \to 0$,
then by the local central limit theorem \eqref{eq:lcltZdsqrt} on $\Z^2$,
\begin{align} \label{eq:durrzahl3}
    \P\big(||\bar{Z}_{g(\beta)}|| \leq d(\beta) \,|\, \bar{Z}_0 = x) &= \textstyle \sum_{y: ||y|| \leq d(\beta)} \P(\bar{Z}_{g(\beta)} = y \,|\, \bar{Z}_0 = x ) \nonumber \\
    &= w (2d(\beta)+1)^2 \cdot O\big(g(\beta)^{-1}\big) \nonumber \\
    & = o(1)
\end{align}
for all $x \in \Z^2 \times \Z_w$.
The fact that $(\hat\zeta_t^{\beta,A^\beta})_{t \leq T} = (\mathring\psi_t^{\beta,A^\beta})_{t \leq T}$
with high probability (w.h.p.) will now follow from an induction argument similar to the one presented on pages 1758-1759 of \cite{durrett2007}.
To see that $\mathring\psi_T^{\beta,A^\beta} \in {\cal A}^\beta$ w.h.p., we first take $M$ so that there are at most $M$ particles w.h.p., by the same argument as in \eqref{eq:durrzahl0}.
In the unscaled process $(\mathring\psi_t^{A})_{t \leq h(\beta)T}$, the probability that at least one particle gives birth in the last $2\tau(\beta)$ time units is at most $M\big(1-\exp(-2\beta\tau(\beta))\big) = o(1)$, since $\tau(\beta) = (1/\beta) \big(1/\sqrt{\log(1/\beta)}\big)$.
If at time $h(\beta)T-2\tau(\beta)$, there exists a parent-daughter pair whose decision period has not yet passed, the daughter will be introduced to the pruned process by time $h(\beta)T-\tau(\beta)$, given that it does not coalesce with its parent.
Then, applying \eqref{eq:durrzahl3} to each pair of particles alive at time $h(\beta)T-\tau(\beta)$, and using the fact that $\tau(\beta) = \omega\big(d(\beta)^2\big)$, we see that the particles in $(\mathring\psi_t^{A})_{t \leq h(\beta)T}$ will be pairwise separated by at least $d(\beta)$ at time $h(\beta)T$ w.h.p.
It follows that
$\mathring\psi_T^{\beta,A^\beta} \in {\cal A}^\beta$ w.h.p. \qedhere
\end{proof}

\subsection{Proof of Lemma \ref{lemma:percolation}} \label{app:percolation}

\begin{lemma}
For $0<\theta<1$ and $L>0$, define (Fig.~\ref{fig:percolation}b)
 \begin{align*}
     & I_0^\theta := [-(1/2)\theta a_w L,(1/2)\theta a_w L]^2 \times \Z_w,  \\
     & I_k^\theta := I_0^\theta + k \cdot \theta a_w Le_1, \; k \in \mathbb{Z},
 \end{align*}
and
 \[
 {\cal A}^{\beta,\theta,K,k} := \{A^\beta \in {\cal A}^{\beta}: |A^\beta \cap I_k^\theta | \geq K\},
 \]
 with ${\cal A}^{\beta}$ defined as in Lemma \ref{lemma:approximation}.
 Let $(\hat\zeta_{t}^{\beta,A^\beta,\theta})_{t \geq 0}$ denote a pruning of $(\hat\zeta_{t}^{\beta,A^\beta})_{t \geq 0}$ with particles killed as soon as they exit the box $I_\Delta^\theta$ with $I_\Delta^\theta := [-2\theta a_w L,2\theta a_w L]^2 \times \mathbb{Z}_w$.
 Then, for any $2/3<\theta<1$ and $\varepsilon>0$, there exist $L=L(\theta)>0$, $K=K(\theta,\varepsilon)>0$ and $\beta_0 = \beta_0(\theta,\varepsilon)>0$ such that for any $A^\beta \in {\cal A}^{\beta,\theta,K,0}$ with $|A^\beta|=K$, and any $\beta \leq \beta_0$,
 \begin{align*}
\P(\hat{\zeta}_{{L}}^{\beta,A^\beta,\theta} \in {\cal A}^{\beta,\theta,K,k}) \geq 1-\varepsilon, \quad k \in \{-1,1\}.
\end{align*}
\end{lemma}

\begin{proof}
Let $L>0$ and $K>0$ be constants to be selected later.
Furthermore, let $(\mathring\psi_{t}^{\beta,A^\beta,\theta})_{t \geq 0}$ denote a pruning of $(\mathring\psi_{t}^{\beta,A^\beta})_{t \geq 0}$ with particles killed as soon as they exit $I_\Delta^{\theta}$.
By Lemma \ref{lemma:approximation}, we have for all $A^\beta \in {\cal A}^{\beta,\theta,K,0}$ with $|A^\beta|=K$ and sufficiently small $\beta$,
\begin{align}
    &\P(\hat{\zeta}_{{L}}^{\beta,A^\beta,\theta} \in {\cal A}^{\beta,\theta,K,k}) \nonumber \\
    &\geq\P\big(|\mathring{\psi}_{{L}}^{\beta,A^\beta,\theta} \cap I_k^\theta| \geq K,  (\hat\zeta_{t}^{\beta,A^\beta})_{t \leq L} = (\mathring\psi_{t}^{\beta,A^\beta})_{t \leq L}, \mathring\psi_T^{\beta,A^\beta} \in {\cal A}^\beta\big) \nonumber \\
    &\geq\P\big(|\mathring{\psi}_{{L}}^{\beta,A^\beta,\theta} \cap I_k^\theta| \geq K\big) - \P\big( \{(\hat\zeta_{t}^{\beta,A^\beta})_{t \leq L} \neq (\mathring\psi_{t}^{\beta,A^\beta})_{t \leq L}\} \cup \{\mathring\psi_T^{\beta,A^\beta} \notin {\cal A}^\beta\}\big) \nonumber \\
    &\geq\P\big(|\mathring{\psi}_{{L}}^{\beta,A^\beta,\theta} \cap I_k^\theta| \geq K\big) - \varepsilon/4, \quad k \in \{-1,1\}. \label{eq:percolation4}
\end{align}
Define a simple BRW $(\psi_t^\beta)_{t \geq 0}$ in terms of $(\mathring\psi_t^\beta)_{t \geq 0}$ analogously to how $(\phi_t^\beta)_{t \geq 0}$ is defined in terms of $(\mathring\phi_t^\beta)_{t \geq 0}$ in Section \ref{sec:markovianupper}.
Take $\theta_1$ so that $2/3<\theta_1<\theta$, and let $(\psi_{t}^{\beta,A^\beta,\theta_1})_{t \geq 0}$ denote a pruning of $(\psi_{t}^{\beta,A^\beta})_{t \geq 0}$ with particles killed as soon as they exit $I_\Delta^{\theta_1}$.
Now, consider the event
\[
\{|{\psi}_{{L}}^{\beta,A^\beta,\theta_1} \cap I_k^{\theta_1}| \geq K\}, \quad k \in \{-1,1\}.
\]
On the above event, pick $K$ particles from ${\psi}_{{L}}^{\beta,A^\beta,\theta_1} \cap I_k^{\theta_1}$, and consider one such particle $Y_L^{\beta}$. 
Lemma \ref{lemma:approximateupperbound2} implies that the distance between the path of this particle (and its ancestors) and the corresponding particle $\mathring Y_L^\beta$ in $\mathring \psi_L^\beta$ (and its ancestors) up until time $L$ is upper bounded by $(1/2)(\theta-\theta_1)a_wL$ with high probability given sufficiently small $\beta$.
This implies that w.h.p., $\mathring Y_L^\beta$ and its ancestors stay within $I_\Delta^{\theta}$ during $[0,L]$, and $\mathring Y_L^\beta$ ends up in $I_k^{\theta}$. Thus,
for sufficiently small $\beta$,
\begin{align*}
    \P\big(|{\psi}_{{L}}^{\beta,A^\beta,\theta_1} \cap I_k^{\theta_1} | \geq K\big) \leq \P(|\mathring{\psi}_{{L}}^{\beta,A^\beta,\theta} \cap I_k^\theta | \geq K) +\varepsilon/4, \quad k \in \{-1,1\},
\end{align*}
which yields by \eqref{eq:percolation4} for sufficiently small $\beta$,
\begin{align}
    &\P(\hat{\zeta}_{{L}}^{\beta,A^\beta,\theta} \in {\cal A}^{\beta,\theta,K,k}) \geq \P\big(|{\psi}_{{L}}^{\beta,A^\beta,\theta_1} \cap I_k^{\theta_1} | \geq K\big)-\varepsilon/2, \quad k \in \{-1,1\}.
 \label{eq:percolation3}
\end{align}
We now wish to estimate the probability on the right-hand side of \eqref{eq:percolation3}. 
Recall that the branching rate of $\psi_t^\beta$ is $\mu_w+o(1)$ by Lemma \ref{lemma:branchingrate}.
Then, for any $z^\beta \in I_0^{\theta_1} \cap h(\beta)^{-1/2}(\Z^2 \times \Z_w)$,
\begin{align} \label{eq:numberofparticlesbrownian}
&\mathbb{E} \big|\psi_{L}^{\beta,z^\beta,\theta_1} \cap I_k^{\theta_1} \big| = \exp\big({(\mu_w+o(1)) L}\big) \cdot \P\big(Z_L^{\beta,z^\beta,\theta_1} \in I_k^{\theta_1}\big), \quad k \in \{-1,1\},
\end{align}
where $(Z_t^{\beta,z^\beta})_{t \geq 0}$ is the scaled version of the SSRW $(Z_t^{z})_{t \geq 0}$ on $\mathbb{Z}^2 \times \Z_w$ with jump rate 1, and $( Z_t^{\beta,z^\beta,\theta_1})_{t \geq 0}$ is a pruned version where the walk is  killed if it exits $I_\Delta^{\theta_1}$ (see e.g.~(7.5) of \cite{durrett95} for why \eqref{eq:numberofparticlesbrownian} is true).
For the probability on the right-hand side, note that
\begin{align*} \label{eq:brownian1}
    & \P\big (Z_L^{\beta,z^\beta,\theta_1} \in I_k^{\theta_1} \big) = \P\big (Z_L^{\beta,z^\beta} \in I_k^{\theta_1} , Z_t^{\beta,z^\beta} \in I_\Delta^{\theta_1}  \text{ for all $t \leq L$}\big),     \quad k \in \{-1,1\}.
\end{align*}
Without loss of generality, set $k := 1$. Take $\delta>0$ so that for any $z^\beta = (z_1^\beta,z_2^\beta,z_3^\beta) \in I_0^{\theta_1} \cap h(\beta)^{-1/2}(\Z^2 \times \Z_w)$ with $z_1^\beta \geq 0$ and $z_2^\beta \geq 0$,
\[
(z^\beta + ([-\delta,0]^2 \times \Z_w)) 
\subseteq I_0^{\theta_1}.
\]
For such $z^\beta$ and $\delta>0$, and $J_\Delta^{\theta_1} := [-(3/2) \theta_1 a_wL, (3/2) \theta_1 a_wL]^2 \times \Z_w$, 
\begin{align} \label{eq:brownian2}
    &\P\big (Z_{L}^{\beta,z^\beta,\theta_1} \in I_1^{\theta_1} \big) \nonumber \\
 &= \P\big (Z_L^{\beta,z^\beta} \in I_1^{\theta_1} , Z_t^{\beta,z^\beta} \in I_\Delta^{\theta_1} \text{ for all $t \leq L$}\big) \nonumber \\
    &\geq \textstyle \P\big (Z_{L}^{\beta,0} \in ([-\delta,0]^2 \times \Z_w)  + \theta_1 a_w L e_1, Z_{t}^{\beta,0} \in J_\Delta^{\theta_1} \text{ for all $t \leq L$}\big) \nonumber \\
    &\geq \textstyle \P\big (Z_{L}^{\beta,0} \in ([-\delta,0]^2 \times \Z_w) + \theta_1 a_w L e_1\big) - \P(Z_{t}^{\beta,0} \notin J_\Delta^{\theta_1}  \text{ for some $t \leq L$}),
\end{align}
where we use the translation invariance of $(Z_{t}^\beta)_{t \geq 0}$. 
To analyze the latter term in \eqref{eq:brownian2}, write $Z_t = (Z_{1,t}, Z_{2,t},Z_{3,t})$, where $Z_{1,t}$ and $Z_{2,t}$ are i.i.d.~copies of the SSRW on $\Z$ with jump rate $p_w/2$ each, and $Z_{3,t}$ is the SSRW on $\Z_w$ with jump rate $1-p_w$, where $p_w$ is defined as in \eqref{eq:birthfirstwodimensions}.
Then note that
\begin{align}
    &\P(Z_{t}^{\beta,0} \notin J_\Delta^{\theta_1}  \text{ for some $t \leq L$}) \nonumber \\
    &\leq 2\,\P(Z_{1,t}^{\beta,0} \notin [-(3/2) \theta_1 a_w L,(3/2)\theta_1 a_w  L] \text{ for some $t \leq L$}) \nonumber \\ 
    &\leq 4\,\P\big(\textstyle h(\beta)^{-1/2} \sup_{t \leq h(\beta)L} Z_{1,t}^{0} > (3/2) \theta_1 a_w L\big). \nonumber
\end{align}
By assumption, $(3/2)\theta_1>1$. The same argument we used to analyze \eqref{eq:manytoone_brw_upper} in the proof of Lemma \ref{lemma:upperbound} (take $\kappa_2 := (3/2)\theta_1>1$ and $m := 0$, and use Doob's inequality to handle the supremum) will show that we can take $\gamma_2>0$ so that for sufficiently small $\beta$,
\begin{align} \label{eq:brownian3}
\exp\big({(\mu_w+o(1)) L}\big) \cdot \P\big(\textstyle h(\beta)^{-1/2} \sup_{t \leq h(\beta)L} Z_{1,t}^{0} > (3/2) \theta_1 a_w L\big) \leq \exp(-\gamma_2 \mu_w L).
\end{align}
To analyze the former term in \eqref{eq:brownian2}, note that the local central limit theorem \eqref{eq:lcltZdsqrt}  on $\Z^2$ 
implies that for all $y \in \Z^2$,
\[
\P\big((Z_{1,h(\beta)L}^0,Z_{2,h(\beta)L}^0) =  y \big) = \big(p_w\pi h(\beta)L\big)^{-1} \exp\big(\!-\!||y||^2/(p_w h(\beta)L)\big) + o_L\big(h(\beta)^{-1}\big),
\]
where $o_L$ signifies that the error term depends on $L$.
Since
\begin{align*}
    &\P\big (Z_{L}^{\beta,0} \in ([-\delta,0]^2 \times \Z_w) + \theta_1 a_w L e_1\big) \\
    &= \P\big ((Z_{1,h(\beta)L}^0,Z_{2,h(\beta)L}^0) \in ([-\delta h(\beta)^{1/2},0]^2 + \theta_1 a_w h(\beta)^{1/2} L e_1) \big) \nonumber \\
    &= \textstyle \sum_{y \in ([-\delta h(\beta)^{1/2},0]^2 + \theta_1 a_w h(\beta)^{1/2} L e_1)} \P\big((Z_{1,h(\beta)L}^0,Z_{2,h(\beta)L}^0)=y\big),
\end{align*}
and the number of terms in the sum is of order $h(\beta)$, we obtain for some $C>0$,
\begin{align*}
    &\P\big (Z_{L}^{\beta,0} \in ([-\delta,0]^2 \times \Z_w) + \theta_1 a_w L e_1\big) \geq (C/L) \cdot \exp(-\theta_1^2 \mu_w L) + o_L(1),
\end{align*}
where we use that $a_w^2/p_w = \mu_w$ since $a_w = p_w \sqrt{\pi w}$ and $\mu_w = p_w \pi w$.
Since $\theta_1<1$, we can then find $\gamma_1>0$ so that for sufficiently small $\beta$,
\begin{align} \label{eq:brownian4}
    & \exp\big({(\mu_w+o(1)) L}\big) \cdot \P\big (Z_{L}^{\beta,0} \in ([-\delta,0]^2 \times \Z_w) + \theta_1 a_w L e_1\big) \nonumber \\
    &\geq (C/L) \cdot \exp(\gamma_1 \mu_w L).
\end{align}
Combining \eqref{eq:brownian3} and \eqref{eq:brownian4} with \eqref{eq:numberofparticlesbrownian} and \eqref{eq:brownian2}, we obtain for sufficiently small $\beta$,
\begin{align*} 
&\mathbb{E} |\psi_{L}^{\beta,z^\beta,\theta_1}\cap I_1^{\theta_1} |
\geq (C/L) \cdot \exp\big(\gamma_1 \mu_w L\big) - 4 \exp\big(-\gamma_2\mu_w L\big),
\end{align*}
where $\gamma_1, \gamma_2>0$. Since the former term tends to $\infty$ as $L \to \infty$ and the latter term tends to 0 as $L \to \infty$, we can select $L=L(\theta_1)$ large enough so that
\begin{align*} \label{eq:brownian4}
\mathbb{E} |\psi_{L}^{\beta,z^\beta,\theta_1} \cap I_1^{\theta_1} | \geq 2, \quad z^\beta \in I_0^{\theta_1}  \cap h(\beta)^{-1/2}(\Z^2 \times \Z_w), \; z_1^\beta \geq 0, \; z_2^\beta \geq 0.  
\end{align*}
The same is true for $k=-1$, as well as $z^\beta$ in the second, third and fourth quadrant of $I_0^{\theta_1} \cap h(\beta)^{-1/2}(\Z^2 \times \Z_w)$.
Now, using the same argument as in Durrett and Z{\"a}hle (\cite{durrett2007}, p.~1760), for the given $L$ and any $\varepsilon>0$, we can select $K = K(\theta_1,\varepsilon)>0$
large enough so that for all $A^\beta \in {\cal A}^{\beta,\theta,K,0}$ with $|A^\beta|=K$ and sufficiently small $\beta$,
\[
\P(|\psi_{L}^{\beta,A^\beta,\theta_1} \cap I_k^{\theta_1}| < K) \leq \varepsilon/2, \quad k \in \{-1,1\}.
\]
Combining with \eqref{eq:percolation3}, which holds for sufficiently small $\beta$ given fixed $K$ and $L$, we obtain the desired result.
\end{proof}

\subsection{Proof of Lemma \ref{lemma:lowerbound}} \label{app:lowerboundproof}

\begin{lemma}
Define $\tau_{\varnothing}^A := \min\{t \geq 0 : \xi_{t}^A = \varnothing \}$ for $A \subseteq \mathbb{Z}^2 \times \Z_w$. For each $2/3<\rho<1$, there exists a constant $L>0$ and a family of random variables $(S_{\beta})_{\beta>0}$, with $\P(S_\beta<\infty|\tau_\varnothing^0=\infty)=1$ for each $\beta>0$, so that
\begin{align*}
\textstyle \lim_{\beta\to 0} 
\liminf_{n\to\infty} \P\big(
\xi_{S_\beta+2nLh(\beta)}^0 \cap \big[2n \rho a_w L h(\beta)^{1/2},\infty\big) e_1 \neq  \varnothing \,\big|\,  \tau_{\varnothing}^0 = \infty \big) = 1,
\end{align*}
where $a_w := p_w \sqrt{\pi w}$. 
\end{lemma}

\begin{proof}
Take $\varepsilon>0$. We segment the proof into three main steps. \\

\noindent {\bf Step 1: Remove conditioning on nonextinction.} Let $C_r$ denote a box in $\Z^2 \times \Z_w$, centered at 0 with side lengths $10r$ and $2r$, i.e.
\[
C_r := \{(x_1,x_2,x_3) \in \mathbb{Z}^2 \times \Z_w: |x_1| \leq 5r, |x_2| \leq r\},
\]
and set $\sigma_{r} := \inf\{t \geq 0: C_{r} \subseteq \xi_t^0\}$.
As in the proof of Lemma \ref{lemma:upperbound}, we note that if $r = r(\beta) = M \cdot h(\beta)^{1/2}$ for some $M>0$, then $\P(\tau_\varnothing^{C_r} < \infty) = o(1)$ as $\beta \to 0$, and $\P(\sigma_{r}<\infty\,|\,\tau_\varnothing^0=\infty)=1$, so $\P(\sigma_{r}<\infty) \geq \P(\tau_\varnothing^0=\infty)$.
Let $L>0$ and $\rho<\rho_1<1$, and set
\begin{align*}
& \gamma := (1/2)\rho_1 a_w L \quad\text{and}\quad r := \gamma h(\beta)^{1/2}.
\end{align*}
We then get by the strong Markov property and the monotonicity property \eqref{eq:monotonicity} of $\xi_t$,
for any $n \geq 1$ and sufficiently small $\beta$,
\begin{align}
    &\P\big( \xi_{\sigma_{r}+2nLh(\beta)}^0 \cap \big[2n \rho a_w L h(\beta)^{1/2},\infty\big) e_1 \neq  \varnothing, \tau_\varnothing^0=\infty\big) \nonumber \\
      &= \!\textstyle\sum_{C_{r} \subseteq \Lambda} \int_s \P\big(\xi_{2nLh(\beta)}^\Lambda \cap \big[2n \rho a_w L h(\beta)^{1/2},\infty\big) e_1 \neq  \varnothing,  \tau_\varnothing^{\Lambda}=\infty \big) \P\big(\sigma_{r} \in ds, \xi_{\sigma_{r}}^0 = \Lambda \big) \nonumber\\
      &\geq \P\big(\xi_{2nLh(\beta)}^{C_{r}} \cap \big[2n \rho a_w L h(\beta)^{1/2},\infty\big) e_1 \neq  \varnothing,  \tau_\varnothing^{C_{r}} = \infty \big) \P(\sigma_{r}<\infty)\nonumber\\
      &\geq  \big(\P\big(\xi_{2nLh(\beta)}^{C_{r}} \cap \big[2n \rho a_w L h(\beta)^{1/2},\infty\big) e_1 \neq  \varnothing\big) - \P\big(\tau_\varnothing^{C_{r}} <\infty\big)\big) \P(\sigma_{r}<\infty) \nonumber\\
       &\geq \big(\P\big(\xi_{2nLh(\beta))}^{C_{r}} \cap \big[2n \rho a_w L h(\beta)^{1/2},\infty\big) e_1 \neq  \varnothing \big) - \varepsilon\big) \P(\tau_\varnothing^0=\infty), \nonumber
\end{align}
from which it follows that
\begin{align}
    &\P\big( \xi_{\sigma_{r}+2nLh(\beta)}^0 \cap \big[2n \rho a_w L h(\beta)^{1/2},\infty\big) e_1 \neq  \varnothing \,\big|\, \tau_\varnothing^0=\infty\big) \nonumber \\
       &\geq \P\big( \xi_{2nLh(\beta)}^{C_{r}} \cap \big[2n \rho a_w L h(\beta)^{1/2},\infty\big) e_1 \neq  \varnothing \big) - \varepsilon.
\label{eq:intermediate1}
\end{align}

\noindent {\bf Step 2: Introduce duality.} Let $K>0$ be a constant to be selected later, and set $d(\beta) := \beta^{-1/2} (\log(1/\beta))^{-1}$ as in Lemma \ref{lemma:approximation}. Define
\begin{align*}
    & A_0 = A_0(\beta) := \lceil d(\beta)\rceil \cdot \llbracket-K+1,0\rrbracket \cdot e_1, \\
    &A_0^\beta := h(\beta)^{-1/2} A_0,
\end{align*}
with $\llbracket m,n\rrbracket = \{m,m+1,\ldots,n\}$ for integers $m<n$ (with possibly $n=\infty$).
Note that $A_0$ has $K$ points which are pairwise separated by at least $d(\beta)$.
Since $d(\beta) = o\big(\sqrt{h(\beta)}\big)$, it follows that $A_0^\beta \in {\cal A}^{\beta,\rho_1,K,0}$ for sufficiently small $\beta$, where ${\cal A}^{\beta,\rho_1,K,0}$ is defined as in Lemma \ref{lemma:percolation}. 
Next define
\[
g(n) = g(n,\beta) := (2\rho_1 a_w L h(\beta)^{1/2})^{-1} \lceil 2n \rho a_w L h(\beta)^{1/2} \rceil,
\]
and note that for fixed $\beta$, $g(n)/n \to \rho/\rho_1 < 1$ as $n\to \infty$.
Continuing on from \eqref{eq:intermediate1}, we obtain using the duality relation \eqref{dual_relation} between $\xi_t$ and $\tilde\zeta_t$, the monotonicity property \eqref{eq:monotonicity} of $\tilde\zeta_t$, the translation invariance \eqref{eq:translinvariance} and symmetry property \eqref{eq:symmetry} of $\tilde\zeta_t$, and the definition \eqref{eq:defnscaleddual} of the scaled dual process $\tilde\zeta_t^\beta$,
\begin{align}
    &\P\big(\xi_{2nLh(\beta)}^{C_{r}} \cap \big[2n \rho a_w L h(\beta)^{1/2},\infty\big) e_1 \neq  \varnothing\big) \nonumber \\
    &= \P\big(\tilde\zeta_{2nLh(\beta)}^{ \llbracket \lceil 2n \rho a_w L h(\beta)^{1/2}\rceil,\infty\rrbracket e_1} \cap C_{r} \neq \varnothing \big) \nonumber \\
        &\geq \P\big(\tilde\zeta_{2nLh(\beta)}^{- A_0+\lceil 2n \rho a_w L h(\beta)^{1/2}\rceil e_1} \cap C_{r} \neq \varnothing \big) \nonumber \\
    &= \P\big(\tilde\zeta_{2nLh(\beta)}^{A_0} \cap (C_{r}+\lceil 2n \rho a_w L h(\beta)^{1/2} \rceil e_1 ) \neq \varnothing \big) \nonumber \\
         &= \P\big(\tilde\zeta_{2nL}^{\beta,A_0^\beta} \cap \big(([-5\gamma,5\gamma] \times [-\gamma,\gamma] \times \Z_w)+2g(n) \cdot \rho_1a_wL e_1\big) \neq \varnothing \big). \label{eq:intermediate1.5}
\end{align}
Recall that $\gamma = (1/2)\rho_1a_wL$, and define $I_0^{\rho_1} := [-\gamma,\gamma]^2 \times \Z_w$ and $I_k^{\rho_1} := I_0^{\rho_1} + k \cdot 2\gamma e_1$ for $k \in \Z$ as in Lemma \ref{lemma:percolation}.
Then 
\begin{align}
& \P\big(\tilde\zeta_{2nL}^{\beta,A_0^\beta} \cap \big(([-5\gamma,5\gamma] \times [-\gamma,\gamma] \times \Z_w)+2g(n) \cdot \rho_1a_wL e_1\big) \neq \varnothing \big) \nonumber \\
    &\geq \P\big(\tilde\zeta_{2nL}^{\beta, A_0^\beta} \cap \big(I_0^{\rho_1} +2\lceil g(n)\rceil \cdot \rho_1a_wL e_1\big) \neq \varnothing \big)  \nonumber\\
     &\geq \P\big(\hat\zeta_{2nL}^{\beta, A_0^\beta} \cap I^{\rho_1}_{2\lceil g(n)\rceil} \neq \varnothing \big),
    \label{eq:intermediate2}
\end{align}
where 
in the last step, we use the lower-bounding property \eqref{eq:dualprunedduallower} of the pruned dual process $\hat\zeta_t$.  We are now ready to apply the percolation construction of Lemma \ref{lemma:percolation}.
\\

\noindent {\bf Step 3: Compare with oriented percolation.}
Let $(\hat\zeta_t^{\beta,A_0^\beta,\rho_1})_{t \geq 0}$ be a pruning of $(\hat\zeta_t^{\beta,A^\beta})_{t \geq 0}$ with particles killed as soon as they exit the box $I_\Delta^{\rho_1}$ with $I_\Delta^{\rho_1} := [-2\rho_1a_wL,2\rho_1a_wL]^2 \times \Z_w$.
Note first that for any $K>0$,
 \begin{align}
    & \P\big(\hat\zeta_{2nL}^{\beta, A_0^\beta} \cap I^{\rho_1}_{2\lceil g(n)\rceil}  \neq \varnothing \big) \geq \P\big(\hat\zeta_{2nL}^{\beta, A_0^\beta,\rho_1} \in {\cal A}^{\beta,\rho_1,K,2\lceil g(n)\rceil}\big).
        \label{eq:intermediate8}
    \end{align}
By assumption, $2/3<\rho_1<1$, so by Lemma \ref{lemma:percolation}, we can choose $K$ and $L$ so that for any $A^\beta \in {\cal A}^{\beta,\rho_1,K,0}$ with $|A^\beta|=K$ and sufficiently small $\beta$,
    \begin{align*} \label{eq:percinequality}
        \P(\hat{\zeta}_{{L}}^{\beta,A^\beta,\rho_1} \in {\cal A}^{\beta, \rho_1, K, k}) \geq 1-\varepsilon/2, \quad k \in \{-1,1\}.
    \end{align*}
Now set
\[
X_n := \{k \in \mathbb{Z}: \text{$k+n$ even} \text{ and } \hat \zeta_{nL}^{\beta,A_0^\beta,\rho_1} \in {\cal A}^{\beta, \rho_1, K, k} \}, \quad n \geq 0.
\]
By Theorem 4.3 of \cite{durrett95}, $X_n$ dominates a one-dependent oriented percolation process $\{\omega_n^0\}_{n \geq 0}$ with density $\geq 1-\varepsilon$ and $\omega_0^0 = \{0\}$, i.e.~$\omega_n^0 \subseteq X_n$ for all $n$. 
Let $\Omega_\infty^0$ denote the event $\{|\bigcup_n \omega_n^0|=\infty\}$, i.e.~the event that percolation occurs, and let $l_n^0 = \min \omega_n^0$ (resp.~$r_n^0 = \max \omega_n^0$) denote the left (resp.~right) edge of the process. Take $\rho_2$ so that $\rho/\rho_1<\rho_2<1$. By Theorem 4.1 of \cite{durrett95}, we have for sufficiently small $\varepsilon$,
\begin{align} \label{eq:percolation1}
\P(\Omega_\infty^0) \geq 1-55\varepsilon^{1/9},    
\end{align}
and by Theorem 3.21 on page 300 of \cite{liggett2012interacting}, we have for sufficiently small $\varepsilon$,
\begin{align} \label{eq:percolation2}
\textstyle \P(r_n^0/n \geq \rho_2\,|\, \Omega_\infty^0) \geq 1-3^{-n+1}.
\end{align}
Since for fixed $\beta$, $g(n)/n \to \rho/\rho_1$ as $n \to \infty$, we further have for sufficiently large $n$,
\begin{align} \label{eq:choosen}
    \lceil g(n)\rceil /n \leq \rho_2.
\end{align}
Continuing on from \eqref{eq:intermediate8}, we obtain for sufficiently small $\beta$ and sufficiently large $n$,
    \begin{align*}
    \P\big(\hat\zeta_{2nL}^{\beta, A_0^\beta,\rho_1} \in {\cal A}^{\beta,\rho_1,K,2\lceil g(n)\rceil}\big) 
    &\geq \P(2\lceil g(n)\rceil \in \omega_{2n}^0) \nonumber \\
    &\geq \P(2\lceil g(n)\rceil \in \omega_{2n}^0, r_{2n}^0 \geq 2\rho_2n, \Omega_\infty^0).
    \end{align*}
    On $\Omega_\infty^0$, we have $\omega_n^0 = \omega_n^{2\mathbb{Z}} \cap \llbracket l_n^0,r_n^0\rrbracket$ (see Section 8 of \cite{Dur84}). By \eqref{eq:choosen}, we can therefore write for sufficiently large $n$,
    \begin{align*}
        \P(2\lceil g(n)\rceil \in \omega_{2n}^0, r_{2n}^0 \geq 2\rho_2n , \Omega_\infty^0) &= \P(2\lceil g(n)\rceil \in \omega_{2n}^{2\mathbb{Z}} , r_{2n}^0 \geq 2\rho_2n , \Omega_\infty^0) \\
        &\geq\P(2\lceil g(n)\rceil \in \omega_{2n}^{2\mathbb{Z}}) + \P(r_{2n}^0 \geq 2\rho_2n , \Omega_\infty^0)-1 \\
        &= \P(2\lceil g(n)\rceil \in \omega_{2n}^{2\mathbb{Z}}) + \P(r_{2n}^0 \geq 2\rho_2n \,|\, \Omega_\infty^0)\P(\Omega_\infty^0)-1.
    \end{align*}
    Furthermore, $\omega_n^0$ is self-dual (see Section 8 of \cite{Dur84}), so
    \[
    \P(2\lceil g(n)\rceil \in \omega_{2n}^{2\mathbb{Z}}) = \P(\omega_{2n}^{2\lceil g(n)\rceil} \neq \varnothing) = \P(\omega_{2n}^0 \neq \varnothing) \geq \P(\Omega_\infty^0).
    \]
    By \eqref{eq:percolation1} and \eqref{eq:percolation2}, we therefore obtain for sufficiently small $\beta$ and sufficiently large $n$,
    \begin{align*}
    &\P\big(\hat\zeta_{2nL}^{\beta, A_0^\beta,\rho_1} \in {\cal A}^{\beta,\rho_1,K,2\lceil g(n)\rceil}\big) 
    \geq (2-3^{-2n+1}) (1-55\varepsilon^{1/9})-1.
    \label{eq:finalinequalitylowerbound}
    \end{align*}
    Combining this with  \eqref{eq:intermediate1}, \eqref{eq:intermediate1.5}, \eqref{eq:intermediate2} and \eqref{eq:intermediate8}, we finally obtain for sufficiently small $\beta$ and sufficiently large $n$,
    \begin{align*}
    &\P\big( \xi_{\sigma_{r}+2nLh(\beta)}^0 \cap \big[2n \rho a_w Lh(\beta)^{1/2},\infty\big) e_1 \neq  \varnothing \,\big|\, \tau_\varnothing^0=\infty\big) \nonumber \\
       &\geq \big((2-3^{-2n+1}) (1-55\varepsilon^{1/9})-1\big) - \varepsilon,
\end{align*}
and the result follows.
\end{proof}

\subsection{Extension of Bramson-Griffeath shape theorem} \label{app:BGtheorem}

In this section, we discuss how the Bramson-Griffeath shape theorem on $\Z^2$ can be extended to $\Z^2 \times \Z_w$.
Bramson and Griffeath's proof is split across two papers.
In \cite{BraGri81}, they show that with probability 1, the biased voter model on $\Z^d$, conditioned on nonextinction, eventually contains a ball which expands linearly in time.
Then, in \cite{BraGri80}, they use results from \cite{BraGri81} to show that the biased voter model on $\Z^d$ satisfies five conditions formulated by Richardson \cite{Richardson73}, which together guarantee the existence of a linearly-expanding asymptotic shape.

The key to extending the BG shape theorem to $\Z^2 \times \Z_w$ is to interpret Richardson's conditions with the appropriate notion of spatial scaling and distance.
Richardson's conditions deal with the ``first infection times'' of distant points of the form $x/\Delta$ for small $\Delta>0$, i.e.~the times at which the process first reaches these points.
For long-run growth to be linear, the infection time of $x/\Delta$ should be of order $1/\Delta$ as $\Delta \to 0$, i.e.~the infection time of $x/\Delta$ scaled by $\Delta$ should be finite.
In \cite{Richardson73}, Richardson proposes five conditions on these scaled infection times, which together imply the existence of a linearly-expanding asymptotic shape.
On $\Z^2 \times \Z_w$, the mutant clone will expand without bound along the first two coordinates, while the third coordinate remains bounded throughout.
As a result, we have defined the scalar multiplication  \eqref{eq:scalarmultpldef} and the distance from the origin \eqref{eq:seminormdef} in the main text as only applying to the first two coordinates, and we view the scaled points $x/\Delta$ as living on $\R^2 \times \Z_w$.

Once Richardson's conditions have been interpreted appropriately, Bramson and Griffeath's arguments on $\Z^2$ extend naturally to $\Z^2 \times \Z_w$ for the most part, since they largely rely on fundamental properties of the biased voter model, such as the ones presented in \eqref{eq:additivity}-\eqref{eq:symmetry} of Section \ref{sec:duality} and the strong Markov property.
Where modifications are necessary, one can generally focus on movement along the first two coordinates and use the same arguments as on $\Z^2$, 
taking into account that the SSRW on $\Z^2 \times \Z_w$ takes a step in the $\Z^2$-direction with probability $p_w$ given by \eqref{eq:birthfirstwodimensions}.
To avoid repeating the arguments, we will focus our discussion on how to interpret Richardson's conditions on $\Z^2 \times \Z_w$, and on how Richardson's conditions give rise to an asymptotic shape of the form $D = \bigcup_{i \in \Z_w} (X \times \{i\})$ given in the main text.

\subsubsection{Richardson's conditions for the existence of an asymptotic shape} \label{sec:richardson}

We begin by stating a slightly stronger version of Richardson's conditions on $\Z^d$, as presented by Bramson and Griffeath in \cite{BraGri80}.
First, let $(\overline{\xi}_t^0)_{t \geq 0}$ be the biased voter model on $\Z^d$ started at the origin and conditioned on nonextinction, i.e.~the process $(\xi_t^0 | \tau_\varnothing^0 = \infty)_{t \geq 0}$. 
For $x \in \Z^d$, define the first infection time of $x$ as
\begin{align*}
    t(x) := \inf\{t : x \in \overline{\xi}_t^0\},
\end{align*}
with $\inf \varnothing = \infty$.
Then introduce the scaled infection time
\begin{align*}
    t_\Delta(x) := \Delta \cdot t(x/\Delta),
\end{align*}
where $\Delta >0$.
Note that $x/\Delta$ is in general not in $\Z^d$.
We therefore extend the definition of $t(y)$ so that for $y \in \R^d$, it takes the same value as the nearest lattice location in $\Z^d$.

When stating Richardson's conditions, we use the same symbol $\|\cdot\|$ for the Euclidean norm on $\R^d$ as for the distance from the origin \eqref{eq:seminormdef} on $\R^2 \times \Z_w$, since they coincide for $d=2$ and $w=1$.
Let $V(x,y)$ denote a random function of $x,y \in \R^d$ which satisfies
\begin{align*}
    \E[V^2(x,y)] = O(||x+y||), 
\end{align*}
and define
\begin{align*}
    V_\Delta(x,y) := \Delta \cdot V(x/\Delta,y/\Delta).
\end{align*}
Let $\varepsilon_i(\Delta)$ denote an $O(\Delta)$ function, and $\varepsilon_i(\Delta,x)$ denote an $O(\Delta)$ function that depends on $x$.
Let $x,y \in \R^d$.
Richardson's conditions are, in Bramson and Griffeath's formulation:
\begin{itemize}
    \item[(A1)] (Scaled infection times are essentially subadditive.) For all $\alpha \in \mathbb{R}$:
    \begin{align*}
        \P(t_\Delta(x+y) \leq \alpha) \geq \P(t_\Delta(x)+s_\Delta(y)+V_\Delta(x,y) \leq \alpha),
    \end{align*}
    where $s_\Delta(y)$ is a copy of $t_\Delta(y)$ which is independent of $t_\Delta(x)$.
    \item[(A2)] (Long-run growth is at most linear.) For some $L>0$,
    \begin{align*}
        \E[t_\Delta(x)] \geq ||x||/L + \varepsilon_1(\Delta,x).
    \end{align*}
    \item[(A3)] (Nearby sites tend to be infected at similar times.) For some $r'>0$,
    \begin{align*}
        \P(\{y: ||x-y|| \leq r'\delta\} \subseteq \{y: |t_\Delta(x)-t_\Delta(y) | \leq \delta\}) \geq 1-\varepsilon_2(\Delta,\delta).
    \end{align*}
    \item[(A4)] (Moment bound on scaled infection times.) $\E[t_\Delta^2(x)]$ exists, and for some $r>0$,
    \begin{align*}
        \E[t_\Delta^2(x)] \leq \|x/r\|^2 + \varepsilon_3(\Delta).
    \end{align*}
    \item[(A5)] (Symmetry condition.)
    \[
    \E[t_\Delta(x)] \leq \E[t_\Delta(-x)] + \varepsilon_4(\Delta,x).
    \]
\end{itemize}
Under slightly weaker conditions, Richardson shows in \cite{Richardson73} that if (A1)-(A5) hold for a given growth process on $\R^d$, then $\lim_{\Delta \to 0} \E[t_\Delta(x)]$ exists. He defines
\begin{align} \label{eq:normdef}
    \textstyle N(x) := \lim_{\Delta \to 0} \E[t_\Delta(x)],
\end{align}
and shows that $N$ is a norm on $\R^d$.
He also shows that ${\rm Var}[t_\Delta(x)] = \varepsilon_5(\Delta,x)$, which implies that $\lim_{\Delta \to 0} t_\Delta(x) = N(x)$ in probability.
    Let $D_R' := \{x \in \R^d: N(x) \leq R\}$ be the ball of radius $R$ under $N$.
    Richardson's main result is that if (A1)-(A5) hold for a given growth process, then for any $\varepsilon>0$, there exists $t_\ast <\infty$ so that
    \begin{align} \label{eq:richardsonresult}
        \P\big(D'_{(1-\varepsilon)t} \cap \Z^d \subseteq \{x: t(x) \leq t\} \subseteq D'_{(1+\varepsilon)t}\big) \geq 1-\varepsilon, \quad t \geq t_\ast.
    \end{align}
    
    On $\R^2 \times \Z_w$, we reinterpret Richardson's conditions (A1)-(A5) in terms of the scalar multiplication in \eqref{eq:scalarmultpldef} and the distance function in \eqref{eq:seminormdef}.
    Set ${e}_3 = (0,0,1)$.
    By definition, $\E[t_\Delta({e}_3)] = \Delta \E[t(e_3)] = O(\Delta)$, and by condition (A4), $\E[t_\Delta^2(e_3)] = O(\Delta)$.
    Thus, the effect of jumping between layers is insignificant under the scaling by $\Delta$.
    In particular, $N(e_3) = 0$, so $N$ does not separate points on $\R^2 \times \Z_w$, which is a consequence of the fact that $\|\cdot\|$ on $\R^2 \times \Z_w$ does not separate points.
    However, $N$ satisfies the triangle inequality and $N(tx) = |t|N(x)$ for $t \in \R$ and $x \in \R^2 \times \Z_w$, i.e.~$N$ has the properties of a seminorm.
    (We do not explicitly refer to $N$ as a seminorm since
    $\R^2 \times \Z_w$ is not a vector space.)
    From the triangle inequality, it follows that $N(x+e_3) = N(x)$ for all $x \in \R^2 \times \Z_w$.
    
    Apart from the above, Richardson's proof that \eqref{eq:richardsonresult} follows from (A1)-(A5) extends to $\R^2 \times \Z_w$ for the most part.
    A slight modification is needed in his Lemma 5, where he makes use of the topology on $\R^d$. 
    The lemma states that the convergence $\lim_{\Delta \to 0} \E[t_\Delta(x)] = N(x)$ is uniform in $x$ on any bounded set around the origin.
    On $\R^2 \times \Z_w$, we can use Richardson's argument to establish uniform convergence layer by layer.
    Then, since there are only finitely many layers, we still get the desired conclusion that for any $R>0$, $|N(x)-\E[t_\Delta(x)]| = O(\Delta)$ for all $x$ with $\|x\| \leq R$.
    In his Lemma 7, Richardson shows that
    \begin{align*} 
        \textstyle {\rm Var}[t_\Delta(x)] \leq {\rm Var}\big[t_\Delta(\frac12 x)+s_\Delta(\frac12 x)\big] + \varepsilon_6(\Delta,x),
    \end{align*}
    where $s_\Delta(\frac12x)$ is an independent copy of $t_\Delta(\frac12 x)$.
    The proof uses that $x=\frac12x+\frac12x$, which is not true in general on $\R^2 \times \Z_w$.
    It is true, however, up to addition by a vector parallel to $e_3$, whose scaled infection time has negligible variance by the above.
    In his Lemma 10, Richardson uses the fact that if a site $x/\Delta$ is first infected at some time $ \leq 1$ in the scaled spacetime, there must be a path from the origin to $x/\Delta$ on which first infection times are $\leq 1$, and along which $N(y)$ varies continuously.
    Since $N(y+e_3) = N(y)$,
    adding jumps across layers will not disrupt this continuity property.
    
    Note that \eqref{eq:richardsonresult} is a weaker statement than the Bramson-Griffeath shape theorem, both because it does not hold almost surely, and because $\{x: t(x) \leq t\}$ describes the set of points that have been infected at some point before time $t$, whereas the shape theorem \eqref{eq:shapethm} makes a claim on the state of the process at time $t$.
    This is the reason that Bramson and Griffeath use slightly stronger conditions than Richardson's original conditions to prove their result.

\subsubsection{Final result on $\Z^2 \times \Z_w$}  \label{sec:bragri2}

As previously mentioned, Bramson and Griffeath's proof that (A1)-(A5) hold for the biased voter model on $\Z^2$, conditioned on nonextinction, will carry over to $\Z^2 \times \Z_w$ with minor modifications.
To state the final result on $\Z^2 \times \Z_w$, let $N$ be defined by \eqref{eq:normdef} and set $D_R' := \{x \in \R^2 \times \Z_w: N(x) \leq R\}$.
Then, for all $\varepsilon>0$,
    \begin{align*}
        \P\big(\exists\; t_\ast < \infty: D'_{(1-\varepsilon)t} \cap (\Z^2 \times \Z_w) \subseteq \xi_t^0 \subseteq D'_{(1+\varepsilon)t} \; \forall t \geq t_\ast \;\big|\; \tau_\varnothing^0=\infty\big) =1,
    \end{align*}
    which is the stronger version of \eqref{eq:richardsonresult}.
In the statement \eqref{eq:shapethm} of the shape theorem in the main text, the set $D$ is the unit ball on $\R^2 \times \Z_w$ under $N$, i.e.~$D := D_1'$.
Define 
\[
X := \{(x_1,x_2) \in \R^2: (x_1,x_2,0) \in D\} = \{(x_1,x_2) \in \R^2: N((x_1,x_2,0)) \leq 1\}.
\]
Since $N(x+e_3)=N(x)$ for all $x \in \R^2 \times \Z_w$, we must have
$D = \cup_{i \in \Z_w} (X \times \{i\})$.
Since $N$ satisfies the triangle inequality and $N(tx)=|t|N(x)$ for $t \in \R$ and $x \in \R^2 \times \Z_w$,
the set $X$ is convex on $\R^2$.
Finally, $X$ inherits all symmetries of the biased voter model.

\begin{appendix}

\section{Boundary condition comparison} \label{app:boundarycondition}

Here, we use simulation to compare the propagation speed of the biased voter model on $w$ layers of two-dimensional integer lattices for two different boundary conditions along the third dimension. On the one hand, we consider $\Z^2 \times \Z_w$ with a periodic boundary condition, and on the other hand, we consider $\Z^2 \times \llbracket 0,w-1\rrbracket$ with a reflecting boundary condition, i.e.~cells on the top (resp.~bottom) layer can only replace cells on the same layer and the layer immediately below (resp.~above).
In Figure \ref{fig:speed_beta}, we show results of simulations of these two processes given tissue thickness $w=2, 3, 4, 5$ and fitness advantage $\beta=0.01, 0.05, 0.1$. We ran at least 30 simulations for each set of parameters and stopped each simulation when the process reached $(100,0,0)$ or $(-100,0,0)$. We then used this data to determine an average speed and a 95\% confidence interval for each set of parameters.

\begin{figure}
\centering
\includegraphics[scale=0.75]{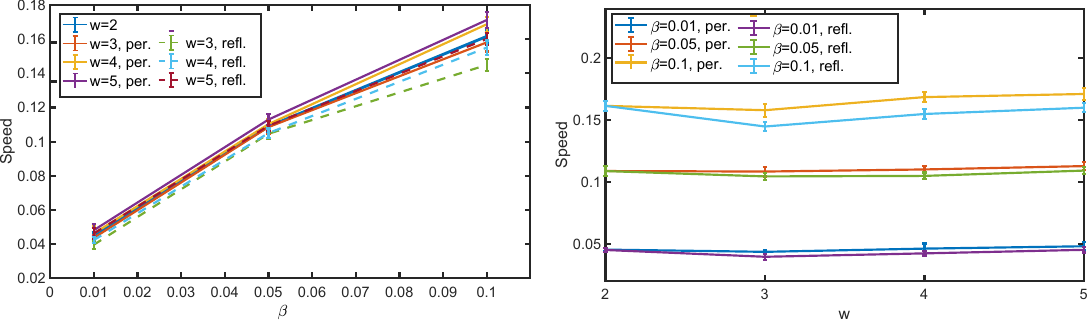}
\caption{Simulation comparison of the propagation speed on $\Z^2 \times \Z_w$, with the third dimension equipped with a periodic boundary condition, and on $\Z^2 \times \llbracket 0,w-1\rrbracket$, with the third dimension equipped with a reflecting boundary condition. In {\bf (a)}, we show the propagation speed as a function of $\beta$ for $w=2$ to $w=5$. In {\bf (b)}, we show the propagation speed as a function of $w$ for $\beta=0.01, 0.05, 0.1$. Error bars indicate 95\% confidence intervals.}
\label{fig:speed_beta}
\end{figure}

Note first that the two boundary conditions are equivalent for $w=2$ two layers. When $w>2$, equipping the model with a reflecting boundary condition along the third dimension will result in a smaller propagation speed than for the periodic case, due to the decreased ability of type-1 cells on the top and bottom layers to spread out. However, the difference is small, especially for smaller values of the fitness advantage $\beta$, which indicates that our modeling decision to equip the third layer with a periodic boundary condition is a reasonable approximation for small $\beta$ (recall that in \cite{bozic2010accumulation}, $\beta = 0.004$ is estimated to be a typical value). \\

\end{appendix}

\noindent {\bf Acknowledgments.}
The authors would like to thank the anonymous reviewer for a careful reading of the manuscript and for insightful comments and suggestions, which helped substantially improve the quality of the manuscript.
We also thank the Associate Editor and the Editor for their constructive comments. \\

\noindent {\bf Funding.}
EBG and KL were supported in part by NSF grant CMMI-1552764. EBG, JF and KS were supported in part by NSF grants DMS-1349724 and DMS-2052465. KL and JF were supported in part by the U.S.-Norway Fulbright Foundation and the Research Council of Norway R\&D Grant 309273. EBG was supported in part by the Norwegian Centennial Chair grant.

\bibliography{tubes}
\bibliographystyle{amsplain}

\end{document}